\documentclass{article}
\usepackage{arxiv}
\usepackage[utf8]{inputenc} 
\usepackage[T1]{fontenc}    
\usepackage{hyperref}       
\usepackage{url}            
\usepackage{booktabs}       
\usepackage{amsfonts}       
\usepackage{nicefrac}       
\usepackage{microtype}      
\usepackage{graphicx}
\usepackage[backend=biber,style=alphabetic,sorting=ynt]{biblatex}
\usepackage{csquotes}
\usepackage{doi}
\usepackage{tabularx}
\usepackage{lipsum}
\usepackage{wrapfig}
\usepackage{listings}
\usepackage{mathrsfs}
\usepackage{amsmath}
\usepackage{amssymb}
\usepackage{mathabx}
\usepackage{bbm}
\usepackage{bm}
\usepackage{esvect}
\usepackage{upgreek}
\usepackage{cases}
\usepackage{nicefrac}
\usepackage{diagbox}
\usepackage{multirow}

\usepackage{amsthm}
\theoremstyle{definition}
\newtheorem{definition}{Definition}[section]

\newtheorem{remark}{Remark}[section]

\newtheorem{conjecture}{Conjecture}[section]
\newtheorem{theorem}{Theorem}[section]

\usepackage{algorithm} 
\usepackage{algpseudocode}

\usepackage{hyperref}

\title{Solving the non-preemptive two queue polling model with generally distributed service and switch-over durations and Poisson arrivals as a Semi-Markov Decision Process}

\author{Dylan Solms\\
	Department of Decision Sciences\\
	University of South Africa\\
	South Africa \\
	\texttt{62652257@mylife.unisa.ac.za} \\
}

\renewcommand{\headeright}{Discussion}
\renewcommand{\undertitle}{Discussion}
\renewcommand{\shorttitle}{Queuing Theory}

\hypersetup{
pdftitle={},
pdfsubject={Queuing Theory},
pdfauthor={Dylan Solms},
pdfkeywords={Queuing Theory, Limits, Steady-state,Ergodic,Little's Law,Modelling},
}

\begin{document}
\maketitle

\begin{abstract}
The polling system with switch-over durations is a useful model with several practical applications. It is classified as a \emph{Discrete Event Dynamic System} (DEDS) for which no one agreed upon modelling approach exists. Furthermore, DEDS are quite complex. To date, the most sophisticated approach to modelling the polling system of interest has been a \emph{Continuous-time Markov Decision Process} (CTMDP). This paper presents a \emph{Semi-Markov Decision Process} (SMDP) formulation of the polling system as to introduce additional modelling power. Such power comes at the expense of truncation errors and expensive numerical integrals which naturally leads to the question of whether the SMDP policy provides a worthwhile advantage. To further add to this scenario, it is shown how sparsity can be exploited in the CTMDP to develop a computationally efficient model. The discounted performance of the SMDP and CTMDP policies are evaluated using a Semi-Markov Process simulator. The two policies are accompanied by a heuristic policy specifically developed for this polling system a well as an exhaustive service policy. Parametric and non-parametric hypothesis tests are used to test whether differences in performance are statistically significant.
\end{abstract}

\tableofcontents

\newpage

\section{Introduction}

The field of \emph{Operations Research} (OR) is often interested in the \emph{performance evaluation} or \emph{control} of industrial systems such as re-entrant manufacturing lines \cite{chen_thesis}, signalised traffic intersections \cite{fleck2015adaptive} or communication networks \cite{GSMDP_admission_control}. These dynamic systems consist of discrete entities that produce abrupt changes in the state of the overall system due to the completion of a timed event i.e. a sensor registering an arrival of a vehicle to a traffic intersection. These are referred to as \emph{Discrete Event Dynamic Systems} (DEDS) and are observed to evolve temporally in a piecewise continuous fashion which is in contrast to \emph{Continuous Variable Dynamic Systems} (CVDS) (see figure~\ref{fig:CVDS_vs_DEDS}).

While physics and engineering have garnered much success studying CVDS under the framework of difference and differential equations, the same cannot be said for the use of DEDS in OR. The state of DEDS in the late eighties is summed up by \cite{fleming1988report}:
\begin{displayquote}
Discrete-event dynamical systems exist in many technological applications, but there are no models of discrete-event systems that are mathematically as concise or computationally as feasible as are differential equations for CVDS. There is no agreement as to which is the best model, particularly for the purpose of control.
\end{displayquote}
\begin{figure}[!htbp]
    \begin{center}
        \includegraphics[width=0.4\textwidth]{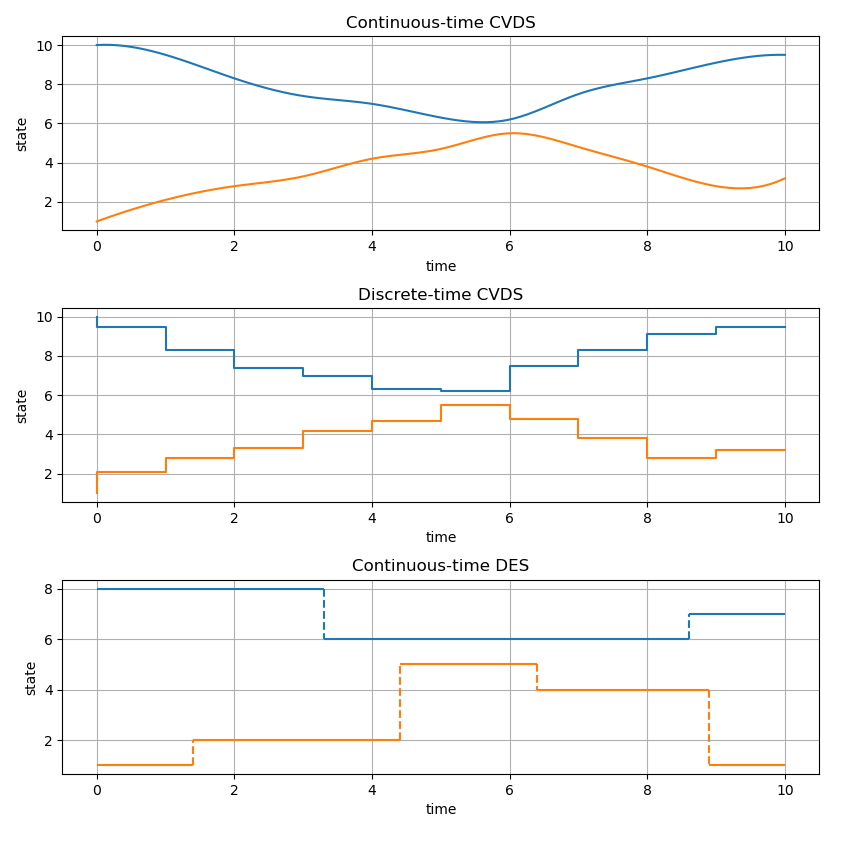}
    \end{center}
    \caption{Sample paths of continuous-time CVDS (top), discrete-time CVDS (middle) and continuous-time DEDS (bottom).}
    \label{fig:CVDS_vs_DEDS}
\end{figure}
This continues to remain true. Table \ref{tab:Different Models of DEDS} presents the diversity present among of the main models used in DEDS\footnote{A survey of these DEDS models applied to a two station production line can be found in \cite{cao1990models}}.

\begin{table}[!htbp]
    \centering
    \begin{tabularx}{0.8\textwidth} { 
    | >{\centering\arraybackslash}X 
    | >{\centering\arraybackslash}X 
    | >{\centering\arraybackslash}X | }
    \hline 
         & \textbf{Timed} & \textbf{Untimed}  \\
         \hline
        \textbf{Logical} &
            Temporal Logic, Timed Petri Net & Finite State Machine, Petri Net\\
        \hline
        \textbf{Algebraic} & Min-Max Algebra & Finitely Recursive Processes\\
        \hline
        \textbf{Performance} & Markov Chains, Generalised Semi-Markov Process, Queue Networks, Stochastic Petri Nets & \\
        \hline
    \end{tabularx}
    \caption{Different Models of DEDS \cite{hoPAbook}}
    \label{tab:Different Models of DEDS}
\end{table}
DEDS are difficult to evaluate or control because they have multiple event processes evolving in an \emph{asynchronous} and \emph{concurrent} manner. The former refers to the fact that events triggered by these processes do not occur at the same time while the latter describes the processes to evolve simultaneously (each at their own rate). The concurrent and asynchronous properties naturally embed memory into the DEDS model. This is because once an event triggers the remaining processes do not forget their ages. Naturally, the first-order \emph{Markov property} \cite{stewart2009probability} does not hold.

A consequence of the absence of the Markov property is that powerful frameworks such as that of \emph{Markov Decision Processes} (MDPs) \cite{BertsekasVol2} cannot be used in the optimal control of DEDS. Policies obtained through solving MDPs are convenient as they specify an action for each state of the system. This allows actions to be taken with respect to feedback from the system which is particularly useful if dynamics are stochastic. Some alternative means of DEDS control have used parameterised open-loop policies such as the duration of green-lights at a signalised traffic intersections \cite{howell_2003,howell_2005,geng2012traffic,fleck2015adaptive} as these do not rely on the Markov property. The parameters of these policies are optimised for using \emph{off-line} or \emph{on-line} simulations-based optimisation techniques. Off-line methods include Kiefer-Wolfowitz finite differences or simultaneous perturbation stochastic approximation \cite{spall1998_SPSA} while perturbation analysis allows for on-line optimisation \cite{howell_2003,howell_2005,geng2012traffic,fleck2015adaptive}. Other approaches include identifying an optimal or ideal trajectory for the system to follow as in chapter 5 of \cite{van_eekelen}. Two controllers or polices are required: one for maintaining the optimal trajectory and another corrective policy for steering the system towards the optimal trajectory. Such an approach may spend most of its time implementing the corrective policy which itself may not be optimal in terms of performing correction in minimum time. This occurs of the corrective policy is a stabilising policy as in \cite{van_eekelen}. Corrective policies can, however, be optimised for at an additional computational cost \cite{van_zwieten_two_queue}. Markov policies can be thought of has having both optimal corrective and maintenance elements.

The Markov property can be regained by making certain modelling assumptions such as assuming all event durations to be \emph{memoryless}. This can only be achieved through using the exponential distribution in continuous-time or the geometric distribution in discrete time \cite{stewart2009probability}. Focusing in continuous-tome applications, the DEDS is abstracted as a \emph{Continuous-time Markov Chain} (CTMC) which can be converted to a \emph{Discrete-time Markov Chain} (DTMC) via \emph{uniformisation} (see chapter 1.4 of \cite{BertsekasVol2,suk1991optimal}) and solved for as a MDP. This approach forgoes modelling the asynchronous and concurrent dynamics of DEDS in return for the Markov property. 

The method of phases has been used in attempting to retain or approximate asynchronous and concurrent features of the original DEDS when converting it to a CTMC, DTMC and MDP \cite{younes_thesis,younes_GSMP_formalism,younes_GSMPD_solving1,younes_GSMPD_solving2}. This results in fictitious states in the policy. Such states cannot directly be observed which makes policy execution an issue related to \emph{Partially Observable Markov Decision Processes} (POMDPs) \cite{POMDP_thesis}. Policy execution has usually been implemented using the $Q_{MDP}$ heuristic as in \cite{younes_GSMPD_solving2} which makes the assumption that all uncertainty vanishes after an action is performed. Failure to fulfill this assumption may lead to sub-optimal outcomes as found in the empirical experiments of chapter 6 in \cite{POMDP_thesis}. Further issues pertaining to the method of phases approach includes additional state-space complexity, model-bias and model uncertainty which are discussed in \cite{chen_GSMDP_paper} and chapter 2.4.2 of \cite{chen_thesis}. 

This paper presents a method of obtaining a MDP policy without assuming all processes to be memoryless as to retain some of the asynchronous and concurrent features of DEDS without introducing any additional complications such as with the method of phases. This is done through identifying event processes that occur in a serial manner. These serial processes have their event time distributions unmodified while all other process are approximated using exponential distributions. The outcome is a \emph{Semi-Markov Process} (SMP) where decisions can be made at points in its \emph{embedded chain} where the Markov property holds such that a MDP can be solved for. This approach is limited by constraining all decisions to be \emph{non-preemptive} (cannot be interrupted) and relies on serial processes to be manually identified. Hence, it is not a complete substitute for the other two approaches. 

This \emph{Semi-Markov Decision Process} (SMDP) approach is illustrated on a non-preemptive two-queue polling model with generally distributed service and switch-over durations and Poisson arrivals.

\section{Background}

\subsection{Polling models}\label{section: polling models}

A polling system\index{polling system} is model in which a \emph{common} resource is sequentially allocated among different classes of waiting entities (scheduling definition). This is synonymous to a \emph{single} server attending multiple parallel queues in some order (queuing theory definition). It is stressed that resources are traditionally \emph{allocated and not distributed}. The latter would refer to the resource being divided and dispensed among a subset of classes.

\subsubsection{Control framework}\label{section: control framework}

Polling models come equipped with a well established \emph{control framework}\footnote{A convenient recipe to follow.} which dictates that decisions be made with regards to the following \cite{terekhov2014queueing}.
\begin{enumerate}
    \item \textbf{Polling order:} the sequence of visits to each of the queues. This may be \emph{fixed} which is the case if a \emph{cyclic} policy is used. The non-weighted Round-Robin scheduling scheme makes use of such a cyclic order. If the order is \emph{dynamic} then it should be in response to the state of the system or due to the scheduling algorithm being of \emph{online} discipline. The weighted Round-Robin scheduling scheme assigns a dynamic polling order.
    \item \textbf{Service policy per queue:} when the server is at a queue, a decision has to be made on how \emph{many} customers to serve or for what \emph{duration}. Time-based service does not require any sensors and only some clock. The duration of service is often fixed. Count-based policies commonly follow the \emph{exhaustive} \cite{liu1992optimal} or \emph{gated} \cite{boxma1991_polling_table} discipline.
    \begin{enumerate}
        \item \emph{Exhaustive\index{Exhaustive} policies} serve the queue until depletion. More often than not, this is optimal policy with respect to optimising a metric associated with system efficiency.
        \item \emph{Gated\index{Gated} policies} serve all customers in the queue that are present at the time of server arrival. Further arriving customers are "gated" off and will receive service at the next server visit. Gated policies are usually selected if fairness is a concern \cite{terekhov2014queueing}.
        \item \emph{Threshold/index} policies \cite{duenyas_van_oyen_2000_finite_buffer_heuristic,duenyas1996heuristic,van_oyen_setup_costs_with_arrivals} serve the queue until its length is equal to or below some integer. Exhaustive policies are a special case where the index is set to zero.
    \end{enumerate}
      Further decisions that need to made pertain to \emph{idling\index{idling}} \cite{duenyas1996heuristic,duenyas_van_oyen_2000_finite_buffer_heuristic,liu1992optimal} - the act of the server not performing any service. 
      \begin{enumerate}
          \item A \emph{non-idling} server never idles at an empty queue while an \emph{idling} server does.
          \item A \emph{greedy\index{greedy} sever} never idles at a non-empty queue.
          \item An idling server is \emph{patient} as it remains at an empty queue. An \emph{impatient} server is non-idling and leaves an empty queue.
      \end{enumerate}
      If the polling system has non-empty queues and the server idles at an empty queue then the non-empty queues are \emph{deadlocked} and the system is in \emph{starvation}.
    \item \textbf{Service policy within each queue:} if a queue happens to have different classes of customers then a \emph{First-Come-First-Serve\index{First-Come-First-Serve}} (FCFS) discipline may be replaced by some other discipline such as Last-Come-First-Serve or Shorted-Processing-Time \cite{harchol2013performance}. A tailored sequence may even be prescribed.
    \item \textbf{Service policy per customer:} a decision has to be made on whether service can be interrupted or not.
    \begin{enumerate}
        \item \emph{Preemptive\index{Preemptive}}: service can be interrupted and allocated to another customer from a different class. If the server returns to this customer and does not resume service then it is \emph{non-work-conserving}\footnote{Note that idling does not play a role in work conservation: a server may not be performing any work during this time but no worked completed is lost or destroyed.}.
        \item \emph{Non-preemptive\index{Non-preemptive}}: once service is allocated to a customer, the server commits until completion. Naturally, this policy is always work-conserving.
    \end{enumerate}
\end{enumerate}

\subsubsection{Notation}\label{section: notation}

This paper will focus on a non-preemptive two queue polling model with generally distributed service and switch-over durations and Poisson arrivals as shown in figure~\ref{fig:polling model with switching diagram}. More specifically, $\lambda_i \in \mathbb{R}_{>0}$ denotes the arrival rate of queue $i$ customers such that an arrival duration $t_{\lambda_i}\sim Exp(\lambda_i)$ is memoryless $\mathscr{P}\left(t\leq \Delta_1 + \Delta_2 \mid t> \Delta_1\right) = \mathscr{P}\left(t\leq\Delta_2\right) =1 - \exp \left\{- \lambda_i \Delta_2 \right\}$. Service processes $\mu_i$ have arbitrary cumulative distribution functions $F_{\mu_i}:\mathbb{R}_{\geq 0} \to \mathbb{R}\cap[0,1]$ such that $t_{\mu_i}\sim F_{\mu_i}$ with mean duration $\bar{t}_{\mu_i} = \mu_i^{-1}$. Switch-over processes $s_{i,j}$ also have arbitrary cumulative distribution functions $F_{s_{i,j}}:\mathbb{R}_{\geq 0} \to \mathbb{R}\cap[0,1]$ such that $t_{s_{i,j}}\sim F_{s_{i,j}}$ with mean duration $\bar{s}_{i,j} = s_{i,j}^{-1}$. Queue lengths are denoted by $n_i \in \mathbb{N}_{0}$.

\begin{figure}[!htbp]
    \centering
    \includegraphics[width=0.7\textwidth]{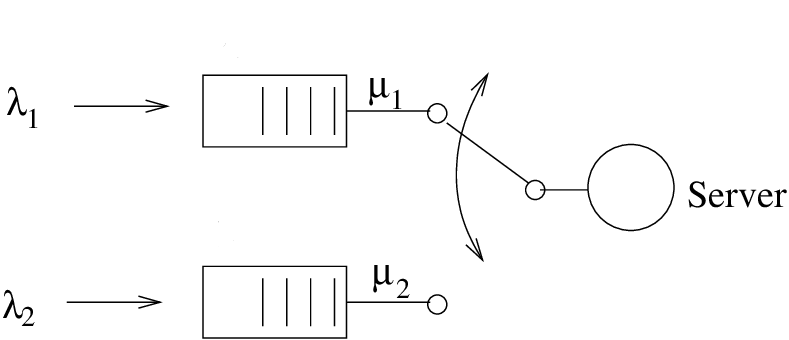}
    \caption{Two-queue polling model with switch-over durations\index{switch-over durations} where $\mu_i$ denotes a service process and $\lambda_i$ an arrival process.}
    \label{fig:polling model with switching diagram}
\end{figure}

The state-space is represented by the tuple $x = \left(n_1,n_2,l_1\right)$ where $l_1 \in \{1,2\}$ is the location of the server such that $x \in \mathcal{X} \subseteq
\mathbb{N}_0 \times \mathbb{N}_0 \times \{1,2\}$. Arrival processes are always ongoing while at most one of $\mu_1,\mu_2,s_{1,2}$ or $s_{2,1}$ can be active at a time and is selected from $\mathcal{E} = \{\mu_1,\mu_2,s_{1,2},s_{2,1},\Lambda_1,\Lambda_2  \}$. Here $\Lambda_i$ denotes the act of a server idling at queue $i$. A server idles until either an arrival to queue 1 or queue 2 occurs. During idling, the system is a CTMC as opposed to a SMDP. Selected actions change the state of the server $l_2 \in \{0,1,2\}$ which translates to idling/free, serving and switching. These actions cannot be interrupted, As such, decision on what to make the server do can only be selected at the instant a server becomes free $l_2=0$. The expanded state-space include the state of the server $\mathcal{Y} = \mathcal{X}\times \{0,1,2\}$. 

Furthermore, not all actions can always be selected. The set of feasible actions depends on the queue lengths and position of the server such that $l_2 \in \mathscr{A}\left(n_1,n_2,l_1\right)\subset \mathscr{A} =\{0,1,2\} $. For example $\mathscr{A}(0,2,1) = \{0,2\}$ such that $\mu_1$ is infeasible. Note that idling or switching is always feasible. Due to the non-preemptive service discipline, it is understood that $\mathscr{A}(x)$ is only consulted when $l_2=0$. This is analogous to defining $\underline{\mathscr{A}}(n_1,n_2,l_1,l_2)$ over the extended state-space such that it returns the empty set $\varnothing$ when $l_2 \neq 0$ and $\mathscr{A}(n_1,n_2,l_2)$ otherwise.

Each queue $i$ customer incurs a holding cost during its time spent in the queue at a finite positive rate $c_i$. Finite non-negative switching costs $K_{i,j}$ can be included in the switching decision. These can represent an additional penalty other than the temporal one such as risk incurred or resources required for the setup of queue $j$.

Lastly, stability of a polling system is important. A necessary condition for stability is that $\rho = \rho_1 + \rho_2 < 1$ \cite{duenyas1996heuristic} where $\rho_i = \lambda_i/\mu_i$. Here, $\rho$ is the \emph{utilisation} of the server and represents the long-run average proportion of time that the server must perform work (be busy) for stability. Ideally, the utilisation realised under the policy $\phi$ is less than this such that $\rho_{\phi}\leq \rho < 1$.

\subsubsection{Literature review}\label{section: literature review}

The study of the two queue polling model can be divided into three main general categories: without switch-over penalties or durations, with switching costs (e.g. $K_{i,j}$) and with switch-over durations (e.g. $t_{s_{i,j}}$).

\paragraph{Without switch-over penalties and durations} \hfill \break
The optimal policy for a polling model with $n$ queues in the case of no switching is given by the $\mu c$-rule. This is a static rule that minimises the \emph{linear} holding cost for any arbitrary random arrival process \cite{baras1985_mu_c_rule,baras1985two,liu1995_sample_path_methods}, deterministic arrivals \cite{smith1956deterministic_arrival} or no arrivals \cite{BertsekasVol2} given that the service distribution is \emph{geometric} \cite{baras1985_mu_c_rule,baras1985two} in the discrete time setting or \emph{exponential} \cite{BertsekasVol2,cassandras_book,liu1995_sample_path_methods} in the continuous-time case. It has also been shown to be optimal under both the non-preemptive\index{non-preemptive} \cite{fernandesscheduling} and preemptive\index{preemptive} server disciplines \cite{baras1985two}. It does not hold under finite buffer queues \cite{kim_van_oyen_loss_queue,suk1991optimal} or non-linear holding costs \cite{mu_c_rule_convex_costs}.

The $\mu c$-rule is a Gittin's index\index{Gittin's index} policy \cite{gittins1979bandit,MAB_gittins_tutorial} where the characteristics that define the index for each bandit/queue does not change upon being played as long as the queue has at least one customer. Each queue has an index $\nu_i:\mathbb{N}_{0}\times \mathbb{R}_{> 0} \times \mathbb{R}_{> 0} \to \mathbb{R}_{\geq 0}$ defined as
\begin{equation}
    \nu_i(x,\mu,c) = \begin{cases}
    \mu c, & x > 0\\
    0 & x = 0
    \end{cases}
\end{equation}
such that $\nu_i(\cdot)$ is the reward rate associated with serving queue $i$. It follows that the queue with the highest index is served \emph{exhaustively}. Once exhausted, its index should drop to the lowest possible value of zero upon which the new highest index is served. At the instant of an arrival that re-activates a queue with a higher index than the one currently in service, a preemptive server will immediately abandon its current job while a non-preemptive server will complete the current job and then switch without delay. As such, the $\mu c$ rule is observed to switch server positions frequently \cite{tava_practical}, it may not be considered an overall fair policy if it always switches to the favoured queue. 

\paragraph{With switching costs} \hfill \break

Switching costs are elements of a model. As such, they may approximate reality or represent it exactly. An example of setup costs being used as \emph{approximations} can be found in \cite{van_oyen_setup_costs_with_arrivals} where setup costs are explained to play a role in \emph{Just-in-time} (JIT) production. These are scenarios where preparation/setup work is done while the sever is serving queue $i$ in order to make the switching duration from queue $i$ to $j$ almost negligible. Setup costs that \emph{exactly} represent reality can occur in a scenario where diggings may occur on a site to search for minerals as in example 1.3.1 of \cite{BertsekasVol2}. The equipment can only be on one site at a time. To move from one site to another involves a large transportation cost of equipment. 

In comparison to the $\mu c$-rule, these policies should aim to minimise the number of times the server switches. This is not the case with the $\mu c$-rule which can be found switching almost infinitely often \cite{tava_practical}. Furthermore, unlike the $\mu c$-rule no closed-form scheduling rule has been found.

Numerical methods such as Value Iteration have been used to obtain policies for when the polling system can be modelled as a MDP. This requires arrivals and services durations to follow the exponential distribution. Such a approach has been used in \cite{koole1997assigning,moustafa1996optimal}. Both papers obtain switching curves from their numerical policies. The existence of such curves are proved under linear holding costs using monotonicity properties identified in the uniformised Bellman equations such as submodularity. Additionally, \cite{koole1997assigning} identifies that the switching curve\index{switching curve} becomes constant once it reaches a threshold queue length with respect to the priority queue (the one with the larger $\mu c$ value). Such a result is used in obtaining an computationally inexpensive heuristic policy.

A departure is taken from Dynamic Programming and the Bellman equations in \cite{van_oyen_1992_no_arrivals}. Sample-path arguments are used to establish optimality of exhaustive service and develop heuristic Gittin's indices to choose which queue to switch to one the current one has been depleted. The discounted Gittin's index\index{Gittin's index} assumes the system to be ergodic and uses renewal-reward theory in its derivation such that
\begin{eqnarray}
    \nu_{i}^{\beta}(x_i,\mu_i,c_i) & = & \frac{\mbox{expected discounted reward earned exhausting queue }i}{\mbox{expected discounted time to exhaust queue }i}. \label{eq:duenyas}
\end{eqnarray}
where no assumption is made on the service distribution of the non-preemptive server. It is, however, presumed that all switching costs are the same and that no arrivals occur. The approach is extended to a system with Poisson distributed arrivals in \cite{van_oyen_setup_costs_with_arrivals}.

\paragraph{With switch-over durations} \hfill \break
In the seminal paper on the study of polling models with switching/setup times \cite{hofri1987optimal}, a two queue model is considered. Exhaustive service is shown to be optimal in both queues using sample-path arguments. Moreover, idling is proven optimal such that the server switches only when the current queue is empty and the other has exceeded a threshold. Both queues have such a threshold, hence the term \emph{double-threshold} policy.

The first paper to completely characterise the optimal policy of a polling model with switch-over durations is the work of \cite{liu1992optimal}. It does so for a completely symmetric system of $N$ queues. Each queue has independent and generally distributed arrival, services and switch-over processes. Such a system is not a Semi-Markov\index{Semi-Markov} or Continuous-time Markov Decision Process. Subsequently, sample-path methods were used instead of Dynamic Programming. The following characterised the policy that minimises the total unfinished work in the system: a server will never idle at a non-empty queue, exhaustive service is optimal, the server switches from an empty queue to the \emph{stochastically largest queue} and the server is patient. The last point pertains to the server remaining at its current queue in an empty system.

The heuristics developed for the polling model with switching costs \cite{van_oyen_setup_costs_with_arrivals} that makes use of (\ref{eq:duenyas}) is extended to switching durations for infinite buffer queues \cite{duenyas1996heuristic} and finite buffer queues \cite{duenyas_van_oyen_2000_finite_buffer_heuristic}.

In \cite{browne1989dynamic}, emphasis is placed on the order of queues to be served. The number of customers to be served non-preemptively per queue visit is determined by gated or exhaustive policies. The order of queues visits is determined online through Dynamic Programming with a rolling-horizon\index{rolling-horizon} under the restriction that a queue can be visited only once in this window. This problem is also addressed in \cite{tava_practical}. However, it goes about by developing a heuristic that balances minimising the mean of the waiting time distribution while aiming for light tails. This paper also performed a total of seven simulation experiments in order to compare the performance of their heuristic policy to others. Their policy was outperformed by \cite{duenyas1996heuristic} in terms of the mean waiting time but scored the best in minimising the variance of waiting time.

Heavy traffic approximations have been used to study the two queue polling model with general arrival, service and switch-over durations \cite{reiman1998dynamic}. Exhaustive and non-exhaustive service in a two queue polling model with Markov-modulated Poisson process arrivals were investigated in \cite{narahari1997optimality}. Modified Policy Iteration algorithms can be found in \cite{kim_and_van_oyen_MDP_polling} for solving the exact scheduling policy when the polling model is a CTMC.

In chapter 5 of \cite{van_eekelen}, the two-queue polling model with switch-over durations has been studied as a system of two deterministic fluid queues. In minimising the average work-in-progress for an asymmetric system, an optimal limit-cycle is characterised to take in one of two forms as in figure~\ref{fig:bow tie curves}. The pure bow-tie curve translates to exhaustive service and immediate switching upon a queue being emptied. The truncated bow-tie curve allows for the server to idle at a priority queue $i$ (e.g. $c_i \mu_i > c_j\mu_j$) if $c_i\lambda_i(\rho_i + \rho_j) - (c_i\lambda_i - c_j\lambda_j)(1-\rho_j) < 0 $. The system does not always follow the optimal limit cycle. A separate corrective policy/controller is required to steer it towards the optimal limit cycle. In \cite{van_eekelen} a policy is used that it not time-optimal whereas \cite{van_zwieten_two_queue} computes an optimal controller by solving a quadratic program.

\begin{figure}[!htbp]
    \centering
    \includegraphics[width=0.4\textwidth]{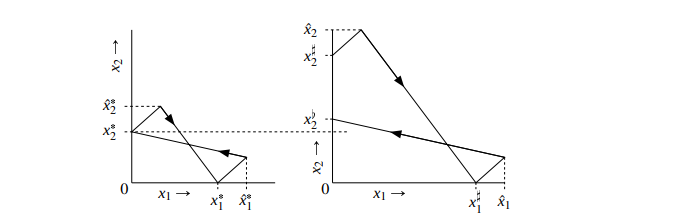}
    \caption{LEFT: pure bow-tie curve RIGHT: truncated bow-tie curve}
    \label{fig:bow tie curves}
\end{figure}

\subsection{Markov Chains}\label{section: markov chains}

A Markov process is a stochastic system where the future is conditionally independent of the past given the current state (chapter 6 of \cite{cassandras_book}). Such processes exhibit the \emph{Markov property} as defined below.

\begin{definition}[\textbf{Markov property} \cite{cassandras_book}]\label{definition:markov property}
Often referred to as a memoryless property, the following two criteria needs to be fulfilled:
\begin{enumerate}
    \item[$\mathbf{P_1}$:] The history of past states is not relevant. In other words, no \emph{state memory} is required.
    \item[$\mathbf{P_2}$:] The duration of time spent waiting in the current state is irrelevant such that no \emph{duration memory} is required.
\end{enumerate}
\end{definition}
If $\mathcal{H}(x_k) = \left\{ (x_0,\Delta t_0),(x_1,\Delta t_1),\cdots,(x_k) \right\}$ denotes the history of states and \emph{waiting durations} up to the instant of entering the current state $x_k$ and $\tau_k$ the amount of time already spent in $x_k$ then the Markov property states that the future state $x_{k+1}$ can be forecast through $\mathbf{P_1}$ such as
\begin{eqnarray}
    \mathscr{P}\left(x_{k+1}=j\mid \mathcal{H}(x_k=i)\right) & = & \mathscr{P}\left(x_{k+1}=j \mid x_{k}=i \right)\label{eq: P1}\\
    & = & P_{i,j}^k \nonumber 
\end{eqnarray}
while $\mathbf{P}_2$ predicts future state $x_{k+1}$ to be reached within $\Delta \tau_k$ units of time
\begin{eqnarray}
    \mathscr{P}\left(x_{k+1}=j,t\leq\tau_k + \Delta \tau_k \mid \mathcal{H}(x_k=i),t>\tau_k\right) & = & \mathscr{P}\left(x_{k+1}=j,t\leq\Delta \tau_k \mid x_{k}=i \right)\label{eq: P2}\\
    & = & P_{i,j}^k\left(\Delta\tau_k\right) \nonumber .
\end{eqnarray}

In continuous time, $\mathbf{P_2}$ implies $\Delta t_k = (\tau_k+\Delta \tau_k) = t_{k+1}-t_k \sim Exp(\lambda)$ where $t_k$ is the time-stamp of the start of the $k^{th}$ waiting interval as opposed to its duration $\Delta t_k$. In certain modelling applications, $\mathbf{P_2}$ is too restrictive and modelling power can be gained from relaxing it. This results in the \emph{Semi-Markov property}.
\begin{definition}[\textbf{Semi-Markov property} \cite{cassandras_book} ]
By relaxing $\mathbf{P_2}$ such that
\begin{eqnarray}
    \mathscr{P}\left(x_{k+1}=j,t\leq\tau_k + \Delta \tau_k \mid \mathcal{H}(x_k=i),t>\tau_k\right) & = & \mathscr{P}\left(x_{k+1}=j,t\leq\Delta \tau_k+\tau_k \mid x_{k}=i,t>\tau_k \right)\label{eq: semi-markov property}\\
    & = & P_{i,j}^k\left(\Delta\tau_k \mid \tau_k\right) \nonumber
\end{eqnarray}
then $\Delta t_k \sim F(\cdot)$ where $F:\mathbb{R}_{\geq 0} \to \mathbb{R}\cap[0,1]$ is an arbitrary cumulative distribution function.
\end{definition}

In the sequel all Markov processes will be assumed to have time-homogeneous transition probabilities. That is $
\forall k_1,k_2 \in \mathbb{N}_{0}$ such that $k_1 \neq k_2$ then $P_{i,j}^{k_1}(\cdot) = P_{i,j}^{k_2}(\cdot)= P_{i,j}(\cdot)$ which holds for equations (\ref{eq: P1})-(\ref{eq: semi-markov property}). Lastly, time-homogeneous \emph{Markov chains} of finite size $N\in \mathbb{N}$ will be presented in matrix form $\mathbf{P}(\cdot) = [P_{i,j}(\cdot)], \quad i,j = 1,2,3,\cdots,N$. These are \emph{stochastic matrices}.
\begin{definition}[\textbf{Stochastic matrix} \cite{stewart2009probability}]
A stochastic matrix $\mathbf{P} = [P_{i,j}]$ has the following properties:
\begin{enumerate}
    \item $P_{i,j}\geq 0, \quad \forall i,j = 1,2,3,\cdots,N$
    \item $\sum_{j=1}^N P_{i,j} = 1, \quad \forall i = 1,2,3,\cdots,N$
    \item Each column of has at least one non-zero entry.
\end{enumerate}
The first two properties require each row to lie \emph{within} the $(N-1)$-dimensional unit simplex $\mathbf{\Delta}^{N-1}$ or \emph{on} $\mathbf{\Delta}^{N}$. The last property eliminates the existence of \emph{ephemeral} states.
\end{definition}

\subsubsection{Discrete-time Markov chain}\label{section:DTMC}
A discrete time Markov chain is sampled at fixed intervals of time. These intervals are usually chosen to be of unit length $\Delta t_k = 1$ such that time is not of importance in analysis. As the duration of time already spent in a state is not of relevance $\mathbf{P_2}$ holds
such that transition probabilities can be computed using (\ref{eq: P1}). 

Initial behaviour of a Markov chain is referred to as \emph{transience}. The transient transition probabilities can be obtained as follows
\begin{eqnarray}
    \mathscr{P}(x_k = j\mid x_0=i) & = & \left(\mathbf{P}^k \right)_{i,j}\\
    & = & \left( \mathbf{P}\mathbf{P}^{k-1} \right)_{i,j}\label{eq:kolomogorov}\\
    & = & \left( \mathbf{P}^{k-1} \mathbf{P} \right)_{i,j}
\end{eqnarray}
where the last two equations are a result of the \emph{Chapman-Kolmogorov equations} (see chapter 9.3 of \cite{stewart2009probability}). This allows for the transient state probabilities to be determined $\phi_j(k) = \mathscr{P}(x_k=j)$ using an initial distribution $\vec{\phi}(0) = [\phi_j(0)]$ such that $\vec{\phi}(k) = \vec{\phi}(0) \mathbf{P}^k$.

Long-term behaviour is analysed using the \emph{limiting} and \emph{stationary} distributions.
\begin{definition}[\textbf{Limiting distribution} \cite{harchol2013performance}]
The \emph{limiting probability} of being in state $j$ is given as
\begin{eqnarray}
    \phi_j & = &  \lim_{n\to\infty}\left\{ \left( \mathbf{P}^n \right)_{i,j} \right\}
\end{eqnarray}
where $\vec{\phi} = [\phi_j]_{j=1}^N \in \mathbf{\Delta}^{N-1}$ and $n\in\mathbb{N}$. As the action of raising $\mathbb{P}$ to the $n^{th}$ power results in averaging over all possible $n$-length trajectories, $\phi_j$ can be considered an \emph{ensemble}-average as opposed to a \emph{time}-average.
\end{definition}

\begin{definition}[\textbf{Stationary distribution} \cite{harchol2013performance}]\label{def: stationary dist}
Solving for $\vec{\phi}$ in the following system of $N+1$ linear equations yields the stationary distribution.
\begin{eqnarray}
    \vec{\phi}  & = & \vec{\phi}\,\mathbf{P}\label{eq:stationary}\\
    \mbox{subject to: } \sum_{j=1}^{N}\phi_j & = &  1
\end{eqnarray}
Note that one of the $N$ equations in (\ref{eq:stationary}) is redundant and can be omitted which results in a well-posed system of $N$ equations with $N$ unknowns.
\end{definition}
The limiting distribution does not always exist as a valid probability distribution i.e. $\vec{\phi} \not \in \mathbf{\Delta}^{N-1}$. It exists if the Markov chain is \emph{ergodic}, in which case the limiting and stationary distribution are equal.
\begin{definition}[\textbf{Ergodic} \cite{harchol2013performance}]\label{def: ergodic}
A Markov chain is ergodic of it satisfies the following three properties:
\begin{enumerate}
    \item \textbf{Irreducible:} a system can reach any state $j \in \mathcal{X}$ from any other state $i \in \mathcal{X}$ such that the initial state chosen is not of long-term importance.
    \item \textbf{Positive recurrence:} ensures that each state is visited infinitely often in the limit such that the process stochastically restarts itself (renewal). Each renewed process counts as an independent sample hence the importance of this property in establishing equivalence between ensemble and time-averages.
    \item \textbf{Aperiodic:} the system is independent of time-steps. More specifically, the absence of a period prevents oscillation which in turn ensures that the ensemble average exists.
\end{enumerate}
\end{definition}
For an ergodic chain, the stationary distribution can be numerically computed using the limiting distribution. In other words, recursively computing the $n^{th}$ power of $\mathbf{P}$ using (\ref{eq:kolomogorov}) until all rows converge to the same vector $\vec{\phi}$ is a simple approach. Another options includes noting that $\vec{\phi}$ is the \emph{normalised eigenvector} of the \emph{left unit eigenvalue} (also the dominant eigenvalue) of $\mathbf{P}$. Performing LU-decomposition on $\mathbf{P}$ allows for two simpler system of linear equations to be solved using Gaussian elimination. More numerical methods can be found in chapter 10.2 and 10.3 of \cite{stewart2009probability}.

\subsubsection{Continuous-time Markov chain}\label{section:CTMC}

A CTMC is a Markov chain where the $k^{th}$ \emph{holding duration} in moving from state $i\to j$ is memoryless $\Delta t_{i,j}^k\sim Exp(\lambda_i)$ where $\lambda_{i,j} \in \vec{\theta}$ is a \emph{rate} parameter and $\vec{\theta}$ is the vector of model parameters. The \emph{waiting duration} is independent of $j$ and memoryless such that $\Delta t_i^k \sim Exp(\gamma_i)$ where $\gamma_i = \max_{j}\left\{ \lambda_{i,j} \right\} = \sum_{j} \lambda_{i,j}$ (chapter 11.4 of \cite{harchol2013performance}). As a result, a CTMC satisfies both $\mathbf{P_1}$ and $\mathbf{P_2}$ such that (\ref{eq: P2}) can be used in describing its dynamics. More specifically, $P_{i,j}(t) = \mathscr{P}(x_{k+1}=j\mid x_{j}=i) \times \mathscr{P}(t\leq\Delta t_i^k ) = q_{i,j}\times \left(1 - \gamma_i \exp\{-\gamma_i t \}\right)$ where $q_i = [q_{i,j}] \in \mathbf{\Delta}^{N-1}$ is the $i^{th}$ row vector from a DTMC that selects the next state $j$ upon entering the current state $i$.

In CTMCs every point in continuous-time exhibits the Markov property as opposed to a DTMC where the system is analysed at discrete epochs. This notion requires some additional notation. Instead if focusing on $x_k$, analysis uses $x(t)$ --- the state in continuous time. As such, $P_{i,j}(s,t) = \mathscr{P}(x(t+s)=j \mid x(s)=i), \forall s \geq 0$. Memorylessness implies $P_{i,j}(t-s) = \mathscr{P}(x(t-s)=j\mid x(0)=i) = P_{i,j}(s,t)$. The transient transition probabilities are obtained as follows
\begin{eqnarray}
    \mathbf{P}(s,t) & = &  e^{\mathbf{Q}(t-s)} \label{eq:ctmc 1}\\
    \mathbf{P}(\Delta t) & = & e^{\mathbf{Q}\Delta t} \label{eq:ctmc 2}
\end{eqnarray}
where $\mathbf{Q} = [\lambda_{i,j}] \in \mathbb{R}^{N\times N}$ is the \emph{generator matrix} such that $\lambda_{i,i} = - \sum_{i\neq j} \lambda_{i,j}$. Equation~(\ref{eq:ctmc 1}) is the solution to $d \mathbf{P}(s,t)/dt = \mathbf{P}(s,t)\mathbf{Q}$. This can be derived from the CTMC Chapman-Kolmogorov equations as in chapter 7.3.4 of \cite{cassandras_book} or chapter 9.10.2 of \cite{stewart2009probability}.
\begin{eqnarray}
    \mathscr{P}\left(x(t)=j\mid x(s)=i\right) & = & \mathbf{P}(s,u)\mathbf{P}(u,t)
\end{eqnarray}
where $0 \leq s \leq u \leq t$. Alternatively, chapter 12.2 of \cite{howard2012dynamic} derives (\ref{eq:ctmc 2}) using Laplace transforms. The transient state probabilities are similarly obtained through solving $d\vec{\phi}(t)/dt = \vec{\phi}(t)\mathbf{Q}$ (see chapter 9.10.4 of \cite{stewart2009probability}) such that 
\begin{eqnarray}
    \vec{\phi}(s,t) & = & \vec{\phi}(s)e^{\mathbf{Q}(t-s)} \label{eq: phi 1} \\
    \vec{\phi}(\Delta t) & = & \vec{\phi}(0)e^{\mathbf{Q}(\Delta t)} \label{eq: phi 2.}
\end{eqnarray}
The matrix exponential is a computationally intensive item to compute
\begin{equation}
    e^{\mathbf{Q}t} = \sum_{k=1}^{\infty} \frac{\left( \mathbf{Q}t \right)^k }{k!}
\end{equation}
and is subject to numerical instabilities for large $t$. For small generator matrices, the \emph{scaling method} (chapter 10.7.1 of \cite{stewart2009probability}) provides stable solutions at time $t=2^m t_0$ where $t_0 < t$ is selected such that $\mathbf{P}_{t_0} = \exp\{\mathbf{Q}t_0\}$ is stable. This method makes use of the fact that $\exp\{\mathbf{Q}t\} = \exp\{(\mathbf{Q}(t/2))^2\}$ such that setting $t_j = t_0$ allows for efficient recursive computation via the Chapman-Kolmogorov equations through squaring
\begin{eqnarray}
    \mathbf{P}_{t_{j+1}} & = & \mathbf{P}_{t_{j}}\mathbf{P}_{t_{j}}\\
    & = & \left( \mathbf{P}_{t_{j}}\right)^2
\end{eqnarray}
where $t_{j+1} = 2 t_j$ such that $\vec{\phi}(t) = \vec{\phi}(t_0)\mathbf{P}_{t_m}$. For larger matrices, the \emph{uniformisation} method is a better alternative for computing $\mathbf{P}(t)$ at a single $t$. Recalling that $\gamma_i = \sum_{j=\neq i}\lambda_{i,j}$ it should be clear that $\gamma_i = \lambda_{i,i} = - \mathbf{Q}_{i,i}$. A \emph{global sampling rate} can be defined $\gamma \geq \max_{\gamma_i}$ such that transitions occur over intervals of length $\Delta t = 1/\gamma$. A DTMC can be derived
\begin{eqnarray}
    \mathbf{P} & = &  \mathbf{Q}\Delta t + \mathbf{I}
\end{eqnarray}
where 
\begin{eqnarray}
    \mathbf{Q} & = & \gamma \left( \mathbf{P} - \mathbf{I}  \right)
\end{eqnarray}
such that $\vec{\phi}(t) = \vec{\phi}(0)\exp\{-\gamma t\}\exp\{\gamma \mathbf{P}t\}$. The uniformisation method recursively computes $\vec{\phi}(0)\exp\{\gamma \mathbf{P}t\}$ through setting $\vec{\nu}_1 = \Delta \vec{\nu}_{1} = \vec{\phi}(0)$ and updating it as
\begin{eqnarray}
    \Delta \vec{\nu}_{k+1} & = & \Delta \vec{\nu}_{k}\left(  \mathbf{P}\frac{\gamma t}{k}\right)\label{eq:step1 uni}\\
    \vec{\nu}_{k+1} & =&  \vec{\nu}_{k} + \Delta \vec{\nu}_{k+1} \label{eq:step2 uni}
\end{eqnarray}
until $|\vec{\nu}_{k^*}-\vec{\nu}_{k^*-1}| < \vec{\epsilon}$ where $\vec{\epsilon}$ is a vector of small error tolerances. The contents within brackets of equation~(\ref{eq:step1 uni}) needs to be computed only once and stored. Lastly, the desired result is obtained $\vec{\phi}(t) = \exp\{-\gamma t\}\vec{\nu}_{k^*}$. The work of this paper will be interested in routinely evaluating integrals of the form $\int_{0}^{\infty} f(t)\vec{\phi}(0)\exp(\mathbf{Q}t)\,dt = \int_{0}^{\infty} f(t)\vec{\phi}(t)\,dt$ using the same generator matrix. The previous two methods are not the most efficient for computing a sequence of state probabilities
\begin{equation}
    \Psi(\mathcal{T}) = \left\{ \left(\tau_k,\vec{\phi}(\tau_k)\right)\right\}_{\tau_k \in \mathcal{T}}
\end{equation}
over the mesh $\mathcal{T}$ used for numerical integration. It turns out that solving the \emph{ordinary differential equations} (ODEs) $d\vec{\phi}(t)/dt = \vec{\phi}(t)\mathbf{Q}$ is a more effective approach (see section 10.7.3 of \cite{stewart2009probability}). ODE solvers have well-developed software which extends the appeal of this approach. The most basic approach would be to use the \emph{explicit forward Euler method} over a mesh with constant step-size $\Delta \tau$ such that that
\begin{eqnarray}\label{eq: forward euler}
    \vec{\phi}(\tau_{k+1}) & = & \vec{\phi}(\tau_{k}) + \Delta \tau_k \vec{\phi}(\tau_{k}) \mathbf{Q}\\
    \tau_{k+1} & = & \tau_{k} + \Delta \tau_k
\end{eqnarray}
A function $g: \mathbb{R}_{\geq 0} \to \mathbf{\Delta}^{N-1}$ can be fitted using interpolation or regression on $\Phi(\mathcal{T})$. This stored function should be computationally cheap to compute as to optimise $\int_{0}^{\infty}f(t)g(t)\,dt \approx \int_{0}^{\infty} f(t)\vec{\phi}(t)\,dt$.

Solving for the stationary distribution follows from the theory of dynamic systems modelled as ODEs where $\vec{\phi}$ is stationary if it solves $d\vec{\phi}/dt = \vec{\phi}\, \mathbf{Q} = 0$ subject to the constraints $\phi_j \geq 0$ and $\sum_{j} \phi_j = 1, \quad \forall i = 1,2,\cdots,N$. As was the case with the DTMC stationary distribution, one of the equations $\phi_j = \sum_{i}\phi_{i} q_{i,j}$ can be replaced by $\sum_{j} \phi_j = 1$ in solving a linear system of $N$ equations with $N$ unknowns. Alternatively, the DTMC obtained through uniformisation~(\ref{eq:step1 uni}) can be solved for as in section~\ref{section:DTMC}.

\subsubsection{Semi-Markov Process}\label{section:SMP}
A SMP is Markov process where the $k^{th}$ holding duration in moving from state $i\to j$ follows an arbitrary distribution $F_{i,j}:\mathbb{R}_{\geq 0} \to \mathbb{R}\cap[0,1]$ parameterised by $\theta_{i,j} \in \vec{\theta}$. In a SMP the duration of time already spent in a state $\tau_k$ affects the future probabilities of transitions as well as the remaining time in the current state. As the holding time process is not memoryless, $\mathbf{P}_2$ does not hold and transitions are described using (\ref{eq: semi-markov property}). Specifically, the holding time component and state transition are independent such that
\begin{eqnarray}
    \mathscr{P}\left(x_{k+1}=j,t\leq\tau_k+\Delta \tau_k \mid x_k = i,t>\tau_k \right) & = & \mathscr{P}\left(t\leq\tau_k+\Delta \tau_k \mid t>\tau_k \right) \mathscr{P}\left(x_{k+1}=j \mid x_k = i \right)\\
    & = & F_{i,j}(\Delta \tau_k \mid \tau_k)\, q_{i,j} \label{eq:smp holding time} \\
    & = & P_{i,j}^k(\Delta \tau_k \mid \tau_k)\nonumber
\end{eqnarray}
where $k$ can be dropped as the process is assumed to be time-homogeneous. Note that the conditional cumulative distribution can be obtained from the original $F_{i,j}(\Delta \tau \mid \tau) = 1 - [1-F_{i,j}(\tau+\Delta\tau)]/[1-F_{i,j}(\tau)]$. Alternatively, the survival function $\bar{F}_{i,j}(\tau)  = 1 - F_{i,j}(\tau)$ could be used. It should be pointed out that (\ref{eq:smp holding time}) is not the only formulation (chapter 11.6 of \cite{howard2012dynamic}). The waiting time distribution $F_{i}(\cdot)$ could be considered
\begin{equation}
    F_{i}(\Delta \tau \mid \tau) = \sum_{i=1}^N q_{i,j} F_{i,j}(\Delta \tau \mid \tau) 
\end{equation}
such that 
\begin{eqnarray}
    P_{i,j}^k(\Delta\tau_k\mid\tau_k) & = & \varphi_{i,j}(\Delta\tau_k\mid\tau_k)F_{i}(\Delta\tau_k\mid\tau_k) \label{eq:smp waiting time}
\end{eqnarray}
where
\begin{equation}
    \varphi_{i,j}(\Delta\tau_k\mid\tau_k) = \frac{q_{i,j} F_{i,j}(\Delta \tau_k \mid \tau_k) }{F_{i}(\Delta\tau_k\mid\tau_k)}.
\end{equation}
Different variants of the SMP exist. A such it can be seen to generalise the DTMC and CTMC.

\begin{definition}[\textbf{Independent Semi Markov Process} \cite{howard2012dynamic}]
If the holding time does not depend on the destination state $x_k=j$ then $F_{i,j}(\cdot) = F_{i}(\cdot),\, \forall i = 1,2,\cdots,N$ and $q_{i,j}(\cdot) = p_{i,j}$ (page 716 of \cite{howard2012dynamic}) such that
\begin{eqnarray}
    p_{i,j}^k(\Delta \tau_k\mid \tau_k) & = & q_{i,j}F_{i}(\Delta \tau_k\mid \tau_k).
\end{eqnarray}
If $F_i(\cdot\mid \theta_i) = Exp(\cdot\mid\gamma_i)$ then the CTMC is a special case of the independent SMP. If the system is inspected with certainty at deterministic intervals of fixed length then $F_{i}(t|\theta_i) = 1, \, \forall t \in \mathbb{R}_{\geq 0}$. Note that $p_{i,i} \geq 0$ such that the state need not change at the next inspection which results in a self-transition. Hence, $F_{i}(t|\theta_i) = 1$ pertains to certainty in the discrete-time inspection and not necessarily some waiting time of the actual process. This condition results in a DTMC.
\end{definition}
Thus far, the interpretation of a SMP has been as follows: it behaves like a Markov chain at state transition instants where it selects the next state $j\sim p_{i}$ and observes a holding time $t_{i,j} \sim F_{i,j}$. This implies a single process $F_{i,j}$ to be active. In complex systems such as DEDS, this can be quite a limiting assumption as multiple process are active.
\begin{definition}[\textbf{Competing process model} \cite{howard2012dynamic}]\label{def: competing process model}
In a competing process model, multiple processes $\mathcal{F}_i = \{F_{i,j}:j=1,2,\cdots,M  \}$ are activated upon entering state $i$ such that multiple holding times are drawn $\mathcal{T}_i(\omega) = \{t_{i,j}\sim F(\cdot\mid \omega ):F\in\mathcal{F}_i \}$ where $\omega \in \Omega$ can be considered as underlying randomness. These are process lifetimes and begin to count down upon entry. A state transition is triggered once the $j^{th}$ has zero residual lifetime i.e. $j = \mbox{argmin}\left\{\mathcal{T}_i\right\}$. An important caveat which will be relaxed in section~\ref{section:GSMP} is that all remaining process be renewed or deactivated once the winning process terminates as to ensure that $\mathbf{P}_1$ holds (page 717 of \cite{howard2012dynamic}). A distinguishing feature of the competing process model is that its transition model is constructed using only the process distributions functions
\begin{eqnarray}
    \mathscr{P}\left(x_{k+1}=i,t\leq\Delta\tau_k+\tau_k\mid x_k=i,t>\tau_k,\mathcal{F}_i \right) & = & \int_{0}^{\Delta \tau_k} p_{i,j}\left(\xi \mid \tau_k,\mathcal{F}_i\right)\, d \xi \label{eq: competing cdf}\\
    & = & P_{i,j}\left(\Delta \tau_k \mid \tau_k,\mathcal{F}_i\right) 
\end{eqnarray}
with the probability density function is determined as 
\begin{equation}
    p_{i,j}\left(\Delta \tau_k \mid \tau_k,\mathcal{F}_i\right) = f_{i,j}(\Delta \tau_k\mid \tau_k) \times \prod_{F_{i,l}\in \mathcal{F}_i \setminus
    \left\{F_{i,j}\right\}} \left( 1- F_{i,l}(\Delta \tau_k\mid \tau_k) \right). \label{eq: competing pdf}
\end{equation}
\end{definition}
The usual descriptors can be recovered. As an example, if the system has just entered state $i$ and $\tau=0$ then the probability of selecting state $j$ is
\begin{equation}
    q_{i,j} = P_{i,j}(\infty\mid 0,\mathcal{F}_i).
\end{equation}

In a SMP, the Markov property holds at transition instants but not in-between. These instants are called \emph{renewal instants} and form what is called the \emph{embedded chain} of the SMP. The embedded chain can be analysed using DTMC methods of section~\ref{section:DTMC}.

\begin{definition}[\textbf{Embedded chain} \cite{stewart2009probability}]\label{def: embedded chain}
The embedded chain concerns only the state of the SMP $x_k \in \widetilde{\mathcal{X}}\subseteq \mathcal{X}$ at the time-stamps at which it enters that state $t_k=\sum_{m=0}^{k-1} \Delta t_m$. This is typically recorded as a sequence of tuples
\begin{eqnarray}
    \widetilde{X}_k & = & \{ \widetilde{x}_0,\widetilde{x}_1,\cdots,\widetilde{x}_k \}\\
    & = & \{(x_0,t_0=0),(x_1,t_1),\cdots,(x_k,t_k)\} \label{eq: embedded chain}
\end{eqnarray}
which is similar to the history defined on page~\pageref{definition:markov property}. In working with the embedded chain, interest lies only in the fact that the system makes the transition $i\to j$ regardless of $\Delta t_{i,j}$ such that $\mathscr{P}(x_{k+1}=j\mid x_k=i) = \lim_{t\to\infty}\{q_{i,j}F_{i,j}(t) \} = q_{i,j}$.The transition matrix of the embedded chain becomes
\begin{equation}
    \widetilde{\mathbf{P}} = [q_{i,j}]
\end{equation}
from which a stationary distribution distribution $\widetilde{\phi} \in \mathbf{\Delta}^{N-1}$ can be computed through solving $\widetilde{\phi} = \widetilde{\phi} \,\widetilde{\mathbf{P}}$ using methods as described on page~\pageref{eq:stationary}. Note that $\widetilde{\phi}_j = \mathscr{P}(\widetilde{x}_k=j)\neq \mathscr{P}(x_k=j)$ in that it does not represent the probability of being found in state $j$ upon inspecting the system. In contrast, it represent the probability of the system being at state $j$ upon inspecting a renewal and will be referred to as the stationary distribution of the embedded chain. The overall stationary distribution $\phi_j = \mathscr{P}(x_k=j)$ can be determined as
\begin{eqnarray}
    \phi_j & = &  \frac{\widetilde{\phi}_j \, \mathbbm
    E[\Delta t_j]}{\sum_{i=1}^N\widetilde{\phi}_i \, \mathbbm{
    E}[\Delta t_i]} \label{eq: overall stationary semi-markov}
\end{eqnarray}
where the expected waiting duration is derived from the holding duration 
\begin{eqnarray}
    \mathbbm{E}[t_i] & = & \sum_{j=1}^N q_{i,j} \mathbbm{E}[\Delta t_{i,j}]\\
    & = & \sum_{j=1}^N q_{i,j} \left(\int_{0}^{\infty} t \,f_{i,j}(t)\, dt\right). \label{eq: expected holding time semi-markov}
\end{eqnarray}
\end{definition}

\subsubsection{Generalised Semi-Markov Process}\label{section:GSMP}
A \emph{Generalised Semi-Markov Process} (GSMP) is formalism used in modelling DEDS \cite{glynn1989gsmp}. Similar to the competing process model, all holding time distributions are generally distributed and need not be memoryless. In contrast to the competing process model, holding time distributions are typically associated with an event $e_m \in \Gamma(i) \subseteq \mathcal{E}$ that induces a state transition $i\to j$ upon triggering where $\mathcal{E}$ is the set of all events and $\Gamma(i)$ is the set of \emph{active events} in state $i$. Moreover, not all process or event lifetimes are renewed when a transition occurs. Hence, a GSMP typically does not exhibit readily occurring renewals such that the Markov property may hold. An example of a renewal in a GSMP or DEDS can be found at the instant when a customer enters an empty $G/G/1$ queue. All GSMPs theoretically have renewal instants \cite{henderson_and_glynn_2001regenerative_GSMP} but these are almost impossible to locate in practice in all but the most simple of DEDS. Hence, attempting to study an embedded chain has been abandoned and it is assumed that only the semi-Markov property holds. 

Ensuring that the semi-Markov property holds is not a trivial task. The GSMP is not conditionally independent of its history $\mathcal{H}(x_k)$ given its current state $x_k$ and may remember its entire trajectory. A solution to this is to study a GSMP using its \emph{General State Space Markov Chain} (GSSMC) \cite{glynn1989gsmp,younes_GSMP_formalism} that describes a transition model over the \emph{extended state space}.

\begin{definition}[\textbf{Extended state-space}]\label{def:extended state space}
The piecewise constant and discrete output of the system $x_n \in \mathcal{X} \subseteq \mathbb{N}_0^d$ is augmented with information regarding the ages $\tau_e$ of the active events $\vec{\tau}_n = [ \tau_{e_1},\cdots,\tau_{e_k} ]$ where $k = |\mathcal{E}|$. If $e_i \in \Gamma(x_n)$ then $\tau_{e_i} \geq 0$ whereas $\tau_{e_i} < 0$ otherwise\footnote{This ensures inactive events do not age. A convenient number such as $-1$ can be selected.}. The extended state-space is denoted as $\hat{x}_n = \left(x_n,\tau_1^{(n)},\cdots,\tau_m^{(n)}\right) = \left(x_n,\vec{\tau}_n\right) \in \hat{\mathcal{X}} = \mathcal{X} \times \mathbb{R}^m$.
\end{definition}
\begin{remark}\label{remark:natural state-space}
The subscript $n$ is used to denote the $n^{th}$ triggered event. Only at these epochs does the \emph{natural state-space} $x_n \in \mathcal{X}$ change.
\end{remark}

Transitions between extended state space describes change in both the \emph{internal}\footnote{The ageing of events is an internal process} and \emph{external} states. This is sufficient to describe the dynamics of a GSMP \cite{younes_GSMP_formalism}.

A GSMP has history-dependent dynamics due to the ongoing concurrent and asynchronous stochastic processes. The vector of active clock ages $\vec{\tau}_n$ is a \emph{sufficient statistic} as it summarises the entire history relevant to the present. This allows for the Markov property to be restored at event occurrences and the semi-Markov property to be present in between as can be found in the GSSMC.

\begin{definition}[\textbf{GSSMC}]\label{def:GSSMC }
Triggering of events form a chain of coordinates in the extended state-space
\begin{eqnarray}
    (x_1,\vec{\tau}_1) \xrightarrow{e_1^*} (x_2,\vec{\tau}_2) \xrightarrow{e_2^*} \cdots \xrightarrow{e_{n-1}^*}   (x_n,\vec{\tau}_n)\nonumber 
\end{eqnarray}
which can be described by the sequence
\begin{eqnarray}
    \left( \hat{x}_k \right)_{k=1}^{n} &=& \left(\hat{x}_1, \hat{x}_2,\cdots, \hat{x}_n  \right)
\end{eqnarray}
\end{definition}
such that $\left( \hat{x}_k \right)_{k=1}^{n}$ is an \emph{embedded Markov Chain}. A notable issue with regards to the GSSMC approach is that the augmented state-space is a hybrid state-space with majority of it (event ages) being continuous. This makes extending a GSMP to a \emph{Generalised Semi-Markov Decision Process} (GSMDP) a difficult task \cite{younes_GSMP_formalism,younes_thesis}. The Bellman equations for a GSMDP have also been presented in \cite{younes_GSMP_formalism,younes_thesis} to have intractable high-dimensional integrals that operate over the event ages $\vec{\tau}$. For this reason, GSMDPs have been mostly solved over augmented state-space using the method of phases \cite{younes_GSMPD_solving1,younes_GSMPD_solving2,younes_thesis} which carries over numerous issues of its own as discussed in \cite{chen_GSMDP_paper} and chapter 2.4.2 of \cite{chen_thesis}.

\begin{remark}
It can be said that the \emph{concurrent} and \emph{asynchronous} properties found in the dynamics of DEDS is attributed to having at least one non-memoryless process remain active once an event is triggered such that its age $\tau_{e_i}$ influences the holding time of the next event. In the competing process model no non-memoryless events remain active once an event has been triggered.
\end{remark}

\subsection{Markov Decision Process}\label{section:MDP}

A \emph{decision process\index{decision process}} can be thought of as any stochastic process that can be controlled at various observation epochs in order to optimise some performance metric driven by the evolution of the system. 

\begin{definition}[\textbf{Markov Decision Process}]
A Markov Decision Process is a framework that allows for sequential decision making problems to be solved in a tractable manner. This framework requires a model to be described in terms of 4-tuple

\begin{equation}\label{eq:mdp tuple}
    \mathcal{M} = \left( \mathcal{X},\mathscr{A},\mathcal{P},\mathscr{C}  \right)
\end{equation}
where
\begin{itemize}
    \item[$\mathcal{X}$\quad] is the state space. In \emph{finite} $n$-dimensional MDPs this consists of a bounded set of integers $\mathcal{X} = \bigtimes_{i=1}^n \mathbb{Z}\cap [\underline{m}_i,\overline{m}_i]$ where $\bigtimes$ denotes the Cartesian product and $\infty<\underline{m}_i<\overline{m}_i<\infty$. If $\mathcal{X}$ has a \emph{continuous} component then an \emph{infinite} dimensional MDP arises.
    \item[$\mathscr{A}$\quad] is the action space. At decision epochs, this space is queried with regards to finite set of available actions $\mathscr{A}\left(x_n\right)$ where $\mathscr{A}: \mathcal{X} \to \mathbb{Z}\cap [\underline{m},\overline{m}]$ and $\infty<\underline{m}<\overline{m}<\infty$.
    \item[$\mathcal{P}$\quad] is the Markov model. It is required to satisfy either the Markov or Semi-Markov property at decision epochs. The state of the system at a decision epoch is often referred to as a \emph{distinguished state} \cite{fox1968semi}.
    \item[$\mathscr{C}$\quad] is the cost or reward structure. 
\end{itemize}
\end{definition}

\subsubsection{Components}\label{section: mdp components}
The MPD tuple (\ref{eq:mdp tuple}) is elaborated upon by discussing how each component plays a role in the execution of the MPD.

\begin{definition}[\textbf{Decision Epoch}]\label{definition:decision epoch}
At time $t$ the system has entered a state $x_n \in \mathcal{X}$ where either the Markov or Semi-Markov property holds. The action set at this state is not empty $\mathscr{A}(x_n) \neq \varnothing$. The decision epoch is not a duration but an instance in time.
\end{definition}

\begin{definition}[\textbf{Distinguished State} \cite{fox1968semi}]
This is a state $\widetilde{x} \in \widetilde{\mathcal{X}} \subseteq \mathcal{X}$ found at a decision epoch. The set of distinguished states is a subset of the embedded chain.
\end{definition}

\begin{definition}[\textbf{Policy} \cite{BertsekasVol2}]
A policy specifies what is action to be taken upon entering a distinguished state. Generally, three binary characteristics describe the nature of the policy.
\begin{enumerate}
    \item \emph{Deterministic vs Randomised: } A deterministic policy is a map $\pi_{D}: \widetilde{\mathcal{X}} \to \mathscr{A}$ such that $a_n = \pi\left(\widetilde{x}_n\right)$ where the decision $a_n$ takes on a given value with probability of one. A randomised policy is a map $\pi_{R}: \widetilde{\mathcal{X}} \to \mathbf{\Delta}^{|\widetilde{\mathcal{X}}|-1}$. Hence $a_n \sim \pi_R\left(\widetilde{x}_n\right)$.
    \item \emph{History-dependent vs independent: } A history dependent policy is a map that requires a history $\mathcal{H}(\widetilde{x}_n)$ of the process up to its current state. A history-dependent decision is provided by a map $\mu_{H}: \mathcal{H} \to \mathscr{A}$ such that $a_n = \mu_{H}(\mathcal{H})$. 
    \item \emph{Stationary vs non stationary policies} This could be rephrased as time-homogeneous vs time-inhomogeneous policies. A non-stationary policy is a time-dependent map $\pi_{t}: \mathbb{R}_{\geq 0} \times \widetilde{\mathcal{X}} \to \mathscr{A}$ such that $a_n(t) = \pi_t\left(\widetilde{x}_n\right)$. For some $t_2 \neq t_1$, it is generally the case that $\pi_{t_1}\left(\widetilde{x}_n=i\right) \neq \pi_{t_2}\left(\widetilde{x}_n=i\right)$ which is attributed to \emph{transient} behaviour. A stationary policy is a map that ignores time $\pi: \widetilde{\mathcal{X}} \to \mathscr{A}$ such that $\forall t \geq 0:\,\, a_n(t) =  a_n = \pi\left(\widetilde{y}_n\right)$.
\end{enumerate}
\end{definition}
If $x_n \in \mathcal{X}\setminus\widetilde{\mathcal{X}}$ then $\pi(x_n) = \varnothing$. Hence, no decisions are made at non-distinguished states. The time horizon also has to be specified as being \emph{finite}, \emph{infinite} or \emph{episodic}. In the finite case it is desirable to control a transient portion of the process. Hence a non-stationary policy would make sense while a stationary policy would make sense for the infinite horizon case. The episodic case implies a goal or absorbing state. The problem runs for an unspecified amount of finite time until it is reached.

The policy interacts with the MDP through specifying a transition and cost model.

\begin{definition}[\textbf{Transition model} \cite{BertsekasVol2}]\label{def: mdp transition model}
A set of Markov chains $\mathcal{P} = \{ \mathbf{P}_{a}:a\in\mathscr{A} \}$ exist as to describe the dynamics of the system between distinguished states under a given action $\widetilde{x}_{n+1} \sim \mathscr{P}(\cdot \mid \widetilde{x}_{n},a_n)$. Ultimately, a single Markov chain $\mathbf{P}_{\pi}=[P_{i,j}^{a}]$ is constructed row-wise by selecting the row from $\mathbf{P}_{\pi(\widetilde{x})}$ that corresponds to $\widetilde{x}$.
\end{definition}

\begin{definition}[\textbf{Cost Structure} \cite{BertsekasVol2}]
The cost structure consists of two parts:
\begin{enumerate}
    \item \emph{Lump-sum cost: } This is the cost incurred upon entering a distinguished state. It is specified by $C_{L}:\widetilde{\mathcal{X}} \to \mathbb{R}_{\geq 0}$ such that 
    \begin{equation}
        C_i^{L,n} = C_{L}\left(\widetilde{x}_n = i\right).\nonumber
    \end{equation}
    where the epoch number $n$ can be dropped when working with a infinite-horizon/stationary problem
    \item \emph{Transition cost: } This is the cost incurred over the duration of a transition between two distinguished states. This cost is incurred over time at some rate $c_i>0$. It is always the case that the transition cost can be converted into a \emph{holding cost} which is a lump-sum cost $C_i = \int_{0}^t c_i\,d\tau$. However, this should be corrected as to include some \emph{discount} or \emph{interest rate} $\beta >0$ such that
    \begin{eqnarray}
        C_i^{\beta} & = & \int_{0}^t c_i \left( \int_{0}^{\tau} e^{-\beta \xi}\, d\xi  \right)\,d\tau \nonumber \\
         & = & \int_{0}^{t} c_i \left(\frac{1-e^{-\beta \tau}}{\beta}\right)\,d\tau
    \end{eqnarray}
\end{enumerate}
The cost structure can be simplified into a single time-homogeneous lump-sum cost.
\begin{equation}\label{eq:general lump sum cost}
    C_i = C_i^{L} +C_i^\beta
\end{equation}
\end{definition}

This fact will play an important role when discussing concept of \emph{cost-per-stage} in the \emph{Bellman Dynamic Programming equations} used minimise the objective function.

\begin{definition}[\textbf{Objective function\index{Objective function}}\cite{cassandras_book}]
In solving for MDPs, the goal is to find a policy that \emph{globally} optimises the objective function. When working with costs, this translates to searching for the global minimum. In the infinite-horizon, an additive discounted cost function solved for using Dynamic Programming finds a globally optimal and \emph{unique} policy (see section 1.2 of \cite{BertsekasVol2}). Such a discounted infinite-horizon objective function starting from an initial distinguished state $\widetilde{x}(t_0=0)=i$ and evolving under policy $\pi$ along with cost rate $C(x(t),\pi(x(t)))$ is computed as below.
\begin{eqnarray}
    J_{\pi}^\beta(i) & = & \mathbbm{E}_{\pi} \left[ \int_{0}^{\infty} e^{-\beta t} C(x(t),\pi(x(t))) \, dt \,\big| \, x(0) =i \right] \\
    & = & \sum_{k=0}^{\infty} \mathbbm{E}_{\pi} \left[  \int_{t_k}^{t_{k+1}}   e^{-\beta t} C(x(t),\pi(x(t))) \, dt \, \big| \, x(t_0) =i  \right] \label{eq:discounted objective function}
\end{eqnarray}
\end{definition}

\subsubsection{Discounted Bellman equations}\label{section: discounted bellman equations}

The Bellman equations allow for a Dynamic Programming solution as to obtain the optimal history-independent deterministic policy $\pi^*$ of a MDP. This policy is at least as good as the best history-dependent randomized policy if the system is Markovian \cite{BertsekasVol2}. Hence without loss of generality, research is restricted to such a class of policies if the Markov or Semi-Markov property holds ( see section 1.1.4 of \cite{BertsekasVol2}).

Time-continuous Operations Research and Financial Engineering applications use an \emph{interest rate} $\beta \geq 0$ to compare future and present values over some operating period $t$ such that $V_{\mbox{present}} = 1/(1+\beta)^t \, V_{\mbox{future}}$. However, the MDP framework uses the notion of a \emph{discount factor} $0\leq \alpha \leq 1$ which is related to the interest rate $\alpha = 1/(1+\beta)$. The discount factor\index{discount factor} usually applies to fixed discrete periods and as such is raised to the power of integer values. 

When $\beta$ is small, it is well known \cite{gosavi2015simulation} that an approximation based on the exponential series leads to an expression for the discount factor over some operating period in terms of the interest rate.
\begin{eqnarray}
    \alpha^t & = & \left( \frac{1}{1+\beta}  \right)^t \nonumber\\
    & \approx & (e^{-\beta})^t \nonumber \\
    & \approx & e^{-\beta t} \label{eq: discount rate}
\end{eqnarray}
This is an important concept as it allows continuous-time problems to be solved over discrete-time intervals while retaining the desired interest rate. Hence $V_{\mbox{present}} = e^{-\beta t}\, V_{\mbox{future}}$. The Bellman Policy \emph{Evaluation}\index{Bellman Policy Evaluation} and Bellman Policy \emph{Improvement}\index{Bellman Policy Improvement}\footnote{The policy evaluation equations lead to the \emph{Policy Iteration} algorithm and the policy improvement equations lead to the \emph{Value Iteration} algorithm as discussed in \cite{gosavi2015simulation}} equations are now presented.

\begin{definition}[\textbf{Discounted policy evaluation equation}]
The intention of this equation is to evaluate the objective function under a \emph{given} stationary policy $\pi$. The \emph{discrete-time} equation is given as
\begin{eqnarray}
    J_{\pi}^{\alpha}(i) & = & C_i^{\pi(i)} + \alpha  \sum_{j \in \mathcal{X}}P_{i,j}^{\pi(i)}  J_{\pi}^{\alpha}(j) \label{eq: discounted bellman policy eval discrete}\\
    & = & \mathbf{T}_{\pi}J_{\pi}^{\alpha}(i) \label{eq: discounted bellman policy eval operator}
\end{eqnarray}
\end{definition}

The Bellman policy evaluation operator $\mathbf{T}_{\mu}$ is a \emph{linear self-mapping operator}. It solves $|\mathcal{X}|$ number of unknown $J_{\pi}(i)$ with $|\mathcal{X}|$ equations such that
\begin{eqnarray}
    \vec{J}_{\pi}^{\alpha} & = & \mathbf{T}_{\pi}\vec{J}_{\pi}^{\alpha}. \label{eq: Bellman linear equations}
\end{eqnarray}
Equation~(\ref{eq: discounted bellman policy eval discrete}) can be formulated as system of $|\mathcal{X}|$ linear equations
\begin{eqnarray}
    \vec{J}_{\pi}^{\alpha} & = & \left( \mathbf{I} - \alpha  \mathbf{P}_{\pi}  \right)^{-1}\vec{C}_{\pi}
\end{eqnarray}
where $\vec{J}_{\pi}^{\alpha}=[J_{\pi}^{\alpha}(i)]$ and $\vec{C}_{\pi} = [C_i^{\pi(i)}]$ are vectors of the state-value functions and lump-sum costs, respectively. It is important to notice that $ I - \alpha  \mathbf{P}_{\pi} $ is always invertible as the eigenvalues of any row stochastic matrix such as $\mathbf{P}_{\pi}$ lies within a unit circle \cite{BertsekasVol2}. This asserts the claim that $\vec{J}_{\pi}^{\alpha}$ is always a unique solution to the \emph{fixed-point} problem (\ref{eq: Bellman linear equations}). 

\begin{definition}[\textbf{Discounted policy improvement equation}]
The intention of this equation is to find a policy $\pi'$ using $\vec{J}_{\pi}$ that has corresponds to $\pi$ such that $\vec{J}_{\pi'} \preceq \vec{J}_{\pi} $ where $\preceq $ denotes an element-wise inequality. The discrete-time equation is formulated as
\begin{eqnarray}
    J_{\pi'}^{\alpha}(i) & = & \min_{\pi'(i) \in \mathscr{A}(i) }\left\{ C_i^{\pi'(i)} + \alpha  \sum_{j \in \mathcal{X}}P_{i,j}^{\pi'(i)}  J_{\pi}^{\alpha}(j)  \right\}\label{eq: discounted bellman policy improve discrete}\\
    & = &  \mathbf{T} J_{\pi}^{\alpha}(i) \label{eq: discounted bellman policy improve operator}
\end{eqnarray}
\end{definition}
where $\mathbf{T}$ is the \emph{non-linear self-mapping} Bellman policy improvement operator.
\begin{definition}[\textbf{Bellman policy improvement operator}]
Technically the Bellman policy improvement\index{Bellman policy improvement} operator, $\mathbf{T}$ solves an optimisation problem as opposed to a system of well-defined linear equations. Some texts \cite{gosavi2015simulation} refer to (\ref{eq: discounted bellman policy improve operator}) as the Bellman optimality equation because $J_{\mu'} \preceq J_{\mu}$. This is optimal in providing a local improvement. Successive application of (\ref{eq: discounted bellman policy improve operator}) is denoted as
\begin{eqnarray}
    J_{\mu_{j+k}} & = &  \mathbf{T}^k J_{\mu_{j}}
\end{eqnarray}
whereby $ \vec{J}_{\pi'} \preceq \vec{J}_{\pi} $ such that under the max-norm $||\cdot||_{\infty}$ and $\mathbf{T}$ being uniformly contracting with modulus $\kappa$ then the \emph{Banach Fixed-Point Theorem} \cite{stachurski2009economic} provides that (where $\ominus$ denotes element-wise subtraction)
\begin{eqnarray}
    || \mathbf{T}^n \vec{J}_{\pi} \ominus \vec{J}_{\pi^*} ||_\infty&  \leq & \kappa^n ||\vec{J}_{\pi} \ominus \vec{J}_{\pi^*}||_{\infty}  \label{eq:bellman limit converge}
\end{eqnarray}
which translates to
\begin{eqnarray}
    \lim_{n\to \infty} \mathbf{T}^n \vec{J}_{\pi} = \vec{J}_{\pi^*} \label{eq:converge}
\end{eqnarray}
where $\vec{J}_{\pi^*}$ is a fixed point and a global optimum in the sense that $\vec{J}_{\pi^*} \preceq \vec{J}_{\pi}, \forall \pi \in \mathscr{A}^{|\mathcal{X}|}$ and $\pi^*$ is the optimal policy. Equation (\ref{eq:bellman limit converge}) is essentially the Value Iteration algorithm \cite{stachurski2009economic}.
\end{definition}

\section{Semi-Markov Decision Process}\label{section: SMDP}

This section contains the main contribution of this paper: to model a controlled non-preemptive two-queue polling model with generally distributed service and switch-over durations and Poisson arrivals.

\subsection{Assumptions}\label{section: SMDP assumptions}
With reference to section~\ref{section: notation}, if all processes in $\mathcal{E}$ were generally distributed (non-memoryless) then the polling model or DEDS would be a GSMP. A GSMP is difficult to solve as a GSMDP while a SMP has been successfully controlled as a \emph{Semi-Markov Decision Process} (SMDP) \cite{howard2012dynamic,gosavi2015simulation} and even a \emph{Partially Observable Semi-Markov Decision Process} (POSMDP) \cite{posmdp}. 

The most direct route in obtaining a SMDP from a GSMDP would be to convert the underlying GSMP into a competing process model (definition~\ref{def: competing process model}). In order to do so, a modelling assumption must be imposed such that all active processes restart once an event has been triggered. As an extreme case, all processes in $\mathcal{E}$ can be assumed to be memoryless. This would lead the system being a CTMC and not a SMP. It is argued that such an approach would be inferior as most non-trivial temporal characteristics would be lost. For example, unreliable service can be modelled as a bi-modal mixture model in a SMP. In a CTMC, an exponential distribution with the same mean would be used as a substitution. The latter case would not capture the risky or unreliable nature of selecting service as a decisions.

For a GSMP to be a competing process model, two scenarios have been identified: 
\begin{enumerate}
    \item Explicitly cancel and restart all generally distributed processes in $\mathcal{E}$ once an event has been triggered.
    \item Have at most once generally distributed process active at a time while the rest are memoryless.
\end{enumerate}

In this paper no valid reason or explanation could be materialised as to justifying an explicit cancel and restart procedure. That is not to say it is not a valid approach. This paper instead proceeds to be more comfortable in adopting the latter as it can be justified in the context of the polling model application. Note that the first scenario has been used in \cite{frank_discrete_GSMDP} which is similar to this paper as it introduces the \emph{resetting discrete GSMDP} as feasible means of modelling a GSMDP/DEDS. It is different in that it focuses on discrete time as opposed to continuous-time used in this paper.

To this end, $\mathcal{E}$ can be partitioned into processes that can only occur in a serial manner to one another $\mathcal{E}_{\perp}$ and those processes that will be observed to run concurrently with at least another process in $\mathcal{E}$. The latter is denoted by $\mathcal{E}_{\parallel}$ such that $\mathcal{E}_{\perp}\cup\mathcal{E}_{\parallel} = \mathcal{E}$ and $\mathcal{E}_{\perp}\cap\mathcal{E}_{\parallel} = \varnothing$. A SMP is obtained through allowing all serial processes $\mathcal{E}_{\perp}$ to remain generally distributed while the concurrent processes $\mathcal{E}_{\parallel}$ are constrained to be memoryless. A third set which contains $\bar{\mathcal{E}}_{\parallel}$ all processes in $\mathcal{E}_{\parallel}$ as non-memoryless can be kept as a backup. The reason for this will become clear below.

In allowing the SMP to execute, at most once process in $\mathcal{E}_{\perp}$ will be active in conjunction with an arbitrary number of active concurrent processes $\mathcal{E}_{\parallel}$. Once this process triggers an event, a renewal will occur as all remaining active events $\Gamma(x_n)$ will be memoryless. A new processes from $F_e \in \mathcal{E}_{\perp}$ can then be appended to the active set $\Gamma(x_{n+1}) = \Gamma(x_{n+1})'\cup\{ F_e \}$ where $\Gamma(x_{n+1})'$ denotes an intermediate set resulting from $\Gamma(x_n)$ by adding/activating or removing/deactivating concurrent processes to suit the needs of $x_{n+1}$. If a new non-memoryless process is not added to $\Gamma(x_{n+1})$ then the process is briefly a CTMC. The decision-maker may be content with this or take the opportunity to swap on of the memoryless processes $G_e \in \Gamma({x_{n+1}})$ for its equivalent general process $F_e \in \bar{\mathcal{E}}_{\parallel}$. Caution would have to be taken as to ensure that this process triggers an event before reintroducing $F_e \in \mathcal{E}_{\perp}$ otherwise scenario 1 occurs of a GSMP is created. This paper chooses not to intervene when the competing process model briefly becomes a CTMC.

It would be ideal if $\mathcal{E}_{\perp} \subseteq \mathscr{A}$ such that decisions pertain to temporally extended actions --- most of which are non-memoryless. 

Lastly, for decisions to be made at distinguished states of a discrete nature as opposed to hybrid states in the embedded chain of a GSSMC (see definition~\ref{def:GSSMC }), decisions can not preempt $F_e \in \Gamma(x_n)\cap \mathcal{E}_{\perp}$. Otherwise, preempting would have to take $\tau_e \in \mathbb{R}_{\geq 0}$ into account. Hence, non-preemptive decisions are taken at decision epochs in the embedded chain of the SMP.

\subsection{Model components}

\subsubsection{Event sets}

The full event-set $\mathcal{E} = \{\lambda_1,\lambda_2,\mu_1,\mu_2,s_{1,2},s_{2,1},\Lambda_1,\Lambda_2  \}$ will be split as $\mathcal{E}_{\perp} = \{ \mu_1,\mu_2,s_{1,2},s_{2,1} \}$ and $\mathcal{E}_{\parallel} = \{\lambda_1,\lambda_2,\Lambda_1,\Lambda_2\}$. The set of feasible actions is consulted at $l_2=0$
\begin{equation}
    \mathscr{A}(n_1,n_2,l_1) = 
    \begin{cases}
    \{ \mu_1,s_{1,2},\Lambda_1 \}, & n_1 > 0 , l_1 = 1 \\
    \{ s_{1,2},\Lambda_1 \}, & n_1 = 0 , l_1 = 1 \\
    \{ \mu_2,s_{2,1},\Lambda_2 \}, & n_2 > 0 , l_1 = 2 \\
    \{ s_{2,1},\Lambda_2 \}, & n_2 = 0 , l_1 = 2
    \end{cases}
\end{equation}
where $l_2=0$ is naturally observed at the embedded chain of the SMP. 

\subsubsection{Policy}
The policy $\pi: \widetilde{\mathcal{X}} \to \mathscr{A}$ is a deterministic, history-independent and stationary policy which guarantees to be optimal in the infinite horizon setting (section~\ref{section: discounted bellman equations}) where $\widetilde{\mathcal{X}}$ is recalled to be the state-space of the embedded chain (definition~\ref{def: embedded chain}). The desirable property $\mathcal{E}_{\perp} \subset \mathscr{A}$ simplifies the decision making process to specifying which generally distributed serial process to commit to at points in the embedded chain. Idling $\Lambda_i$ translates to not introducing a serial process and waiting for the system to transition to another point in the embedded chain i.e. an arrival of any class occurs.

\subsubsection{Transition model}\label{section: smdp transition model}

Once a non-preemptive temporally extended action has been committed to, the possibility of class 1 or class 2 arrival events arises during this interval. The transition model has to take these arrivals into account. As an example, for a completed class 1 service to transition from $x_{k} = (n_1,n_2,l_1=1)$ to $x_{k+1} = (n_1-1,n_2,l_1=1)$ a service event must have triggered without any class 1 or class 2 arrivals occurring \emph{before} it. With reference to the example, the probability $P_{x_k,x_{k+1}}^{l_2=\mu_1} = \mathscr{P}(n_{\lambda_1}=0,n_{\lambda_2}=0\mid e=\mu_1)$ needs to be modelled where $n_{\lambda_i} \in \mathcal{N}_i =  \mathbb{N}\cap[0,N_i]$ denotes the number of class $i$ arrivals that have occurred, $N_i \in \mathbb{N}$ is a truncation/upper bound and $e$ is the triggered event that frees up the server and performs a transition within the embedded chain $x_k \to x_{k+1}$. The triggering event occurs $t_e$ units of time after the system entered $x_k$.

The probabilities $\mathscr{P}(n_{\lambda_1},n_{\lambda_2},t_e=t)$ can be computed using the transient state probabilities (\ref{eq: phi 2.}) as $(n_{\lambda_1},n_{\lambda_2})$ is a \emph{bi-variate birth-process} which is a CTMC. An appropriate generator matrix needs is constructed to model the birth process. As $\mathbf{Q}$ is a two-dimensional object that maps coordinates, $(n_{\lambda_1},n_{\lambda_2})$ needs to be flattened. Let $n_\lambda \in \mathcal{N} = \mathbb{N}\cap [0, N_1 \times N_2]$ be the flattened version of $(n_{\lambda_1},n_{\lambda_2})$ such that it is constructed using a function $\mbox{\texttt{coord2idx}}: \mathcal{N}_1 \times \mathcal{N}_2 \to \mathcal{N} $ such that $i = \mbox{\texttt{coord2idx}}(n_{\lambda_1},n_{\lambda_2}) = n_{\lambda_1}\times N_2 + n_{\lambda_2}$. To reconstruct the original coordinates a function that relies on floor and modulo division $\mbox{\texttt{idx2coord}}:\mathcal{N} \to \mathcal{N}_1 \times \mathcal{N}_2$ is defined as it allows for $(n_{\lambda_1},n_{\lambda_2}) = \mbox{\texttt{idx2coord}}(i) = \left(\,\lfloor i/N_2 \rfloor\,, i \mod N_2\right) $. The generator matrix is constructed over the flat space as follows
\begin{equation}\label{eq: bi-variate generator matrix}
    \mathbf{Q}_{i,j} = 
    \begin{cases}
    \lambda_1, & i = \mbox{\texttt{coord2idx}}(n_{\lambda_1},n_{\lambda_2}), j = \mbox{\texttt{coord2idx}}(n_{\lambda_1}+1,n_{\lambda_2})\\
    \lambda_2, & i = \mbox{\texttt{coord2idx}}(n_{\lambda_1},n_{\lambda_2}), j = \mbox{\texttt{coord2idx}}(n_{\lambda_1},n_{\lambda_2}+1)\\
    -\gamma = - (\lambda_1 + \lambda_2), & i = \mbox{\texttt{coord2idx}}(n_{\lambda_1},n_{\lambda_2}), j = \mbox{\texttt{coord2idx}}(n_{\lambda_1},n_{\lambda_2}) \\
    0, & \mbox{else}
    \end{cases}
\end{equation}
which allows for the transient state probabilities to be computed
\begin{eqnarray}
    \mathscr{P}(n_{\lambda_1},n_{\lambda_2},t) &=& \left( \vec{e}_0 e^{\mathbf{Q}t} \right)_{\mbox{\texttt{coord2idx}}(n_{\lambda_1},n_{\lambda_2})}\\
    & = & \left( \vec{\phi}(t) \right)_{\mbox{\texttt{coord2idx}}(n_{\lambda_1},n_{\lambda_2})}\\
    & = & \left(\bm{\phi}(t) \right)_{n_{\lambda_1},n_{\lambda_2}}
\end{eqnarray}
where $\vec{\phi}(0) = \vec{e}_0$ is a unit vector with a one at index zero. This translates to the fact that it is known with certainty that no arrivals can occur the instant the system enters state $x_k$. Furthermore, $\mathbf{\phi}$ is a $N_1 \times N_2$ matrix as opposed to a length $N_1\times N_2$ vector. Both data structures are probability mass functions i.e. $\sum_{n_{\lambda_1}}\sum_{n_{\lambda_2}}\bm{\phi}_{n_{\lambda_1},n_{\lambda_2}}=1$ and $\sum_{i} \vec{\phi}_i = 1$.

The probabilities of interest can no be obtained $\mathscr{P}(n_{\lambda_1},n_{\lambda_2}\mid e)$ for an event $e$ with generally distributed lifetime $t_e \sim F_{e}$. Using the probability mass function of this lifetime $f_e$, the expected time-invariant probabilities are found
\begin{eqnarray}
    \mathscr{P}(n_{\lambda_1},n_{\lambda_2}\mid e) & = & \left( \int_{0}^{\infty} \vec{e}_0 e^{\mathbf{Q}t}f_{e}(t)\, dt \right)_{\mbox{\texttt{coord2idx}}(n_{\lambda_1},n_{\lambda_2})}\label{eq: arrival probs}\\
    & = & \left( \int_{0}^{\infty} \vec{\phi}(t)f_{e}(t)\, dt \right)_{\mbox{\texttt{coord2idx}}(n_{\lambda_1},n_{\lambda_2})}\\
    & = & \left(\,\vec{\phi}_e\,\right)_{\mbox{\texttt{coord2idx}}(n_{\lambda_1},n_{\lambda_2})} \\
    & = & \left(\,\bm{\phi}_e\, \right)_{n_{\lambda_1},n_{\lambda_2}}.
\end{eqnarray}
Equation~(\ref{eq: arrival probs}) is all that is needed to fully describe the transition probabilities of the entire model --- a Markov chain for each action. Under the service action $l_2 = 1$, $e=\mu_i$ and $f_e = f_{\mu_i}$. Furthermore, let $x_k = (n_1,n_2,l_1=i)$ and $x_{k+1} =\left(n_1+n_{\lambda_1}-\mathbbm{1}_{\{\mu_i=\mu_1\}},n_2+n_{\lambda_2}-\mathbbm{1}_{\{\mu_i=\mu_2\}},l_1=i\right)$ from which the transition model follows.
\begin{eqnarray}
    P_{x_k,x_{k+1}}^{l_2=1} & = & \mathscr{P}(n_{\lambda_1},n_{\lambda_2}\mid e=\mu_i) \\
    & = &  \left(\,\vec{\phi}_{\mu_i}\,\right)_{\mbox{\texttt{coord2idx}}(n_{\lambda_1},n_{\lambda_2})}.
\end{eqnarray}
The switch-over model follows from setting $l_2=2$, $e = s_{i,j}$, $f_e = f_{s_{i,j}}$, $x_{k} = (n_1,n_2,l_1=i)$ and $x_{k+1}=\left(n_1+n_{\lambda_1},n_2+n_{\lambda_2},l_1=j\right)$ where $i \neq j$ such that only the probabilities of arrivals occurring are required
\begin{eqnarray}
    P_{x_k,x_{k+1}}^{l_2=2} & = & \mathscr{P}(n_{\lambda_1},n_{\lambda_2}\mid e=s_{i,j}) \\
    & = &  \left(\,\vec{\phi}_{s_{i,j}}\,\right)_{\mbox{\texttt{coord2idx}}(n_{\lambda_1},n_{\lambda_2})}.
\end{eqnarray}
Idling may seem more complicated as it is defined as waiting for either a single class 1 or class 2 arrival to occur. More than one arrival does not occur! The probability distribution function for the duration of such an arbitrary arrival makes use of one of the many interesting properties of the exponential distribution $t_{\Lambda_i} = \min\{t_{\lambda_1},t_{\lambda_2} \}$ then  $F_{\Lambda_i} =  Exp(\lambda_1+\lambda_2) = Exp(\gamma)$ where $\gamma$ is the global sampling rate or uniformisation factor. Subsequently, if $l_2=0$, $e=\Lambda_i$, $f_e = f_{\Lambda_i}$, $x_k = (n_1,n_2,l_1=i)$, $\chi_{k+1}=\left\{ \left(n_1+1,n_2,l_1=i\right),\left(n_1,n_2+1,l_1=i\right)  \right\}$ then for $x_{k+1} \in \chi_{k+1}$
\begin{eqnarray}
    P_{x_k,x_{k+1}}^{l_2=0} & = & \mathscr{P}(n_{\lambda_1},n_{\lambda_2}\mid e=\Lambda_i) \\
    & = &  \left(\,\vec{\phi}_{\Lambda_i}\,\right)_{\mbox{\texttt{coord2idx}}(n_{\lambda_1},n_{\lambda_2})}.
\end{eqnarray}
While this fits the framework followed by the serve and switch-over actions, a simpler approach using uniformisation can be used
\begin{equation}
    P_{x_k,x_{k+1}}^{l_2=0} = 
    \begin{cases}
    \frac{\lambda_1}{\gamma}, & n_{\lambda_1}=1,n_{\lambda_2}=0\\
    \frac{\lambda_2}{\gamma}, & n_{\lambda_1}=0,n_{\lambda_2}=1\\
    \end{cases}.
\end{equation}
For each action, a  $|\mathcal{X}|\times |\mathcal{X}|$ Markov chain $\mathbf{P}_{l_2} = \left[P_{x_k,x_{k+1}}^{l_2}\right]$ can be constructed. While these Markov chains are stochastic matrices with interpretable probabilities, the discounted Bellman equations of section~\ref{section: smdp bellman equations} will require discounted matrices $\mathbf{P}_{l_2}^{\beta}$ that are not stochastic matrices. The reason behind this will become clear. However, it should be noted that both $\mathbf{P}_{l_2}$ and $\mathbf{P}_{l_2}^{\beta}$ should be kept --- the former for computing the stationary state distribution of the optimal policy and the latter for use in the Bellman equations. The discounted transition probabilities proceed in an identical manner to what has been presented in this section, except that (\ref{eq: arrival probs}) is modified by integrating with the discount factor
\begin{eqnarray}
    \mathscr{P}_{\beta}(n_{\lambda_1},n_{\lambda_2}\mid e) & = & \left( \int_{0}^{\infty} \vec{e}_0 e^{\mathbf{Q}t}f_{e}(t)e^{-\beta t}\, dt \right)_{\mbox{\texttt{coord2idx}}(n_{\lambda_1},n_{\lambda_2})}\label{eq: discounted arrival probs}\\
    & = & \left(\,\vec{\phi}_{e}^{\beta}\,\right)_{\mbox{\texttt{coord2idx}}(n_{\lambda_1},n_{\lambda_2})} \\
    & = & \left(\,\bm{\phi}_e^\beta\, \right)_{n_{\lambda_1},n_{\lambda_2}}.
\end{eqnarray}
where $\mathscr{P}_{\beta}$ is not a probability but conveys the idea that it has been derived from one. For the decision to idle, uniformisation can once more be implemented as a computationally efficient approach
\begin{equation}
    \mathscr{P}_{\beta}(n_{\lambda_1},n_{\lambda_2}\mid \Lambda_i) = 
    \begin{cases}
    \alpha\frac{\lambda_1}{\gamma}, & n_{\lambda_1}=1,n_{\lambda_2}=0\\
    \alpha\frac{\lambda_2}{\gamma}, & n_{\lambda_1}=0,n_{\lambda_2}=1\\
    \end{cases}
\end{equation}
where $\alpha$ is a discount factor
\begin{equation}
    \alpha = \frac{\beta}{\beta+\gamma}.
\end{equation}

\subsubsection{Cost model}\label{section: smdp cost model}
In this section, a discounted lump-sum holding cost is derived. Apart from a switching penalty $K_{i,j}$ incurred upon committing to the decision to switch, all costs are transition costs. The transition cost can be broken up into two parts:
\begin{enumerate}
    \item The expected discounted holding cost of the existing customers (those seen by the system as it enters $x_k$) over the duration of the decision event $t_e$
    \begin{eqnarray}
        C_{\mathcal{H}}(n_1,n_2\mid e) & = & \int_{0}^{\infty} \left(\int_{0}^t e^{-\beta \tau}\,  d\tau\right) \times \left( c_1 n_1 + c_2 n_2 \right) \times f_e(t) \, dt \\
        & = & \left( c_1 n_1 + c_2 n_2 \right) \times \int_{0}^{\infty} \left(\frac{1-e^{-\beta t}}{\beta}\right) \times f_e(t)\, dt \\
        & = &  \left( c_1 n_1 + c_2 n_2 \right) \times C_{\mathcal{H}}(e) \label{eq: existing customer cost}
    \end{eqnarray}
    such that $C_{\mathcal{H}}(e)$ only needs to be computed once for each $e \in \mathcal{E}_{\perp}$. However, uniformisation offers alternative means of computing the costs for the idling decisions. Uniformisation will be preferred.
    \item The expected discounted holding cost of a specified number of customer arrivals during the interval $t_e$
    \begin{eqnarray}
        C_{\mathcal{I}}\left(n_{\lambda_1},n_{\lambda_2}\mid e\right) & = & \int_{0}^{\infty}f_e(t) \times \left( \int_{0}^t e^{-\beta \tau} \left(\vec{c}\, \odot \vec{\phi}_e(\tau)\right)\, d\tau \right)_{\mbox{\texttt{coord2idx}}(n_{\lambda_1},n_{\lambda_2})} \, dt\\
        & = & \int_{0}^{\infty}f_e(t) \times \left(\vec{c}\, \odot \vec{\Phi}_e^{\beta}(t) \right)_{\mbox{\texttt{coord2idx}}(n_{\lambda_1},n_{\lambda_2})} \, dt \label{eq: arrival cost}
    \end{eqnarray}
    where $\vec{c}$ is a cost vector that contains $c_1 n_1 + c_2n_2$ at index $i = \mbox{\texttt{coord2idx}}(n_1,n_2)$, $\odot$ is the element-wise dot product or Hadamard product and $\vec{\Phi}_e^{\beta}(t) = \int_{0}^t e^{-\beta \tau} \vec{\phi}_e(\tau)\, d\tau$ is a vector function $\vec{\Phi}_{e}^{\beta}: \mathbb{R}_{\geq 0} \to \mathbb{R}_{\geq 0}^{|\mathcal{N}|}$ that relies heavily on $\vec{\phi}_{e}(\tau)$. As such it relies heavily on repeated computation of the matrix exponential over a series of time points. This is a time consuming process if not optimised for. With reference to the ODE approach of page~\pageref{eq: forward euler}, it is recommended that for each $e \in \mathcal{E}_{\perp}\setminus\{\Lambda_1,\Lambda_2\}$ a sequence $\Psi_e(\mathcal{T}) = \left\{ \left( \tau_k, \vec{\phi}_e(\tau_k) \right) \right\}_{\tau_k \in \mathcal{T}}$ be computed using an ODE solver such as the explicit forward Euler method~(\ref{eq: forward euler}). This would allow $\vec{\Phi}_e^\beta(t)$ to be approximated by through $\exp\left\{-\beta t\right\}\vec{g}(t)$ where $\vec{g}(t):\mathbb{R}_{\geq 0} \to \mathbb{R}_{\geq 0}^{|\mathcal{N}|}$ is some function fitted to $\Psi_e(\mathcal{T})$. This is a worthwhile trade-off between memory and computation time. Finally, to get the expected discounted holding cost of an unspecified number of customer arrivals a double summation is performed
    \begin{equation}
        C_{\mathcal{I}}^e = \sum_{n_{\lambda_1}=0}^{N_1} \sum_{n_{\lambda_2}=0}^{N_2} C_{\mathcal{I}}\left(n_{\lambda_1},n_{\lambda_2}\mid e\right)\label{eq: arrival cost summed}
    \end{equation}
    which only needs to be computed once for each $e$ and stored.
\end{enumerate}
The discounted lump-sum holding cost for state $x_k = (n_1,n_2,l_1)$ under server action $l_2=e$ is the sum of these two costs
\begin{eqnarray}\label{eq: lump sum cost final}
    C_{n_1,n_2}^{\beta,e} & = &C_{\mathcal{H}}(n_1,n_2\mid e) +  C_{\mathcal{I}}^e
\end{eqnarray}
such that a length $|\mathcal{X}|$ vector $\vec{C}_{l_2}^\beta$ can accompany its corresponding $\mathbf{P}_{l_2}^\beta$ and $\mathbf{P}_{l_2}$ where $C_{n_1,n_2}^{\beta,e}$ has been placed at index $i = \mbox{\texttt{coord2idx}}(n_1,n_2)$. For the action to idle, uniformisation provides a simple alternative. Furthermore, it should be noted that idling only requires the holding cost pertaining to existing customers as it ends once an unseen customer arrives. The discounted lump-sum holding cost for is provided as follows
\begin{eqnarray}
    C_{n_1,n_2}^{\beta,\Lambda_i} & = & \left(\frac{c_1}{\beta+\gamma}\right)n_1 + \left(\frac{c_2}{\beta+\gamma}\right)n_2 \\
    & = & \bar{c}_1 n_1 + \bar{c}_2 n_2.
\end{eqnarray}

\subsection{Model errors due to truncation}

A computer implementation of the model can only consider a finite number of arrivals and transitions between a finite state-space. This corresponds to performing truncation on the generator matrix and transition probability matrix, respectively. 

\subsubsection{Generator matrix truncation}\label{section: generator matrix truncation}

The generator matrix of the bi-variate birth process~(\ref{eq: bi-variate generator matrix}) has played a fundamental role in building the SMDP cost and transition models. However, (\ref{eq: bi-variate generator matrix}) did not address the fact that $\mathbf{Q}$ is of finite dimensions $|\mathcal{N}|\times |\mathcal{N}|$. Naturally, the question arises of how $\mathbf{Q}$ should be truncated for all the cases when $n_{\lambda_i} = N_i,\,i=1,2$. Two options are presented in this section: the \emph{absorbing state} approach and the \emph{unassigned outflow} approach. 

The absorbing state approach assigns all flows to eventually reach the largest state/index $\bar{i} = \mbox{\texttt{coord2idx}}(N_1,N_2)$ such that the generator matrix takes the form
\begin{equation}
    \mathbf{Q}_{|\mathcal{N}|\times |\mathcal{N}|} = 
    \begin{bmatrix}
    \mathbf{Q}_{(|\mathcal{N}|-1)\times (|\mathcal{N}|-1)} & \vec{\lambda}_{(|\mathcal{N}|-1)\times 1}\\
    0 & 0 
    \end{bmatrix}\label{eq: absorbing state representation}
\end{equation}
where $\vec{\lambda}$ is a vector of rates flowing into the absorbing state which is found at the bottom right corner. This approach allows equation~(\ref{eq: bi-variate generator matrix}) to be truncated as below
\begin{equation}
    \mathbf{Q}_{i,j} = 
    \begin{cases}
    \lambda_1, & i = \mbox{\texttt{coord2idx}}(N_1-1,n_{\lambda_2}), j = \mbox{\texttt{coord2idx}}(N_1,n_{\lambda_2})\\
    \lambda_2, & i = \mbox{\texttt{coord2idx}}(n_{\lambda_1},N_2-1), j = \mbox{\texttt{coord2idx}}(n_{\lambda_1},N_2)\\
    -\lambda_1, & i = \mbox{\texttt{coord2idx}}(N_1-1,n_{\lambda_2}), j = \mbox{\texttt{coord2idx}}(N_1-1,n_{\lambda_2})\\
    -\lambda_2, & i = \mbox{\texttt{coord2idx}}(n_{\lambda_1},N_2-1), j = \mbox{\texttt{coord2idx}}(n_{\lambda_1},N_2-1)\\
    0, &  i = \mbox{\texttt{coord2idx}}(N_1,N_2), j = \mbox{\texttt{coord2idx}}(N_1,N_2)
    \end{cases}\label{eq: absorbing state cases}
\end{equation}
The unassigned outflow approach proceeds to pretend if the matrix was never truncated through extracting a sub-generator matrix from a seemingly infinite or much larger parent

\begin{equation}
    \mathbf{Q}_{(|\mathcal{N}|+|\mathcal{M}|)\times(|\mathcal{N}|+|\mathcal{M}|)} = 
    \begin{bmatrix}
    \mathbf{Q}_{|\mathcal{N}|\times |\mathcal{N}|} & \mathbf{Q}_{|\mathcal{N}|\times |\mathcal{M}|} \\
    \mathbf{Q}_{|\mathcal{M}|\times |\mathcal{N}|} & \mathbf{Q}_{|\mathcal{M}|\times |\mathcal{M}|}
    \end{bmatrix}\label{eq: unassigned outflow representation}
\end{equation}
where $|\mathcal{M}|>1$. The top left entry is extracted as the desired child matrix. The name of this approach comes from the fact that all states in the child matrix will have the usual inflow $-(\lambda_1 + \lambda_2)$ but not all states with have the full outflow $(\lambda_1 + \lambda_2)$. More specifically, the states subject to truncation will not have the full outflow. The missing outflow is usually assigned to states in the top right entry. This entry is a fictitious entity such that the these outflows are never assigned a designated state. Hence, the outflow are simply unassigned. This approach truncates (\ref{eq: bi-variate generator matrix}) as follows
\begin{equation}
    \mathbf{Q}_{i,j} = 
    \begin{cases}
    \lambda_1, & i = \mbox{\texttt{coord2idx}}(N_1-1,n_{\lambda_2}), j = \mbox{\texttt{coord2idx}}(N_1,n_{\lambda_2})\\
    \lambda_2, & i = \mbox{\texttt{coord2idx}}(n_{\lambda_1},N_2-1), j = \mbox{\texttt{coord2idx}}(n_{\lambda_1},N_2)\\
    -(\lambda_1+\lambda_2), & i = \mbox{\texttt{coord2idx}}(N_1,N_2), j = \mbox{\texttt{coord2idx}}(N_1,N_2).
    \end{cases}\label{eq: unassigned outflow cases}
\end{equation}
In comparing (\ref{eq: absorbing state cases}) with (\ref{eq: unassigned outflow cases}) it is apparent that the absorbing state case maintains a valid generator matrix where inflow and outflow is balanced. The same cannot be said for the unassigned outflow method. The outcome of this is that when $t$ is large enough such that most of the weight of state probabilities are near the truncated region on $\vec{\phi}(t)$ then the following is observed:
\begin{itemize}
    \item The absorbing state approach maintains a valid probability mass function $\vec{\phi}(t)$ but biases the absorbing state/index $\bar{i}$.
    \item The unassigned outflow approach does not introduce any bias but it fails to produce a valid probability mass function $\sum_{j} \vec{\phi}_j(t) < 1$.
\end{itemize}
These results are illustrated in figure~\ref{fig:state probs truncation}. The lowest entry of this figure contains simulated state probabilities. Simulation is not subject to truncation and can in theory deal with very large state spaces if given enough time to run and large enough sample size. Its purpose is to illustrate that the unassigned outflow approach does not introduce any bias. In essence, simulation reveals the \emph{true model} $\vec{\varphi}$ impervious to truncation.

It is evident that truncation may introduce two types of errors: \emph{bias} $\epsilon_1(N_1,N_2,t) = \sum_{j=1}^{N_1\times N_2}|\vec{\varphi}_j(t) - \vec{\phi}_j(t)|$ and \emph{validity} $\epsilon_2(N_1,N_2,t) = 1 - \sum_{j=1}^{N_1\times N_2} \vec{\phi}_j(t)$. These errors can be avoided if $N_1$ and $N_2$ are chosen large enough. However, as $t \to \infty$ then $N_i \to \infty$ in order to maintain $\epsilon_1$ and $\epsilon_1$ at zero. While integrals such as those found in the transition model~(\ref{eq: arrival probs}) and cost model (\ref{eq: arrival cost}) contain an integral over the domain $t \in \mathbb{R}_{\geq 0}$, relevant computation is performed only where $\forall t \in \mathbb{R}_{\geq 0}: \, f_e(t) > \varepsilon $. Here $\varepsilon$ is a very small constant such as \emph{machine precision}. This means a finite upper bound $\bar{t}_e$ can be established as to allow for finite $N_i$ ensuring zero errors. 

Some of the event duration density functions may have finite support $t \in [ \underline{t},\overline{t}]$ such that the upper bound is clear. However, for those with infinite support $t \in \mathbb{R}_{\geq 0}$ (such as the exponential distribution) an approximate upper bound is established through using a percentile $\alpha$ very close to one. It follows that $\bar{t}_e = F_e^{-1}(\alpha)$ where $F_e^{-1} : \mathbb{R}\cap[0,1] \to \mathbb{R}_{\geq 0}$ is the inverse of the cumulative function also known as the quantile function. Even though finite $N_i$ can be selected such that $\epsilon_j(N_1,N_2,\bar{t}_e) = 0$, the resulting $(N_1N_2) \times (N_1N_2)$ sparse generator matrix may not allow for computation in reasonable time. As such a $N_i^{*} < \underline{N}_i$ can be selected as a trade-off between error and computation budget where $\underline{N}_i$ is the smallest integer in asserting $\epsilon_j ( \underline{N}_1, \underline{N}_2, \overline{t}_e) = 0$.

\begin{figure}[!htbp]
    \centering
    \includegraphics[width=0.6\textwidth]{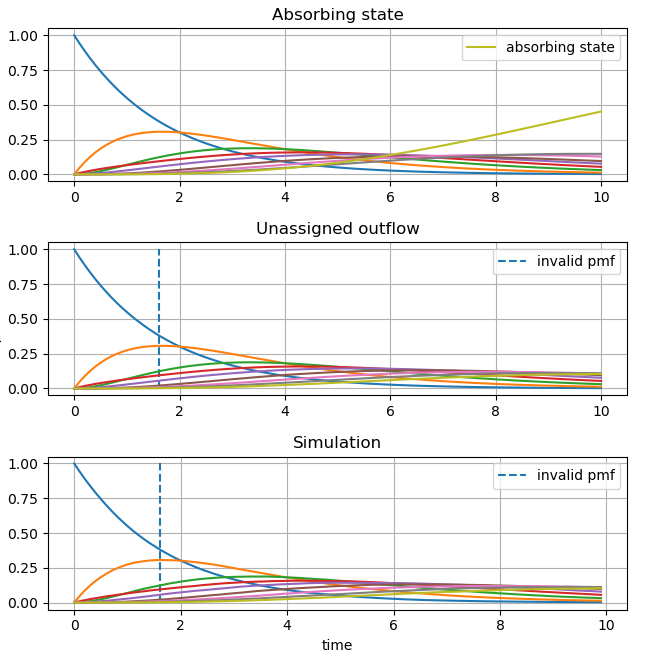}
    \caption{State probabilities $\vec{\phi}(t)$.}
    \label{fig:state probs truncation}
\end{figure}

In performing truncation with $N_i^*$, it would be helpful to understand how error is produced by the two methods. Figure~\ref{fig:cost truncation} presents an empirical comparison between the two approaches in computing (\ref{eq: arrival cost summed}). The two approaches have been compared against the true model $\vec{\varphi}(t)$ as to see which method follows it most closely. The evolution of cost has been presented graphically as it is a more reasonable quantity to visualise than the transition model. Nonetheless, the expected arrival probabilities (\ref{eq: arrival probs}) have been presented as a vector $\vec{P}_e$ for the same problem in table~\ref{tab:transition probabilites} when $N_i^* < \underline{N}_i$ such that error is introduced. More specifically, $\vec{P}_e = [\mathscr{P}(n_{\lambda_1},n_{\lambda_2}|e)]$ where this entry is found at index $i = \mbox{\texttt{coord2idx}}(n_{\lambda_1},n_{\lambda_2})$.

\begin{figure}[!htbp]
    \centering
    \includegraphics[width=0.6\textwidth]{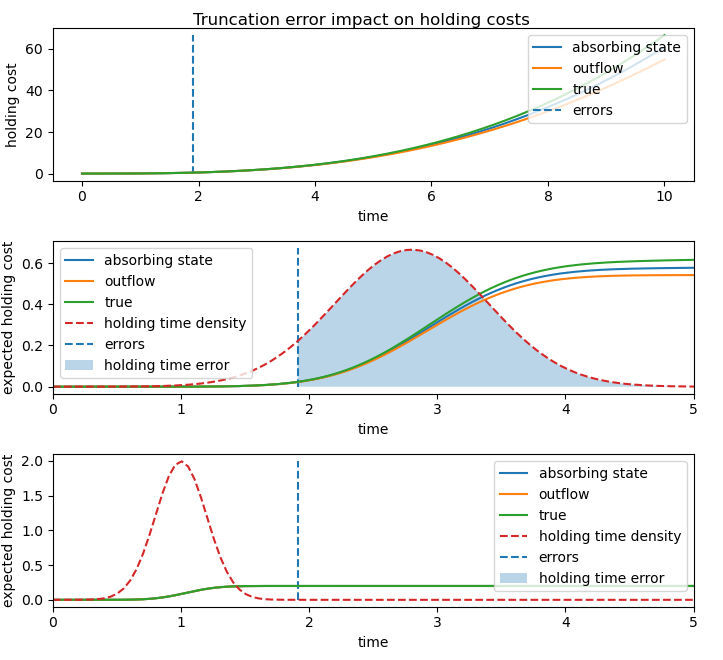}
    \caption{Holding cost error analysis. TOP: holding cost over time MIDDLE: expected holding cost with $N_i^*< \underline{N}_i$ BOTTOM: expected holding cost with $N_i^* \geq \underline{N}_i$}
    \label{fig:cost truncation}
\end{figure}

\begin{table}[!htbp]
    \centering

\begin{tabular}{lrrrr}
\toprule
Index &  Absorbing &  Unassigned &  True (Snapshot) & True (Pooled) \\
\midrule
0 &   0.575335 &    0.575335 &  0.575335 &  0.575335 \\
1 &   0.156952 &    0.156952 &  0.156952 &  0.156952 \\
2 &   0.022444 &    0.022444 &  0.022444 &  0.024857 \\
3 &   0.159181 &    0.159181 &  0.156952 &  0.156952 \\
4 &   0.045060 &    0.045060 &  0.044888 &  0.044888 \\
5 &   0.006699 &    0.006699 &  0.006688 &  0.007435 \\
6 &   0.025593 &    0.023133 &  0.022444 &  0.024857 \\
7 &   0.007496 &    0.006743 &  0.006688 &  0.007435 \\
8 &   0.001239 &    0.001037 &  0.001033 &  0.001287 \\
\bottomrule
\end{tabular}
    \caption{Transition probabilities under truncation ($N_i^* < \underline{N}_i$) and the true model. Note that $N_1^*=N_2^*=3$.}
    \label{tab:transition probabilites}
\end{table}

Before proceeding to discuss figure~\ref{fig:cost truncation} and table~\ref{tab:transition probabilites}, the following two conjectures are made.
\begin{conjecture}\label{conjecture: absorbing state error}
In performing error-prone truncation ($N_i^* < \underline{N}_i$) using the \emph{absorbing state} approach then $\epsilon_1(N_1^*,N_2^*,\bar{t}) > 0$ and $\epsilon_2(N_1^*,N_2^*,\bar{t}) = 0$. Hence, it can maintain a valid probability mass function but biases larger states.
\end{conjecture}

\begin{conjecture}\label{conjecture: unassigned outflow error}
In performing error-prone truncation ($N_i^* < \underline{N}_i$) using the \emph{unassigned outflow} approach then $\epsilon_1(N_1^*,N_2^*,\bar{t}) = 0$ and $\epsilon_2(N_1^*,N_2^*,\bar{t}) > 0$. Hence, it does not bias any states but cannot maintain a valid probability mass function.
\end{conjecture}

The top entry of figure~\ref{fig:cost truncation} shows that in the long-run, the absorbing state approach is able to follow the true model more closely. The same result holds when computing the expected cost with respect to the holding time density of the decision event $f_e$. The vertical blue dotted line represent the point at which the absorbing state starts to introduce bias and the unassigned outflow method fails to maintain a valid probability mass function. The bottom most figure shows that when error-free truncation ($N_i^* \geq \underline{N}_i$) is performed both methods follow the true model and the choice of method is trivial.

Table~\ref{tab:transition probabilites} presents the expected arrival probabilities for the error-prone case. In other words, it presents $\vec{P}_e$ for the same experiment as the middle plot in figure~\ref{fig:cost truncation}. Note that $N_i^*$ has been chosen to be 3. Two columns have been allocated to the true model. The true model is essentially a model with a very large generator matrix $N_i^{\varphi} \ggg N_i^*$. Hence, the \emph{snapshot} column returns $\vec{P}_{j}^{\varphi}$ for $j$ that corresponds to $i$ in $\vec{P}^{\phi}$ such that both indices represent the same coordinate\footnote{Note that the true model will have different \texttt{coord2idx} and \texttt{idx2coord} functions than the truncated models.}. Subsequently, it depicts only a restricted portion of the model and does not sum up to one. The \emph{pooled} column corrects for this. As the name implies, it pools the unseen probabilities of the true model into one of the seen/snapshot states. This is done by pooling under three cases such that $\forall n_{\lambda_1} \in \mathbb{N}\cap [0,N_1^{\varphi}],n_{\lambda_2}\in \mathbb{N}\cap [0,N_2^{\varphi}]:$ 

\begin{equation} \label{eq: state probs pooling}
    \vec{P}_j^{\varphi} \xrightarrow[]{+} 
    \begin{cases}
    \vec{P}_{i_1}^{\phi}, &  n_{\lambda_1} \leq N_1^*, n_{\lambda_2} > N_2^* \\
    \vec{P}_{i_2}^{\phi}, &   n_{\lambda_1} > N_1^*, n_{\lambda_2} \leq N_2^* \\
    \vec{P}_{i_3}^{\phi}, &  n_{\lambda_1} > N_1^*, n_{\lambda_2} > N_2^*
    \end{cases}
\end{equation}
where $i_1 = \mbox{\texttt{coord2idx}}(n_{\lambda_1},N_2^*)$, $i_2 = \mbox{\texttt{coord2idx}}(N_1^*,n_{\lambda_2})$ and $i_3 = \mbox{\texttt{coord2idx}}(N_1^*,N_2^*)$ using the truncated model mapping function. Furthermore, $j =\mbox{\texttt{coord2idx}}(n_{\lambda_1},n_{\lambda_2})$ is found using the mapping function of the true model $\varphi$. In (\ref{eq: state probs pooling}), $\xrightarrow[]{+}$ denotes the pooling/accumulator operator such that all entries in the left that are mapped to a common index on the right will be summed together. The pooled column contains biased results.

Table~\ref{tab:transition probabilites} helps to support the logic behind conjectures~\ref{conjecture: absorbing state error}~and~\ref{conjecture: unassigned outflow error}. The unassigned outflow method corresponds to the snapshot column while the absorbing state approach corresponds to the pooled column. Hence, the absorbing state can be seen as a corrected version of the unassigned outflow method. The fact that the unassigned outflow method follows the true model should be of little surprise as it was always intended to result from a sub-generator matrix of the true model (\ref{eq: unassigned outflow representation}). Furthermore, the absorbing state performs pooling \emph{without} actually computing the larger true model matrix.

An alternative that has been ignored is to normalise the unassigned outflow approach such that a valid probability mass function may be obtained. This would bias all the states and in an overly optimistic manner. Put differently, lower indices that corresponds to less arrivals will have additional and unjustified probability assigned to it. The pooling and absorbing state corrections allocate the unseen probabilities of higher indices (worse states) to the most pessimistic indices. While pessimistic, it can be seen from figure~\ref{fig:cost truncation} that the absorbing state approach produces expected costs that are optimistic when compared to the true cost --- it still underestimates during error-prone truncation.

In conclusion, if $N_i^* < \underline{N}_i$ due to constraints on the computational budget then the absorbing state approach to truncation should be preferred. It underestimates the expected holding cost (due to arrivals) to a lesser degree and provides less overly optimistic arrival probabilities (recall that arrivals dictate the next state in the embedded chain). Otherwise, if $N_i^* \geq \underline{N}_i$ then the choice of method is trivial.

\subsubsection{Transition matrix truncation}\label{section: tpm truncation}

The $|\mathcal{X}|\times |\mathcal{X}|$ transition matrix $\mathbf{P}_{l_2}$ describes a transition from $x_k = (n_1,n_2,l_1) \to  x_{k+1} = (n_1+n_{\lambda_1}-e_{\mu_1},n_2+n_{\lambda_2}-e_{\mu_2},l_1')$ using probabilities extracted from $\vec{P}_{l_2}$. It is seen that arrivals play a large role in driving state transitions. Here, $e_{\mu_i}$ denotes the fact that a completed service event may have decremented the queue number and $l_1'$ proposes that a switching event may have changed the server location. Section~\ref{section: generator matrix truncation} has already discussed the truncation errors found in $\vec{P}_{l_2}$ due to a truncated generator matrix. This section briefly deals the more straightforward issue of where to assign $\mathscr{P}(x_{k+1}\mid x_{k},l_2)$ given that $x_{k+1}> |\mathcal{X}|$.

Section~\ref{section: notation} described $\mathcal{X} \subseteq \mathbb{N}_{0}$ as for generalisation purposes and to highlight the fact that a queuing system may have dimensions of infinite length. If $\mathcal{X}_i = \mathbb{N}\cap [0,X_i]$ where $X_i \in \mathbb{N}\cap [1,\infty)$ then $\mathcal{X}$ can be constrained to be finite $\mathcal{X} = \mathcal{X}_1 \times \mathcal{X}_2 \times \{ 0,1 \}$. Hence, $\mathbf{P}_{l_2}$ is a finite data-structure that will receive construction instructions from an infinite size kernel $\mathscr{P}(x_{k+1}|x_k,l_2)$. Donor entries will come from the arrival probability vector $\vec{P}_{l_2}^{\phi}$ while truncation is performed by the pooling method, similar to (\ref{eq: state probs pooling}). An additional pair of navigation functions will be required: $\mbox{\texttt{triple2idx}}(n_1,n_2,l_1) = 2 n_1 X_2 + 2 n_2 + l_1$ and $\mbox{\texttt{idx2triple}}(i) = \left( \lfloor i/(2X_2) \rfloor \, , \, \lfloor (i \mod 2 X_2)/2 \rfloor\, , \, (i\mod 2X_2) \mod 2\right)$.

Consider the $i^{th} = \mbox{\texttt{triple2idx}}(n_1^k,n_2^k,l_1^k) = \mbox{\texttt{triple2idx}}(x_k)$ of $\mathbf{P}_{l_2}$ such that the following construction takes place using the notation $\Delta n_i = n_i^{k+1} - n_i^{k}$
\begin{equation}\label{eq: transition model pooling}
    \left(\vec{P}_{l_2}^{\phi}\right)_{\xi} \xrightarrow[]{+} 
    \begin{cases}
    \left(\mathbf{P}_{l_2}\right)_{i,j}, & n_1^{k+1} < X_1,n_2^{k+1} < X_2 \\
    \left(\mathbf{P}_{l_2}\right)_{i,j_1}, & n_1^{k+1} \geq X_1,n_2^{k+1} < X_2 \\
    \left(\mathbf{P}_{l_2}\right)_{i,j_2}, & n_1^{k+1} < X_1,n_2^{k+1} \geq X_2 \\
    \left(\mathbf{P}_{l_2}\right)_{i,j_3}, & n_1^{k+1} \geq X_1,n_2^{k+1} \geq X_2 
    \end{cases}
\end{equation}
where
\begin{itemize}
    \item $\xi =\mbox{\texttt{coord2idx}}(\Delta n_1^k,\Delta n_2^k)$ using the mapping function that pertains to $\phi$. If $l_2 = 1$ then $\xi = \mbox{\texttt{coord2idx}}(\Delta n_1^k +\mathbbm{1}_{\left\{l_1=1 \right\}},\Delta n_2^k+\mathbbm{1}_{\left\{l_1=2 \right\}})$ as to correct for the serviced customer.
    \item $j = \mbox{\texttt{triple2idx}}(n_1^{k+1},n_2^{k+1},l_1^{k+1})$ is a unique index that will receive only one probability. This is the truncation free part of the transition model.
    \item The following indices will be reused an denote entries that will receive more than one probability. This region of the transition model experiences truncation.
    \begin{itemize}
        \item $j_1 = \mbox{\texttt{triple2idx}}(X_1,n_2^{k+1},l_1^{k+1})$
        \item $j_2 = \mbox{\texttt{triple2idx}}(n_1^{k+1},X_2,l_1^{k+1})$
        \item $j_3 = \mbox{\texttt{triple2idx}}(X_1,X_2,l_1^{k+1})$
    \end{itemize}
\end{itemize}
For idling $l_2=2$, the same constructions holds. The only difference is that a single arrival will occur. This greatly reduces the amount of pooling performed and truncation error. 

Additional caveats includes dealing with infeasible actions and choosing a size $N_i$ as to regulate the number of permissible arrivals. The former occurs at rows $i = \mbox{\texttt{triple2idx}}(x_k)$ where at $n_j=0,l_1=j$. The entries in this row are trivial as this row will never be consulted by the Bellman Equations of algorithms that solve for the optimal policy. As such, it can be left as a row of zeros. Such a decision would make $\mathbf{P}_{l_2}$ a non-stochastic matrix. However, $\mathbf{P}_{l_2}$ are data-structures intended for storage and were never required to be stochastic matrices. On the other hand, $\mathbf{P}_{\pi}$ (see definition~\ref{def: mdp transition model}) is required to be a stochastic matrix.

In dealing with the appropriate selection of $N_i$, two aspects require attention.
\begin{enumerate}
    \item If $N_i$ is too small a truncation error in the generator matrix $\epsilon_i$ will be present.
    \item If $N_i \geq X_i$ for $i=1,2$ then $\forall x_{k},x_{k+1} \in \mathcal{X}: \, \mathscr{P}(x_{k+1}|x_k,l_2)\geq 0$ if $j_{k+1} \geq j_k$ else $\mathscr{P}(x_{k+1}|x_k,l_2)= 0$ where $j_k = \mbox{\texttt{triple2idx}}(x_k)$. This case considers the most potential transitions and would be recommended if the model has long switch-over durations. If $N_i < X_i$ sparsity will be introduced to the upper right of the transition model. Wit short switch-over durations, this approach may yield reduced memory requirements of a sparsity aware data structure is used for $\mathbf{P}_{l_2}$ such as a \emph{list-of-lists} or \emph{Harwell-Boeing format} (see chapter 10.3.4 of \cite{stewart2009probability}).
\end{enumerate}

Lastly, the discounted transition model $\mathbf{P}_{l_2}^{\beta}$ will have to be constructed as in (\ref{eq: transition model pooling}) using the discounted arrival probabilities $\vec{P}^\beta_{l_1}$ that were computed from (\ref{eq: discounted arrival probs}).

\subsection{Bellman Equations}\label{section: smdp bellman equations}

The standard MDP framework of section~\ref{section:MDP} does not explicitly take time into account. It only sees transitions and does not distinguish between them based on some holding time. In fact, it assumes that all transitions are of the same unit length as is apparent from the constant discount factor $\alpha$ in equation~(\ref{eq: discounted bellman policy eval discrete}). The standard Bellman equation can handle SMDPs without any modification if time has been implicitly accounted for in the cost $\vec{C}_{\pi}$ and transition $\mathbf{P}_{\pi}^{\beta}$ model. As the transition model already takes variable discounting into account (varies based on decision event $l_2=e$), setting $\alpha=1$ gives the following Bellman policy evaluation equations
\begin{eqnarray}
    \vec{J}_{\pi}^{\beta} & = & \vec{C}_{\pi} + \mathbf{P}_{\pi}^{\beta}\vec{J}_{\pi}^{\beta}\label{eq: smdp bellman policy evaluation equations}\\
    & = & \mathbf{T}_{\pi} \vec{J}_{\pi}^{\beta}
\end{eqnarray}
where $\vec{C}_{\pi}$ is constructed from $\mathcal{C} = \{C_{l_2=a}: a \in \mathcal{A} \}$ and $\mathbf{P}_{\pi}^\beta$ from  $\mathcal{P}^{\beta} = \{ \mathbf{P}_{l_2=a}^\beta: a \in \mathcal{A} \}$. For reference, $\vec{C}_{l_2}$ is defined by (\ref{eq: lump sum cost final}) while $\mathbf{P}_{l_2}^{\beta}$ is defined by (\ref{eq: discounted arrival probs}) an truncated using (\ref{eq: transition model pooling}). This paper uses the policy iteration algorithm as opposed to value iteration. This is because the system of well-defined linear equations (\ref{eq: smdp bellman policy evaluation equations}) is deemed small enough to solve numerically using using commercial solvers. If the systems of equations were too large, \emph{partial policy evaluation}\footnote{This is an iterative method that iterates the bellman policy evaluation equations (\ref{eq: discounted bellman policy eval discrete}) until some convergence is met. It is like value iteration but without the $\max/\min$ operator.} (see chapter 5.8 of \cite{gosavi2015simulation}) or value iteration would be more suitable options for an exact solution. Simulation-based or \emph{reinforcement learning} (RL) methods that provide approximate solutions are not discussed as it is outside the scope of this paper. Such methods can be found in chapter 6 of \cite{gosavi2015simulation} where these RL algorithms are applied to MDPs and SMDPs in both the discounted and average reward cases.

\begin{algorithm}[!htbp]
    \caption{SMDP Policy Iteration}
    \label{algorithm: smdp policy iteration}
    \begin{algorithmic}[1]
    \Procedure{PolicyIteration}{$\mathcal{C}$,$\mathcal{P}^\beta$,$\vec{\pi}_0$,\texttt{maxiter}}
    \If {$\vec{\pi}_0$ \textbf{is} None}
        \State Randomly initialise $\vec{\pi}_0 = [a \sim \mathcal{A}(i)],\quad \forall i \in \mathcal{X}$\Comment{Uniform sampling}
    \EndIf
    \State $k \leftarrow 0$
    \State \texttt{continue} $\leftarrow$ True
    \While {\texttt{continue} \textbf{is} True}
    \State Construct $\vec{C}_{\pi_k}$ and $\mathbf{P}_{\pi_k}^\beta$ from $\mathcal{C}$ and $\mathcal{P}^\beta$ under instructions from $\vec{\pi}_k$.
    \State \textit{Policy evaluation}: solve for $\vec{J}_{\pi_k}$ in
    \begin{equation}\label{eq: policy evluation solve for J}
        \left( \mathbf{I} - \mathbf{P}_{\pi_k}^{\beta} \right)\vec{J}_{\pi_k}  = \vec{C}_{\pi_k}.
    \end{equation}
    \State \textit{Policy Improvement}: create an empty vector $\vec{\pi}_{k+1}$ of length $|\mathcal{X}|$. 
    \State $\Delta_k \leftarrow$ False
    \For {$i = 1,2,\cdots, |\mathcal{X}|$}
        \State Select the optimal action from a greedy one-step look-ahead optimisation.
        \begin{equation}
            \vec{\pi}_{k+1}(i) = \mbox{argmin}_{a \in \mathcal{A}(i)} \left\{ \vec{C}_{a}(i) + \sum_{j=1}^{|\mathcal{X}|}  \mathbf{P}_{a}^{\beta}(i,j) \vec{J}_{\pi_k}(j)   \right\} 
        \end{equation}
        \If {$\vec{\pi}_{k+1}(i) \neq \vec{\pi}_{k+1}(i)$ \textbf{and} $\Delta_k$ \textbf{is} False}\Comment{Element-wise equality check}
        \State $\Delta_k \leftarrow $ True
        \EndIf
    \EndFor
    \If {$\Delta_k$ \textbf{is} False \textbf{or} $k =$ \texttt{maxiter}}\Comment{Breaking conditions}
        \State \texttt{continue} $\leftarrow$ False
    \EndIf
    \State $k \leftarrow k + 1$
    \EndWhile
    \State \Return $\vec{\pi}^* = \vec{\pi}_{k+1}$\Comment{Returns the optimal policy}
    \EndProcedure
    \end{algorithmic}
\end{algorithm}

\section{Continuous-time Markov Decision Process}\label{section: CTMDP}

This paper intends to \emph{test} the hypothesis that retaining some event processes as non-memoryless in a DEDS would lead to a model of the polling system that produces a better performing policy. To do this, an accompanying memoryless model must be available --- a CTMDP. Such a CTMDP can be obtained using the procedures of section~\ref{section: SMDP} whereby $\forall e \in \mathcal{E}: \, e \sim Exp(\cdot|\theta_e)$. While such an approach is modular in the sense that code written for the SMDP will solve for a CTMDP by simply changing the input of event distributions, it is not very efficient. 

A CTMDP can be converted into a DTMDP via uniformisation and the solved for using the standard Bellman equations. This would circumvent the use of the bi-variate generator matrix and all integral involving it. In fact, the uniformisation approach would not perform a single integral.

In solving polling models as a CTMDP, the model has typically relied on the \emph{preemptive} server discipline. Evidence of this can be found in chapter 9.5.3 of \cite{cassandras_book}, example 1.4.2 of \cite{BertsekasVol2} and the following two papers \cite{koole1997assigning,suk1991optimal}. A \emph{non-preemptive} server discipline can be found in \cite{moustafa1996optimal} whereas both server disciplines are treated in \cite{fernandesscheduling}. The paper that \emph{solves} for the optimal policy of a non-preemptive two queue polling model with switch-over durations as a CTMDP is \cite{duenyas_van_oyen_2000_finite_buffer_heuristic}. However, it does not explicitly show how to \emph{model} the CTMDP whereas most of the other papers do. It only recommends that the state of the server be tracked in the state-space as to enforce the non-preemptive condition of the service and switch-over events. 

Modelling the preemptive version is simpler than its non-preemptive counterpart due to the fact that the state of the server need not be tracked. Just as for the SMDP of section~\ref{section: SMDP assumptions}, decision epochs can only take place at the instant that one of the serial events $\mathcal{E}_{\perp}$ have triggered\footnote{In contrast to the SMDP where only decision epochs exhibited the Markov property, the CTMDP retains the Markov property even during an ongoing action.}. This ensures the server to be free $l_2=0$ when making a decision such that no ongoing process is interrupted. This section will develop uniformised CTMDP models for both server disciplines as to illustrate some of the caveats and awkward features of the non-preemptive version. The reader content with using the SMDP machinery instead of uniformisation can skip this section and move onto section~\ref{section: experiments}.

\subsection{Preemptive server discipline}\label{section: preemptive ctmdp}

A global sampling rate will be selected as $\gamma = \lambda_1 + \lambda_2 + \max\left\{ \mu_1,\mu_2,s_{1,2},s_{2,1}  \right\}$ where $\mu_i$ and $s_{i,j}$ are the \emph{rate} parameters\footnote{Another popular parameter is the \emph{scale} parameter which is the inverse of the rate. The scale plays an important role in the characterising \emph{family of exponential distributions} also known as the \emph{Koopman–Darmois family} (see section 4.9 of \cite{pawitan_book}).} of an exponential distribution just like $\lambda_i$. This is the most conservative sampling rate. Some texts use $\gamma  = \lambda_1 + \lambda_2 +  \mu_1+ \mu_2+ s_{1,2}+ s_{2,1}$ \cite{koole1997assigning} which results in sampling the system more frequently and observing more self transitions. Both approaches yield the same optimal policy. 

The cost model computes the holding cost of all existing customers until the system is sampled again
\begin{eqnarray}
    C_{n_1,n_2}^{\beta,e} & = &  \left( \frac{c_1}{\gamma + \beta} \right)n_1 + \left( \frac{c_2}{\gamma + \beta} \right)n_2 \label{eq: ctmdp holding cost} \\
    & = & \bar{c}_1 n_1  + \bar{c}_2 n_2
\end{eqnarray}
such that $\vec{C}_{l_2}$ can be constructed as a $|\mathcal{X}|$-length vector with $C_{n_1,n_2}^{\beta,e}$ inserted at index $i = \mbox{\texttt{triple2idx}}(n_1,n_2,l_1)$. Note that $\forall a \in \mathcal{A}: \, \vec{C}_{l_2=a}  =\vec{C}$ such that $\mathcal{C} = \{ \vec{C} \}$ consists of only a single cost vector. Hence, the cost-per-stage is invariant to the decision. In constructing the transition model, $\mathbf{P}_{l_2}$ can be constructed directly without using a vector of arrival probabilities $\vec{P}_{l_2}$. Furthermore, no discounted version needs to accompany it as a fixed discount factor is used $\alpha = \gamma/(\gamma + \beta)$. The transition kernel for a service at queue $l_1^k = l_1^{k+1}=i$ is defined as
\begin{equation}
    \mathscr{P}(x_{k+1}\mid x_k,e=\mu_i) = 
    \begin{cases}
    \frac{\mu_i}{\gamma}, & n_i^{k+1} = n_i^k - 1, n_j^{k+1} = n_j^k, i \neq j \\
    \frac{\lambda_f}{\gamma}, & n_f^{k+1} = n_f^k + 1, n_j^{k+1} = n_j^k, f \neq j\\
    1 - \frac{\mu_1 - \lambda_1 - \lambda_2}{\gamma},&  n_i^{k+1} = n_i^k , n_j^{k+1} = n_j^k, i \neq j
    \end{cases}
\end{equation}
where $x_k = (n_1^k,n_2^k,l_1^k)$. A switch-over from server $l_1^k=i \to l_1^{k+1}=j $ specifies
\begin{equation}
    \mathscr{P}(x_{k+1}\mid x_{k}, e=s_{i,j}) = 
    \begin{cases}
    \frac{s_{i,j}}{\gamma}, & l_1^k = i, l_1^{k+1} = j,n_i^{k+1} = n_i^k , n_j^{k+1} = n_j^k, i \neq j \\
    \frac{\lambda_f}{\gamma}, & l_1^k = l_1^{k+1}=i, n_f^{k+1} = n_f^k + 1, n_q^{k+1} = n_q^k, f \neq q\\
    1 - \frac{s_{i,j} - \lambda_1 - \lambda_2}{\gamma},& l_1^k = l_1^{k+1}=i, n_i^{k+1} = n_i^k , n_j^{k+1} = n_j^k, i \neq j
    \end{cases}
\end{equation}
while an idling decision at queue $l_1^k = l_1^{k+1}=i$ dictates that
\begin{equation}\label{eq: ilding transition kernel}
    \mathscr{P}(x_{k+1} \mid x_j , e=\Lambda_i) = 
    \begin{cases}
    \frac{\lambda_f}{\gamma}, & n_f^{k+1} = n_f^k + 1, n_q^{k+1} = n_q^k, f \neq q\\
    1 - \frac{\lambda_1 - \lambda_2}{\gamma},&  n_i^{k+1} = n_i^k , n_j^{k+1} = n_j^k, i \neq j.
    \end{cases}
\end{equation}
These three kernels are each used to construct a $|\mathcal{X}| \times |\mathcal{X}|$ matrix $\mathbf{P}_{l_2}$ as to obtain the desired set $\mathcal{P} = \{ \mathbf{P}_{a}: a \in \mathcal{A} \}$. Pooling is used to perform truncation as in section~\ref{section: tpm truncation}. Furthermore, the same caveat applies in that when $\mu_i \not \in \mathcal{A}(x_k)$ then the row of $\mathbf{P}_{\mu_i}$ corresponding to $x_k$ is trivial. Lastly, all states are decision making states such that $\mathcal{D} = \mathcal{X}$.

\subsection{Non-preemptive server discipline}\label{section: non-preemptive ctmdp}

\subsubsection{Assumptions and constraints}

The non-preemptive server discipline needs to take both $l_1$ and $l_2$ into account hence $x_k = (n_1^k,n_2^k,l_1^k,l_2^k) \in \mathcal{X} = \mathcal{X}_1 \times \mathcal{X}_2 \times \{0,1 \} \times \{ 0,1,2 \}$ where $\mathcal{X}_i = \mathbb{N}\cap [0,N_i]$ and $N_i \in \mathbb{N}\cap [1,\infty)$. The index functions are defined using $M_1 =  6X_2$, $M_2= 6$ and $M_3 = 3$ such that $\mbox{\texttt{quad2idx}}(n_1,n_2,l_1,l_2) = n_1 M_1 + n_2 M_2 + l_1 M_3 + l_2$. For some index $i$ the following relations are defined $n_1 = \lfloor i/M_1 \rfloor$, $R_1 = (i \mod M_1$), $n_2 = \lfloor R_1/M_2 \rfloor$, $R_2 = (R_1\mod M_2$), $l_1 = \lfloor R_2/M_3 \rfloor$ and $l_2 = (R_2 \mod M_3)$ such that $\mbox{\texttt{idx2quad}}(i) =(n_1,n_2,l_1,l_2)$. 

Only states where $l_2 = 0$ are decision making states such that $\mathcal{D} = \{ (n_1,n_2,l_1,0): n_i \in \mathcal{X}_i, l_1 \in \{ 0,1 \} \}$. Once a decision is made, the state of the server is assigned a new value $l_2 \in \{0,1,2  \}$. Naturally, the decision to idle $\Lambda_i$ does not change $l_2$. Furthermore, if a $l_2$ is changed as $l_2^k=0$ to $l_2^{k+1} \neq 0$ then an instantaneous transition takes place with probability one, zero cost incurred and no discounting performed ($\alpha=1$). These will be referred to as \emph{linking transitions}. If a decision is made to idle $l_2^k = l_2^{k+1} = 0$ then the non-preemptive requirement enforces the next states to only be those corresponding to arrival occurrences. Self transitions would violate the non-preemptive constraint. Transitions to these states are probabilistic, incur a holding cost and are subject to discounting. Transition from an idling state are not linking transitions but of the usual CTMDP variant. However, the will be referred to as \emph{idling transitions}.

Linking transitions serve the purpose of acting as a "circuit switch" in the transition graph of the CTMDP. That is, they connect the desired regions responsible for the stochastic dynamics of the CTMDP. These regions are only connected at node/states where a decision event has completed $l_2=0$. This is illustrated in figure~\ref{fig: non-preemptive CTMDP} for all three actions executed at a free server.

\begin{figure}[!htbp]
    \centering
    \includegraphics[width=0.85\textwidth]{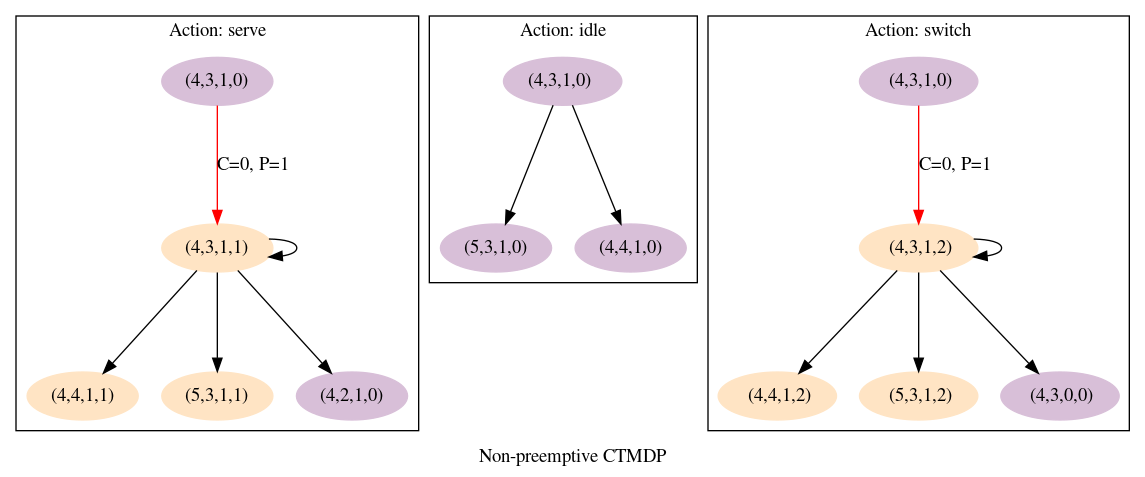}
    \caption{Decision states $x_k \in \mathcal{D}$ (thistle) and dynamics states $x_k \in \mathcal{D}'$ (bisque) with red arrows as decisions and black arrows as probabilistic transitions.}
    \label{fig: non-preemptive CTMDP}
\end{figure}

Apart from idling, the dynamics of the system is governed by $x_k \in \mathcal{D}' = \mathcal{X} \setminus \mathcal{D}$ in that these states allow for real holding times, incur holding costs and are subject to discounting. As mentioned, dynamics states are connected to via the decision states through a linking transition $x_k=i \xrightarrow[]{} \bar{x}_k \in \mathcal{D} \xrightarrow[]{\pi(i) \in \{1,2\}} x_{k+1}\in \mathcal{D}'=j$ where $\xrightarrow[]{\pi(i)\in \{1,2\}}$ represents a linking transition and $\xrightarrow[]{}$ a CTMDP transition. Linking transitions will only connect a decision state to a dynamics state. In contrast, idling transitions will connect decision states to decision states  $x_k=i \xrightarrow[]{} \bar{x}_k \in \mathcal{D} \xrightarrow[]{\pi(i) =1} x_{k+1}\in \mathcal{D}=j$. In such a scenario, prolonged idling can build complexes or clusters of decision states that might serve as an intermediate between dynamics states. Note that the dynamics states do not contain idling. However, idling could be seen as a hybrid between decision and dynamics states.

\subsubsection{Components and matrix model}

In respect to the MDP model, only the rows in $\mathbf{P}_{\pi}$ and entries of $\vec{C}_{\pi}$ corresponding to index $i = \mbox{\texttt{quad2idx}}(x_k), \, \forall x \in \mathcal{D}$ are modified under the non-preemptive policy $\pi: \mathcal{D} \to \mathcal{A}$. The rest of the model remains fixed. More specifically, the cost model can be initialised as a $|X|$-length vector of zeros and the transition model as a $|X|\times |X|$ matrix of zeros. The cost model is filled in at entry $i$ using $C_{n_1,n_2}^{\beta,e}$ as defined in (\ref{eq: ctmdp holding cost}) while the $i^{th}$ row of $\mathbf{P}_{\pi}$ is constructed according to an appropriate transition kernel (depends on the action in progress) where $i = \mbox{\texttt{quad2idx}}(x_k), \, \forall x_k \in \mathcal{D}'$. The service transition kernel requires that the server be at the correct location $l_1^{k}=l_1^{k+1} = i$ and in use $l_2^{k} = 1$.
\begin{equation}
    \mathscr{P}(x_{k+1}\mid x_k,e=\mu_i) = 
    \begin{cases}
    \frac{\mu_i}{\gamma}, & l_2^{k+1}=0,n_i^{k+1} = n_i^k - 1, n_j^{k+1} = n_j^k, i \neq j \\
    \frac{\lambda_f}{\gamma}, & l_2^{k+1}=1, n_f^{k+1} = n_f^k + 1, n_j^{k+1} = n_j^k, f \neq j\\
    1 - \frac{\mu_1 - \lambda_1 - \lambda_2}{\gamma},& l_2^{k+1}=1, n_i^{k+1} = n_i^k , n_j^{k+1} = n_j^k, i \neq j
    \end{cases}
\end{equation}
While a switch-over is in progress from server $i \to j $ the following is given
\begin{equation}
    \mathscr{P}(x_{k+1}\mid x_{k}, e=s_{i,j}) = 
    \begin{cases}
    \frac{s_{i,j}}{\gamma}, & l_1^k = i, l_1^{k+1} = j,l_2^{k+1}=0,n_i^{k+1} = n_i^k , n_j^{k+1} = n_j^k, i \neq j \\
    \frac{\lambda_f}{\gamma}, & l_1^k = l_1^{k+1}=i,l_2^{k+1}=1, n_f^{k+1} = n_f^k + 1, n_q^{k+1} = n_q^k, f \neq q\\
    1 - \frac{s_{i,j} - \lambda_1 - \lambda_2}{\gamma},& l_1^k = l_1^{k+1}=i,l_2^{k+1}=1, n_i^{k+1} = n_i^k , n_j^{k+1} = n_j^k, i \neq j.
    \end{cases}
\end{equation}
This completes the specification of the dynamics model. Note that the entire dynamics model uses the same sampling rate $\gamma = \lambda_1 + \lambda_2 + \max\left\{ \mu_1,\mu_2,s_{1,2},s_{2,1}  \right\}$ as in section~\ref{section: preemptive ctmdp} and hence the same discount factor $\alpha = \gamma/(\gamma+\beta)$.

Given a $|\mathcal{D}|$-length policy $\vec{\pi}$, the cost and transition models can be modified through selecting the $i^{th}$ component/row from the set object corresponding to $a=\vec{\pi}_i$ where $i = \mbox{\texttt{quad2idx}}(x_k), \, \forall x_k \in \mathcal{D}$. The aforementioned sets are $ \mathcal{C} = \{\vec{C}_a:a \in \mathcal{A}  \} $ and $\mathcal{P} = \{ \mathbf{P}_{a}: a \in \mathcal{A} \}$. For the service and switch-over decisions, the kernels describe linking transitions
\begin{equation}
    \mathscr{P}(x_{k+1}\mid x_k , a=\{\mu_i,s_{i,j}\}) = 
    \begin{cases}
    1, & l_2^k = 0 ,  l_2^{k+1}=1, l_1^k=l_1^{k_1}=i, n_f^{k+1} = n_f^k, n_q^{k+1} = n_q^k, f \neq q \\
    0, & \mbox{else}
    \end{cases}
\end{equation}
where $x_{k+1} \in \mathcal{D}'$. Non-preemptive idling requires a sampling rate $\gamma_\Lambda = \lambda_1 + \lambda_2$ such that 
\begin{equation}
    \mathscr{P}(x_{k+1}\mid x_k , a=\Lambda_i) = \begin{cases}
    \frac{\lambda_f}{\gamma_{\Lambda}}, & l_2^{k}=l_2^{k+1}=0, l_1^k=l_1^{k_1}=i, n_f^{k+1} = n_f^k+1, n_q^{k+1} = n_q^k, f \neq q\\
    0, & \mbox{else}
    \end{cases}
\end{equation}
where $x_k,x_{k+1} \in \mathcal{D}$. As for the cost model, $\vec{C}_a,\,  a = 1,2$ are $|\mathcal{X}|$-length vectors of zeros while $\vec{C}_{l_2=0}$ is similarly defined but with entries at  $i = \mbox{\texttt{quad2idx}}(x_k), \, \forall x_k \in \mathcal{D}$ determined through 
\begin{equation}
    C_{n_1,n_2}^{\beta,\Lambda} = \left( \frac{c_1}{\beta + \gamma_{\Lambda}} \right)n_1 + \left( \frac{c_2}{\beta + \gamma_{\Lambda}} \right)n_2. \label{eq: non-preemptive ctmdp holding costs}
\end{equation}
Lastly, the Bellman equations will require state-action-dependent discount factors

\begin{subnumcases}{\alpha_{\pi}(x_k)=}
\frac{\gamma}{\gamma + \beta}, & $x_k \in \mathcal{D}'$\label{eq: discount non-decision state} \\
\frac{\gamma}{\gamma + \beta}, & $x_k \in \mathcal{D} , \pi_k(x_k) = 0$ \label{eq: discount idling}\\
1 , & $x_k \in \mathcal{D} , \pi_k(x_k) \in  \{ 1,2 \}$.\label{eq: discount serve and switch}
\end{subnumcases}

To deal with state-action dependent discount factors, a $|\mathcal{X}|$-length vector $\vec{\alpha}_{\pi}$ needs to be maintained. Similar to $\vec{C}_{\pi}$, it can be initialised by filling in the entries that will remain fixed: allocate (\ref{eq: discount non-decision state}) to entry $i = \mbox{\texttt{quad2idx}}(x_k), \, \forall x_k \in \mathcal{D}'$. The other entries will be modified in accordance with the policy by selecting from $\mathbb{A} = \{ \vec{\alpha}_a : a \in \mathcal{A}\}$ where $\vec{\alpha}_a$ is a zero vector with entries at $i = \mbox{\texttt{quad2idx}}(x_k), \, \forall x_k \in \mathcal{D}$. Additionally, $\vec{\alpha}_0$ uses (\ref{eq: discount idling}) while $\vec{\alpha}_1$ and $\vec{\alpha}_2$ rely on (\ref{eq: discount serve and switch}). Provisioning $\mathcal{C}$, $\mathcal{P}$ and $\mathbb{A}$ permits algorithm~\ref{algorithm: smdp policy iteration} to be modified for solving the CTMDP of this section.

\begin{algorithm}[!htbp]
    \caption{CTMDP Policy Iteration}
    \label{algorithm: ctmdp policy iteration}
    \begin{algorithmic}[1]
    \Procedure{PolicyIteration}{$\mathcal{C}$,$\mathcal{P}$,$\mathbb{A}$,$\vec{\pi}_0$,\texttt{maxiter}}
    \State Initialise $\vec{C}_{\pi}$, $\vec{\alpha}_\pi$ and $\mathbf{P}_{\pi}$ by filling in entries pertaining to $\mathcal{D}'$ \Comment{New}
    \If {$\vec{\pi}_0$ \textbf{is} None}
        \State Randomly initialise $\vec{\pi}_0 = [a \sim \mathcal{A}(i)],\quad \forall i \in \mathcal{X}$
    \EndIf
    \State $k \leftarrow 0$
    \State \texttt{continue} $\leftarrow$ True
    \While {\texttt{continue} \textbf{is} True}
    \State Modify $\vec{C}_{\pi_k}$, $\vec{\alpha}_{\pi_k}$ and $\mathbf{P}_{\pi_k}$ using $\mathcal{C}$, $\mathbb{A}$ and $\mathcal{P}$ under instructions from $\vec{\pi}_k$ at entries concerning $\mathcal{D}$. 
    \State $\bm{\alpha}_{\pi_k} \leftarrow \operatorname{diag}\left(  \vec{\alpha}_{\pi_k}\right)$ \Comment{New}
    \State \textit{Policy evaluation}: solve for $\vec{J}_{\pi_k}$ in \Comment{Updated}
    \begin{equation}\label{eq: policy evluation solve for J alg}
        \left( \mathbf{I} - \bm{\alpha}_{\pi_k}\mathbf{P}_{\pi_k} \right)\vec{J}_{\pi_k}  = \vec{C}_{\pi_k}.
    \end{equation}
    \State \textit{Policy Improvement}: create an empty vector $\vec{\pi}_{k+1}$ of length $|\mathcal{X}|$. 
    \State $\Delta_k \leftarrow$ False
    \For {$i = 1,2,\cdots, |\mathcal{X}|$}
        \State Select the optimal action from a greedy one-step look-ahead optimisation. \Comment{Updated}
        \begin{equation}
            \vec{\pi}_{k+1}(i) = \mbox{argmin}_{a \in \mathcal{A}(i)} \left\{ \vec{C}_{a}(i) + \sum_{j=1}^{|\mathcal{X}|} \vec{\alpha}_{a}(i) \mathbf{P}_{a}(i,j) \vec{J}_{\pi_k}(j)   \right\} 
        \end{equation}
        \If {$\vec{\pi}_{k+1}(i) \neq \vec{\pi}_{k}(i)$ \textbf{and} $\Delta_k$ \textbf{is} False}
        \State $\Delta_k \leftarrow $ True
        \EndIf
    \EndFor
    \If {$\Delta_k$ \textbf{is} False \textbf{or} $k =$ \texttt{maxiter}}
        \State \texttt{continue} $\leftarrow$ False
    \EndIf
    \State $k \leftarrow k + 1$
    \EndWhile
    \State \Return $\vec{\pi}^* = \vec{\pi}_{k+1}$\Comment{Returns the optimal policy}
    \EndProcedure
    \end{algorithmic}
\end{algorithm}

\subsubsection{Sparsity aware model}

Algorithm~\ref{algorithm: ctmdp policy iteration} is attractive if it can be used with a commercial linear algebra package that has a good solver for a system of linear equations. The non-preemptive CTMDP has a remarkably sparse transition model with each row containing at most four entries. Some linear algebra software has additional routines for dealing with sparsity. However, computational efficiency can be obtained by using \emph{value iteration} over the computational graph of the MDP. The computational graph relies on \emph{state-action values} $Q(x_k,a_k)$ in conjunction with \emph{state value}s $J(x_k)$ which are related as below
\begin{eqnarray}
    J(x_k) & =& \operatorname{argmin}_{a_k \in \mathcal{A}(x_k)} \left\{ Q(x_k,a_k) \right\} \label{eq: J from Q}\\
    Q(x_k,a_k) & = & C(x_k) + \alpha(x_k,a_k) \sum_{x_{k+1} \in \mathscr{N}(x_k)}  \mathscr{P}(x_{k+1} \mid x_k, a_k) J(x_{k+1}) \label{eq: Q from J}
\end{eqnarray}
where $\mathcal{N}(x_k)$ is the set of all states that can be reached from $x_k$ and is called the \emph{neighbours set}. This is analogous to all states in a row of the transition model that have non-zero entries. This relationship is depicted in figure~\ref{fig: mdp state-value graph} as a state-value computational graph. It reflects the fact that $J(x_k)$ can be considered as a node (state node) that contains $Q(x_k,a_k)$ as embedded nodes (state-action nodes) which it derives its value from. Furthermore, the greedy one-step look-ahead optimal action (from policy improvement in algorithms \ref{algorithm: smdp policy iteration} and \ref{algorithm: ctmdp policy iteration}) can be obtained through consulting the embedded nodes of the state node once all embedded nodes have received a \emph{Bellman update} (\ref{eq: Q from J}). After this update, the state node should also have its value updated as to reflect the change. Model components such as the holding cost, discount factor, transition model and neighbours set can be stored on each state-action node as attributes. Only the transition probabilities that correspond to states in the neighbours set need be stored.

The central feature of the sparsity-exploiting approach is that all state-action nodes are accessible by looping through a global object that stores \emph{pointers} or \emph{memory locations} of these nodes $\mathbb{Q} = \{Q(x_k,a_k)\forall x_k,a_k \in \mathcal{X}\times \mathcal{A}  \}$. Upon accessing the object, a Bellman update can be performed which involves a summation loop over at most four neighbours. Choosing when to perform the update of the state node (\ref{eq: J from Q}) determines whether \emph{asynchronous}\footnote{This is also referred to as \emph{Gauss-Siedel} value iteration as in chapter 5.6 of \cite{gosavi2015simulation}} or \emph{synchronous} value iteration is being performed. By updating $J(x_k)$ once all its embedded nodes have received its Bellman update results in the asynchronous variant. It would be wise to order state-action nodes in $\mathbb{Q}$ such that nodes belonging to the same state are contiguous. This can be performed by a function \texttt{ContiguousArrange}\footnote{The loop can have it performance enhanced if the entries of $\mathbb{Q}$ are contiguously stored in computer memory.}. The synchronous variant would loop through $\mathbb{J} = \{J(x_k): x_k \in \mathcal{X}  \}$ and perform (\ref{eq: J from Q}) after the loop though $\mathbb{Q}$ has completed. Both approaches allow for the state-values to converge to their optima in the limit (\ref{eq:converge}), however, the asynchronous variant has generally been observed to do this at a faster rate (chapter 5.6 of \cite{gosavi2015simulation}).

\begin{figure}[!htbp]
    \centering
    \includegraphics[width=0.3\textwidth]{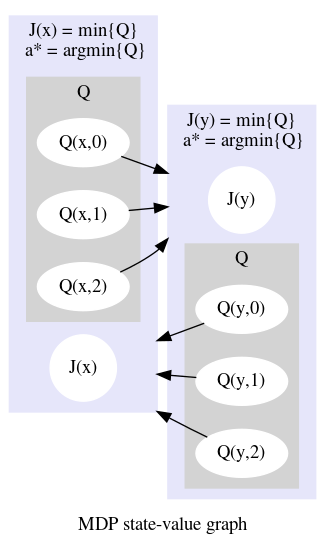}
    \caption{MDP state-value computational graph for value iteration.}
    \label{fig: mdp state-value graph}
\end{figure}

\begin{algorithm}[!htbp]
    \caption{CTMDP sparsity-aware asynchronous Value Iteration}
    \label{algorithm: ctmdp value iteration}
    \begin{algorithmic}[1]
    \Procedure{ValueIteration}{$\mathbb{Q}$,$\mathbb{J}$,$\epsilon$,\texttt{maxiter}}
    \State $\mathbb{Q} \leftarrow \mbox{\texttt{ContiguousArrange}}(\mathbb{Q})$ \Comment{Important for the asynchronous update}
    \For{$J(x) \in \mathbb{J}$}
    \State $J(x) \leftarrow 0$ \Comment{Most simple arbitrary initial value}
    \EndFor
    \State $\vec{J}_0 \leftarrow \mbox{\texttt{set2vector}}(\mathbb{J})$
    \State $\Delta_0 \leftarrow \infty$ \Comment{Largest change in value functions}
    \State $k \leftarrow 0$
    \While {$\Delta_k > \epsilon$}
        \State $X \leftarrow $ None  \Comment{Current state of $Q$-values}
        \State $n \leftarrow $ None \Comment{Number of $Q$-values left}
        \For {$Q(x,a) \in \mathbb{Q}$}
            \If{$x \neq X$}
                \State $X \leftarrow x$
                \State $n \leftarrow |\mathcal{A}(x)|$
            \EndIf
            \State Perform Bellman update (\ref{eq: Q from J}).
            \State $n \leftarrow n - 1$
            \If{$n = 0$}
                \State Update $J(x_k)$ using (\ref{eq: J from Q}). \Comment{Asynchronous update}
            \EndIf
        \EndFor
        \State $\vec{J}_{k+1} \leftarrow \mbox{\texttt{set2vector}}(\mathbb{J})$
        \State $\Delta_k = \max\left\{ \left\vert \vec{J}_{k+1} \ominus \vec{J}_{k} \right\vert \right\}$ \Comment{$\ominus$ is element-wise subtraction}
    \EndWhile
    \State \Return $\vec{\pi}^* \leftarrow \mbox{\texttt{OptimalActions}}(\mathbb{J})$ \Comment{Returns the optimal policy}
    \EndProcedure
    \end{algorithmic}
\end{algorithm}

In algorithm~\ref{algorithm: ctmdp value iteration} two helper functions have been used: \texttt{set2vector} and \texttt{OptimalActions}. Both functions return a $|\mathcal{X}|$-length vector by looping through $J(x) \in \mathbb{J}$ and filling the vector at index $i = \mbox{\texttt{quad2idx}}(x)$. The first function places $J(x)$ at this index while the second consults its embedded nodes such that the optimal action $a^* = \operatorname{argmin}_{a \in \mathcal{A}(x)}\{Q(x,a)\}$ is inserted.

\section{Discussion}\label{section: discussion}

\subsection{Homogeneous and non-homogeneous arrival rates}\label{section: homogeneous arrival rates}

In building the components for the SMDP model of section~\ref{section: SMDP}, a very general approach was taken. Although not explicitly stated, the bi-variate generator matrix modelling arrivals in equation~\ref{eq: bi-variate generator matrix} can be used along with non-homogeneous arrival rates. More specifically, the arrival rates can depend on its position in the generator matrix such that (\ref{eq: bi-variate generator matrix}) can be substituted for the following
\begin{equation}\label{eq: non-homogeneous bi-variate generator matrix}
    \mathbf{Q}_{i,j} = 
    \begin{cases}
    \lambda_1(n_{\lambda_1},n_{\lambda_2}), & i = \mbox{\texttt{coord2idx}}(n_{\lambda_1},n_{\lambda_2}), j = \mbox{\texttt{coord2idx}}(n_{\lambda_1}+1,n_{\lambda_2})\\
    \lambda_2(n_{\lambda_1},n_{\lambda_2}), & i = \mbox{\texttt{coord2idx}}(n_{\lambda_1},n_{\lambda_2}), j = \mbox{\texttt{coord2idx}}(n_{\lambda_1},n_{\lambda_2}+1)\\
    -\gamma(n_{\lambda_1},n_{\lambda_2}), & i = \mbox{\texttt{coord2idx}}(n_{\lambda_1},n_{\lambda_2}), j = \mbox{\texttt{coord2idx}}(n_{\lambda_1},n_{\lambda_2}) \\
    0, & \mbox{else}
    \end{cases}
\end{equation}
where $\gamma(n_{\lambda_1},n_{\lambda_2}) = \lambda_1(n_{\lambda_1},n_{\lambda_2}) + \lambda_2(n_{\lambda_1},n_{\lambda_2})$. Such a modification presents no change to the existing framework of section~\ref{section: SMDP} due to the general approach taken. However, such generality comes at the cost of forgoing model simplification that allows for faster computation.

If the arrival rates are homogeneous then the very expensive computation for holding costs due to arrivals (\ref{eq: arrival cost}) can be simplified. This is because the total holding cost up to time $t$ due to arrivals is a linear function of time
\begin{equation}
    c_{\mathcal{I}}(t) = (c_1 \lambda_1 + c_2 \lambda_2) \times t
\end{equation}
such that (\ref{eq: arrival cost summed}) can be replaced with
\begin{eqnarray}
    C_{\mathcal{I}}^e &=& \int_{0}^{\infty} c_{\mathcal{I}}(t) f_e(t) \, dt \\
    & = & (c_1 \lambda_1 + c_2 \lambda_2) \times \int_{0}^{\infty}t f_e(t) \, dt \\
    & = & (c_1 \lambda_1 + c_2 \lambda_2) \times \mathbbm{E}\left[ t_e \right] \label{eq: homogeneous arrival costs}
\end{eqnarray}
where $\mathbbm{E}\left[ t_e \right]$ is the expected duration of the event. This is fast to compute and does not suffer from any truncation error introduced by the finite size generator matrix. The use of non-homogeneous arrival rates can be useful though due to its additional modelling power. As an example, one might consider the polling model to be an online server attending to two types of queued customer requests (computational jobs). When a server processes queue $i$, customers gradually become aware of this and submit jobs to queue $i$ instead of $j\neq i$ in hope that the selected queue will be served exhaustively. Hence, arrival rates increase which can be modelled through using the non-homogeneous approach.

\subsection{SMDP vs CTMDP: similarities and differences}\label{section: SMDP vs CTMDP differences}

Before presenting and comparing the SMDP and CTMDP policies as well as gauging their relative performances as in section~\ref{section: experiments}, it is worthwhile to discuss some similarities and differences in the components of the models. This will help understand some of the outcomes. The objective function used by both MDPs is the long-run expected discounted holding cost. This is an additive function which allows for the \emph{principle of optimality} to apply as to permit a dynamic programming solution \cite{BertsekasVol1,BertsekasVol2}. This additive function is constructed from the expected discounted cost incurred between decision epochs. When the arrivals rates during such an interval is constant/homogeneous then only the expected length of this interval is important. For the SMDP this is clear from (\ref{eq: homogeneous arrival costs}) used in conjunction with (\ref{eq: lump sum cost final}). For the CTMDP this is clear from (\ref{eq: non-preemptive ctmdp holding costs}). However, (\ref{eq: homogeneous arrival costs}) can be used for \emph{both} the non-preemptive CTMDP and SMDP models as the entire framework of section~\ref{section: SMDP} can build a non-preemptive CTMDP model by specifying $\forall e \in \mathcal{E}: t_e \sim Exp$. This means that the shape/spread of the event-duration distribution over $\mathbb{R}_{\geq 0}$ does not play a role in determining the outcome of the cost. Subsequently, attempting to introduce risk through assigning long-tailed distributions to $t_e$ will not affect the cost per stage in a SMDP and allow it to differ from the CTMDP in this aspect.

The shape/spread of the event-duration distribution does, however, affect the transition model. This is where the SMDP can differ from the CTMDP. Both models will have the same transition model for idling events. However, the SMDP model can be more (or less) pessimistic with regards to arrivals received during a service or switch-over event. This is important as it assigns a larger probability of transitioning to a less favourable future state as compared to the CTMDP. Hence, long-tailed distributions can implicitly model some form of risk adverse behaviour into the SMDP. 

Risk is usually modelled explicitly in the objective function of an MDP (see section 3.1 of \cite{fu_risk_MDP_tut}) through various quantities. Variance-penalised MDPs \cite{gosavi_variance_MDP} have been popular as well as the use of downside-risk \cite{gosavi_downside_risk}. Downside-risk is useful as it distinguishes between "good" and "bad" spread of the cost distribution. Downside-risk, as employed in modelling rewards, characterises "good" spread as that below the mean of the distribution. Variance-penalised risk fails to distinguish between good and bad variance. For a model that relies on costs, upside-risk naturally stems from downside-risk. Long-tailed distributions could be seen to implicitly model such upside-risk in the transition functions.

\section{Empirical Experiments}\label{section: experiments}

The main concern of this section is to test the hypothesis that a Semi-Markov decision process provides an advantage over existing policies when controlling a polling model with switch-over durations. This is a two-part process:
\begin{enumerate}
    \item If for another policy $\pi$ one finds that $\forall x = (n_1,n_2,l_1) \in \mathcal{X}: \pi_{SMDP}(x) = \pi(x)$ then they are identical and the Semi-Markov Decision Process cannot possibly provide an advantage. In this scenario, the SMDP policy is likely to be the more computationally burdensome policy to compute.
    \item If the two policies differ, the expected performance advantage should be better by a margin that is statistically significant.
\end{enumerate}

This section considers various parameter and distribution combinations for the two-queue polling models and for each seeks to reject the hypothesis that the Semi-Markov Decision Process policy is the same. Hence, this section does not \emph{prove} any results nor does it seek to \emph{accept} a hypothesis.

\subsection{Method}

To empirically assess the Semi-Markov Decision policies the following is required:
\begin{enumerate}
    \item Polling models with stable dynamics. Attempting to control an unstable system is pointless.
    \item Polling models with defining characteristics. Certain policies may perform better on polling models with certain attributes.
    \item Other policies polling model policies. The Continuous-time Markov Decision process has already been discussed in section~\ref{section: CTMDP}. It has been reported to provide the best performing policy on the long-run discounted holding cost \cite{duenyas1996heuristic}. Computationally inexpensive heuristic policies designed for polling models with switch-over costs will also be considered.
    \item Means of performing reliable, consistent and comparable simulations of a polling model under a given control policy. 
    \item The infinite horizon discounted holding costs must be obtained from finite length simulations.
    \item A test must be performed to assess whether the infinite horizon discounted holding costs are different in a way that is statistically significant i.e. not attributed to randomness.
\end{enumerate}

The following subsections discuss how these requirements are met.

\subsubsection{Generating polling model scenarios}\label{method: polling model scenarios}

A \emph{necessary} condition to ensure that a two queue polling model has stable dynamics is given below.

\begin{definition}[\textbf{Necessary stability criteria} \cite{duenyas1996heuristic}] \label{def: statbility criteria}
A two-queue polling model with non-zero switch-over durations can have stable dynamics if
\begin{eqnarray}
    \rho_1 + \rho_2 & < &  1 \label{eq: necessary stability}
\end{eqnarray}
where the utilisation is computed as
\begin{eqnarray}
    \rho_i & = & \frac{\lambda_i}{\mu_i} \label{eq: rho_i} \\
    & = & \frac{\lambda_i}{\int_{0}^\infty t\, f_{\mu_i}(t) \, dt}\\
    & = & \lambda_i \bar{t}_{\mu_i}
\end{eqnarray}
and $\rho = \rho_1 + \rho_2$ for convenience.
\end{definition}
Essentially, condition~(\ref{eq: necessary stability}) asks that the server needs to be performing work for only a $\rho$ proportion of time. In other words, a server that must always be busy serving jobs in order to remain stable will never be able to switch such that at least one of the queues will grow without bound. For example, if $\rho=1$\footnote{A system with no switch-over can still be stable if $\rho=1$. This is because the server can always be busy as switches are instantaneous.} then a server can remain at one queue as to ensure it remains stable while the other neglected queue grows. Alternatively, the server can perform switches that it cannot afford. In this case, it arrives to a queue that was on average longer than its last arrival to it.

The condition is not sufficient as a policy can still make the system unstable. This can happen if the one queue has a much larger holding costs rate than the other $c_i \ggg c_j$ which leads to a policy that is optimal to serve queue $i$ and idle at it until its next arrival hence never switching to queue $j$\footnote{This can be explained using rewards rates instead of costs as in \cite{duenyas1996heuristic}. Policies often try to maximise a reward rate that is equivalent to minimising a cost function; an idea presented in \cite{harrison1975priority}. Hence, the server idles at a queue $i$ if a very high reward is obtained from serving another of its customers.}. A high discount rate with less server cost rate difference can lead to the same scenario as the policy becomes severely short-sighted/myopic. Nonetheless, condition~(\ref{eq: necessary stability}) is a good starting point to test whether a proposed polling model is stable \emph{before} controlling it.

The preceding paragraph discusses a potential role that the cost rate plays in influencing what a policy might do to the polling system. Before allocating a large chunk of computational resources to obtain a Semi-Markov Decision Process policy, it would be desirable to have an idea of what it might look like. Such a ballpark sketch of the policy will avoid computing a SMDP policy that turns out to be optimal but induces instability. To do this, the theory of chapter 5 from \cite{van_eekelen} is used. The idea is to assess the optimal limit-cycle of a polling system modelled as deterministic fluid queues. The optimal limit-cycle was first introduced in figure~\ref{fig:bow tie curves} where two variants were presented. This section discusses how to determine which type of curve is to be expected, how to compute its corner coordinates and to interpret stability from it.

The pure bow-tie curve differs from the truncated bow-tie curve in that the priority queue has a slow-mode i.e. the optimal trajectory has the server idle at the priority queue. Without loss of generality, queue 1 will be assigned as the priority queue which requires that 
\begin{equation}
    c_1 \lambda_1 > c_2 \lambda_2. \label{eq: priority}
\end{equation}
In the event that $c_1 \lambda_1 = c_2 \lambda_2$ then system is symmetric, no priority queue exists and pure bow-tie curve is adopted. For the system with a priority queue, a truncated bow-tie curve only occurs if the \emph{slow-mode condition} is met such that
\begin{equation}
    c_1 \lambda_1 \rho - \left( c_1\lambda_1 - c_2 \lambda_2 \right)\left( 1-\rho_2 \right) < 0 \label{eq: slow-mode condition}
\end{equation}
This follows from theorem 5.3 of \cite{van_eekelen}. Clearly, if no priority queue exists then the right-hand term vanishes and the above condition is failed. A closer inspection of figure~\ref{fig:bow tie curves} reveals that the pure curve is a special case of the truncated curve. More specifically, the truncated curve has five corners of which one specifies until where the server idles. The pure curve has this point equal to the preceding coordinate where exhaustive service ends. Hence, this four point curve can be seen as a five-point curve with a degenerate coordinate. This perspective avoids the cumbersome procedure of formulating both curve separately. Start at an exhausted queue 1
\begin{equation}
    \left(0\,,\, \lambda_2\times\left(\bar{t}_{s_{2,1}}+\frac{\rho_1\left(\bar{t}_{s_{1,2}}+\bar{t}_{s_{2,1}}\right)\left(1+\alpha_1 \rho_2\right)}{1-\rho}\right) \right) \tag{C1} \label{coord: c1} 
\end{equation}
idle at queue 1 until
\begin{equation}
    \left(0\,,\, \lambda_2\times\left( \bar{t}_{s_{2,1}}+\frac{ \left(\bar{t}_{s_{1,2}}+\bar{t}_{s_{2,1}}\right)\left( \alpha_1\left(1-\rho_1\right)\left(1-\rho_2\right)+\rho_1 \right)}{1-\rho}\right) \right)  \tag{C2}\label{coord: c2}
\end{equation}
and then switch over to queue 2. The server arrives at queue 2 when the queues are at the following lengths
\begin{equation}
    \left(\lambda_1 \bar{t}_{s_{1,2}}\,,\, \lambda_2\left(\bar{t}_{s_{1,2}}+\bar{t}_{s_{2,1}}\right)\times\left(\frac{ \left(1+ \alpha_1\left(1-\rho_1\right) \right)\left(1-\rho_2\right)}{1-\rho}\right) \right)  .\tag{C3}\label{coord: c3}
\end{equation}
The server then exhausts queue 2 such that
\begin{equation}
    \left(
    \lambda_1 \times \left( \bar{t}_{s_{1,2}} +
    \frac{\rho_2\left(\bar{t}_{s_{1,2}}+\bar{t}_{s_{2,1}}\right)\left(1+ \alpha_1\left(1-\rho_1\right)\right)}
    {1-\rho}
    \right)
    \,,\, 0 \right)  \tag{C4} \label{coord: c4}
\end{equation}
upon which it immediately switches back to queue 1. The arrival to queue 1 is met with the following state of the queues
\begin{equation}
    \left(   
    \lambda_1\left(\bar{t}_{s_{1,2}}+\bar{t}_{s_{2,1}}\right)\times\left(\frac{\left(1+\alpha_1\rho_2\right)\left(1-\rho_2\right)}{1-\rho}\right) 
    \, ,\, \lambda_2 \bar{t}_{s_{2,1}} \right) \tag{C5} \label{coord: c5}
\end{equation}
where exhaustive service proceeds until coordinate (\ref{coord: c1}) is reached and the cycle restarts. The pure bow-tie curve emerges by setting $\alpha_1 = 0$ and ignoring (\ref{coord: c3}). If, condition~(\ref{eq: slow-mode condition}) is passed then $\alpha_1$ is solved through quadratic equations
\begin{eqnarray}
    \alpha_1 & = &  \alpha_1^{(1)}\mathbbm{1}_{\left\{\alpha_1^{1}>0\right\}} + \alpha_1^{(2)}\mathbbm{1}_{\left\{\alpha_1^{2}>0\right\}}\\
    \alpha_1^{(n)} & = & \operatorname{Re}\left\{\frac{-b + (-1)^n \sqrt{b^2 - 4ac}}{2a}\right\}
\end{eqnarray}
where only one of the $\alpha_1^{(n)}$ will be positive. The coefficients are determined as follows:
\begin{eqnarray}
    a & = &  c_1 \lambda_1 \rho_2^2(1-\rho_1) + c_2 \lambda_2 (1-\rho_1)^2(1-\rho_2) \\
    b & = & 2c_1 \lambda_1 \rho_2^2 + 2c_2 \lambda_2 (1-\rho_1)(1-\rho_2) \\
    c & = & c_1 \lambda_1 \rho - (c_1 \lambda_1-c_2 \lambda_2)(1-\rho_2).
\end{eqnarray}
Pure-bow tie curves will allow for stable Semi-Markov Decision Process policies. However, recall from section~\ref{section: tpm truncation} that the transition model needs to be truncated through selecting a maximum considered queue size $\mathcal{X}_i \in \mathbb{N}$ for each queue. Such a choice does not need to be a completely random guess; each $\mathcal{X}_i$ needs to be larger than than its queue's greatest entry from the optimal limit-cycle (\ref{coord: c1})-(\ref{coord: c5}). The value still needs to be sufficiently larger as to avoid the truncation error propagating from the truncated edges over to the useful part of the policy during policy iteration (or value iteration). Hence, the limit-cycle is useful for providing a lower bound to start performing truncation. It may even turn out that the lower bounds suggest a model that the available computational resources cannot handle in which case the polling model should be abandoned. Additional attention needs to be given to the truncated bow-tie curve. It is the experience of this research that a slow-mode of sufficient length will result in a SMDP policy $\pi$ that chooses to never switch away from the priority queue. As such $\mathbf{P}_{\pi}$ is no longer irreducible, subsequently not ergodic and never reaches a stationary distribution as the non-priority queue grows indefinitely. While a stable SMDP policy most likely exists for such a scenario, it could be that a very large $\mathcal{X}_i$ is required to deal with long slow-modes. This may be prohibitively large. Nonetheless, this remains an issue that requires further attention. For now it is best to treat a long-slow mode as a warning sign to be considered before investing in the construction and solving of the costly Semi-Markov Decision Process.

\subsubsection{Other polling model policies}\label{method: other policies}

Two simple policies that require little to no model-specific computation will be considered for performance analysis. The rationale behind this is to determine whether the additional costs of computing exactly optimal MDP policies result in a statistically significant performance enhancement. It is known from \cite{liu1992optimal} that for a symmetric polling model (both queues have the same parameters and event distributions) that exhaustive service is optimal. Furthermore, if the server empties the system at queue $i$ such that $x = (0,0,l_1=i)$ then it should idle and not switch. This \emph{exhaustive policy} outperformed other policies in the simulation experiments of \cite{tava_practical} when conducted using a symmetric polling system (as would be expected). This includes the use of event distributions other than the exponential distribution. Its performance degraded as the system became more asymmetric. In applying it to a symmetric system\footnote{The original paper formulates the optimal policy for $N$ symmetric queues.}, this policy can be viewed as \emph{exactly optimal}. It becomes a \emph{heuristic} policy when applied to asymmetric models.

In \cite{duenyas1996heuristic} a heuristic policy was developed for an asymmetric two-queue\footnote{Later section of the paper extends it to $N$ ranked queues.} polling system with a priority queue. This policy is based in formulating reward-indices for each queue as to maximise a notion of reward that is equivalent to minimising the long-run discounted holding cost \cite{harrison1975priority}. Assuming queue 1 to be the priority queue (\ref{eq: priority}), a case-based policy follows as in algorithm~\ref{algorithm: heuristic policy}. 

\begin{algorithm}[!htbp]
    \caption{Priority queue heuristic}
    \label{algorithm: heuristic policy}
    \begin{algorithmic}[1]
    \Procedure{Heuristic}{$\lambda_1$,$\lambda_2$,$\mu_1$,$\mu_2$,$\bar{t}_{s_{1,2}}$,$\bar{t}_{s_{2,1}}$,$c_1$,$c_2$}
    \State $\rho \leftarrow \lambda_1/\mu_1 + \lambda_2/\mu_2$
    \State assert $\rho < 1$ \Comment{Test for stability.}
    \State assert $c_1 \lambda_1 > c_2 \lambda_2$ \Comment{Insist on having a priority queue.}
    \State $m_2 = $  False \Comment{At least one queue 2 job has been served?}
    \State $l_2 = $ None \Comment{Server action to be returned and executed.}
    \If{$l_1=0$}
        \If{$n_1 >0$}
            \State $l_2\leftarrow 1$ 
        \Else{}
            \If{$n_2 > \lambda_2 \bar{t}_{s_{2,1}}$}
                \State $l_2 \leftarrow 2$
            \Else{}
                \State $l_2 \leftarrow 0$
            \EndIf
        \EndIf
    \Else{}
        \If{$n_2>0$}
            \State $\varphi \leftarrow \left(n_1 + \lambda_1 \bar{t}_{s_{1,2}} \right)/\left(n_1 + \mu_1 \bar{t}_{s_{1,2}} + (\mu_1-\lambda_1)  t_{s_{2,1}}\right) $
            \If{$\varphi \leq c_1\mu_1 \rho + c_2\mu_2(1-\rho)$}
                \State $l_2 \leftarrow 1$
                \State $m_2 \leftarrow$ True
            \Else{}
                \If{$m_2$ is False}
                    \State $l_2 \leftarrow 1$
                    \State $m_2 \leftarrow$ True
                \Else{}
                    \State $l_2 \leftarrow 2$
                    \State $m_2 \leftarrow$ False \Comment{Reset for next time the server visits.}
                \EndIf
            \EndIf
        \Else{}
            \If{$n_1 > \lambda_1 \bar{t}_{s_{1,2}}$ }
                \State $l_2 \leftarrow 2$
                \State $m_2 \leftarrow$ False
            \Else{} 
                \State $l_2 \leftarrow 0$
                \State $m_2 \leftarrow$ False
            \EndIf
        \EndIf
    \EndIf
    \State \Return $l_2$
    \EndProcedure
    \end{algorithmic}
\end{algorithm}

\subsubsection{Simulation}\label{method: simulation}

The most complex simulator used will be a Semi-Markov Process. This simulator will consult the given policy $\pi$ at the $k^{th}$ epoch of its embedded chain $\widetilde{x}_k$ as to select the state of the server $l_2^{k}=\pi(\widetilde{x}_k)$ for the duration of the transition to the next point in the embedded chain $\widetilde{x}_{k+1}$. During this temporally extended action arrivals may occur over the duration $\Delta t_k \sim F_{l_2^{k}}$ such that the simulation time can be updated as $t_{k+1} = t_{k}+\Delta t_k$ where $t_0 = 0$. These arrivals and the existing customers in $\widetilde{x}_k$ will contribute to generating a discounted holding cost $c_k^{\beta}$. Hence, a trajectory $\mathcal{T}(\widetilde{x}_0,\vec{\omega}_i,T,\pi) = \left\{ (\widetilde{x}_0,l_2^0,c_{0}^{\beta},\Delta t_0,t_0=0),\cdots,(\widetilde{x}_k,l_2^k,c_{k}^{\beta},\Delta t_k,t_k\geq T) \right\}$ of prescribed duration $T \in \mathbb{R}_{\geq 0}$ can be generated from an initial point in the embedded chain using a policy and list of random numbers $\vec{\omega}_i$.

The list of random numbers ensures that different policies can be run on the same environments as to compare their performance in a manner that is fair. In other words, this prevents one policy from receiving a bad batch of extreme responses from the environment (i.e. burst of arrivals) while the other policy is spared from such adversity. Perhaps the latter simulation even receives highly favourable outcomes. In order to implement such a list of random numbers, two approaches are proposed. 

The first approach assumes finite computer memory not to be an issue. This approach is simple, intuitive but wasteful. The list of random numbers is used to realise a large list of occurrences/samples for each event type $ \forall e \in \mathcal{E}: \, \vec{e} = \left[ e_{\omega} \sim F_{e}(\cdot | \omega): \omega \in \vec{\omega}_i \right]$. These lists can be stored as to be used for each simulation under a different policy. Instead of the simulation drawing an event sample $e_k$, it consults its correct copy of $\vec{e}$ and pops\footnote{Pop is an operation performed on lists or first-in-first-out queues. It simply removes the first entry of the list and returns it.} a sample $e_k = \vec{e}\mbox{.pop()}$.

The second approach takes finite memory into account and generates samples on demand. Each event type receives a copy of the same random number list. Each time an event is to be sampled a random number is obtained $\omega_k = \vec{\omega}$.pop(). The event is then generated as $e_k \sim F_{e}(\cdot | \omega_k)$. 

In both simulation approaches, the lists must be sufficiently long enough to last the duration of the simulation. A na\"{i}ve approach can use the following guideline as to obtain a ballpark estimate of the length of such a list. Pessimistically, assume that the simulation has a sampling rate $\nu = 1/\sum_{e \in \mathcal{E}} \bar{t}_e$ where $\bar{t}_e = \int_{0}^\infty t f_{e}(t)\, dt$. Then the length of the lists can best set as $L \geq \lceil T/\nu \rceil$. Alternatively, it can be acknowledged that $\vec{\omega}$ already exists on a system that generates \emph{pseudo-random numbers}. With respect to the length of simulations, this list $\vec{\Omega}$ can be regarded as infinitely long\footnote{The list of pseudo-random numbers is usually a circular list such that it never runs out. It is rare that it would be traversed.}. Hence, the second simulation approach can allocate a list of random numbers to each event by providing it with an index $\zeta_0 \in \mathbb{Z}$ instead. This index is often referred to as a \emph{random seed} and is used to access an element from $\Omega$ such that $\omega_k = \Omega_{\zeta_k} = $ after which it is updated as $\zeta_{k+1} = \zeta_{k} + 1$. Using this method of generating common random numbers along with the second simulation approach is most memory efficient as and least wasteful. Such a strategy was adopted in simulation experiments of this research.

A remaining issue is to compute $c_k^{\beta}$. This involves computing the discounted integral of a step-wise increasing function over the interval $t \in [0,\Delta t_k]$. The function is step-wise increasing due to arrivals occurring at $t_i^{\lambda} \in [0,\Delta t_k]$. The following helper function can be used to efficiently compute this integral
\begin{eqnarray}
    \mathcal{Z}(c,t_i^{\lambda},\Delta t_k) & = &  \int_{0}^{\Delta t_k} c e^{-\beta t} \mathbbm{1}_{\{t\geq t_i^{\lambda} \}} \, dt\nonumber\\
    & = & c \int_{t_i^{\lambda}}^{\Delta t_k} e^{-\beta t}\, dt \nonumber\\
    & = & \frac{c}{\beta} \left(  e^{-\beta t_i^{\lambda}} - e^{-\beta \Delta t_k} \right)
\end{eqnarray}
in conjunction with the set of arrivals over the transition of the embedded chain $\Lambda_k = \{ (j_i,t_i^{\lambda}):i = 1,2,3,c\cdots \}$ where $j_i \in \{1,2\}$ is the customer class of the $i^{th}$ arrival. Starting at the embedded chain $\widetilde{x}_k = (n_1^k,n_2^l,l_1^k)$ the desired cost is computed as below
\begin{equation}
    c_k^{\beta} = \mathcal{Z}(c_1 n_1+c_2 n_2,0,\Delta t_0) + \sum_{(j,t^{\lambda}) \in \Lambda_k} \mathcal{Z}(c_j,t^{\lambda},\Delta t_k). \label{eq: step-wise cost}
\end{equation}
The above computation will be referred to as \textsc{StepWiseCost} in algorithm~\ref{algorithm: system-based performance}. This cost is only locally discounted in the sense that discounting has been applied to the holding time of a single transition in the embedded chain. Further global discounting will be applied again when computing the overall cost of the simulated trajectory. Global discounting takes into account where in the trajectory the cost was incurred. This is illustrated below where the discounted holding cost of a simulated trajectory is obtained

\begin{equation}
    \mathcal{J}\left(\mathcal{T}(\widetilde{x}_0,\vec{\omega},T,\pi)\right) = \sum_{(c_{k}^{\beta},t_k) \in \mathcal{T}} e^{-\beta t_k} c_k^{\beta} .  \label{eq: compute rollout}
\end{equation}
In fact, equation~(\ref{eq: compute rollout}) is the empirical analog of the state-value functions (\ref{eq:discounted objective function}) used in Dynamic Programming. Moreover, it commonly referred to as a Monte-Carlo rollout in the Reinforcement Learning literature \cite{suttonRLbook}. What remains is to select an initial point in the embedded chain $\widetilde{x}_0$. This turns out to be quite important in determining what type of performance is being gauged and is discussed in the next section.


\subsubsection{Performance analysis}\label{method: performance analysis}

The discounted-cost Bellman equation (see section~\ref{section: discounted bellman equations} and \ref{section: smdp bellman equations}) make use of \emph{state-value functions} $\vec{J}_{\pi}^\beta$ both to \emph{evaluate} and \emph{improve} policies. The state-value functions are thus seen to play a crucial role in the theory of MDPs. Furthermore, for discounted MDPs the state-value functions are easy an intuitive to interpret ---- they are the objective function (\ref{eq:discounted objective function}) starting from a given state. Such simplicity does not carry over to the average-cost MDPs (see chapter 5 of \cite{BertsekasVol2}) which is complicate by gains, biases ans well as the existence of uni-chain and multi-chain policies. While state-value functions are useful for optimisation purposes, performance analysis might be interested in how well the polling system performs overall as opposed to how well it performs from a given state. It is thus necessary to distinguish between \emph{state-based optimality} and \emph{system-based optimality}.

\begin{definition}[\textbf{State-based optimal}]
A system is state-based optimal (in the minimising sense) under policy $\pi^* \in \Pi$ if
\begin{equation}
    \forall \pi \in \Pi\setminus\{\pi^*\}, \, \forall x \in \mathcal{X}: \quad J_{\pi^*}(x) \leq J_{\pi}(x) \label{eq: state-based optimal}
\end{equation}
where strict inequality holds at least once in the state-wise comparison for a unique optima. In vector notation this reads
\begin{equation}
    \forall \pi \in \Pi\setminus\{\pi^*\}: \quad \vec{J}_{\pi^*} \preceq \vec{J}_{\pi} \label{eq: state-based optimal vector}
\end{equation}
where $\preceq$ is an element-wise inequality.
\end{definition}
To formulate a definition for system-based optimality, a scalar is required to measure performance such as the state-value function does for state-based optimality. System-based performance should be understood as the long-run discounted cost incurred when a policy is applied regardless of the state it finds the system in. The system may have been operating under a different policy, the same policy or under some random rule. Either way, a probability distribution is induced over the state-space $\mathscr{P}(\mathcal{X})$. As the state-space is discrete, $\mathscr{P}(\mathcal{X})$ will be a probability mass function represented as a vector $\vec{p}_{\pi}$. System-based performance is obtained from this distribution and the state-value functions
\begin{equation}
    \eta_\pi = \left( \vec{p}_{\pi} \right)^T \vec{J}_{\pi}.
\end{equation}
\begin{definition}[\textbf{System-based optimal}]
System-based optimality is achieved under policy $\pi^* \in \Pi$ if 
\begin{equation}
    \forall \pi \in \Pi \setminus \{ \pi^* \}: \quad \eta_{\pi^*} \leq \eta_\pi
\end{equation}
where strict equality holds for a unique optima.
\end{definition}

It would be hoped for that a policy achieving state-based optimality would guarantee a system to achieve system-based optimality. It turns out that such a result does not always hold with the outcome relying on the choice of $\vec{p}_\pi$. The connection between state-based optimality and system-based optimality for the discounted and average-cost case has been investigated in \cite{solms_performance_system_based} under the uniform distribution $\vec{U}$ and stationary distribution of the given policy $\vec{\phi}_\pi$. When $\vec{\phi}_\pi $ is used, it was shown that when comparing two policies $\pi',\pi \in \Pi$ under the discounted-cost criteria
\begin{equation}
    \vec{J}_{\pi'}^\beta \preceq \vec{J}_\pi^\beta  \implies \eta_{\pi'}^\beta  \leq \eta_\pi^\beta 
\end{equation}
only if $\bar{J}_{\pi'} \leq \bar{J}_{\pi}$ along with both $\mathbf{P}_{\pi'}$ and $\mathbf{P}_{\pi}$ being uni-chain. Here $\bar{J}_{\pi}$ is a scalar called the \emph{gain}. It is obtained using the average-cost Bellman policy evaluation equations (see section 2.1 of \cite{solms_performance_system_based}). Such policy evaluation also produces a $|\mathcal{X}|$-length vector $\vec{h}_\pi$ called the \emph{bias} which reports relative\footnote{Relative is used here as once of the entries in the bias vector is set to zero.} state-based transient performance. Hence, a policy that is better in the state-based discounted sense is only discounted system-based better if its gain is better. As for optimality
\begin{equation}
    \forall \pi \in \Pi \setminus \{ \pi^* \}: \, \vec{J}_{\pi^*}^\beta \preceq \vec{J}_\pi^\beta  \implies \forall \pi \in \Pi \setminus \{ \pi^* \}: \, \eta_{\pi^*}^\beta  \leq \eta_\pi^\beta 
\end{equation}
if an only if $\pi^*$ is \emph{Blackwell optimal} (see section 3.5 of \cite{solms_performance_system_based} and chapter 5.1.2 of \cite{BertsekasVol2}). Furthermore, $\eta_{\pi^*}^\beta$ is a unique optima (strict inequality) if the set of gain-optimal polices consists of a single policy. It can be seen that using $\vec{\phi}_\pi$ requires several technicalities to be tended to. Furthermore, its usefulness is questionable as it assumes the existing policy to have already been in operation and that interest in performance only occurs once it has induced a stationary regime.

Using $\vec{U}$ is devoid of such issues. State-based optimality will always guarantee system-based optimality that holds with strict inequality. It may be obvious to see why this is the case. When $\vec{\phi}_\pi$ was used, system-based optimality depended on both the distribution and state-values functions which \emph{both differ} across policies. Now that system-based performance uses a \emph{common distribution}, the outcome depends only on the state-value functions

\begin{theorem}\label{theorem: system-based optimality}
If a common distribution over states $\vec{p} \in \mathbf{\Delta}^{|\mathcal{X}|-1} $ is used in obtaining system-based performance under discounted costs then
\begin{equation}
    \forall \pi \in \Pi \setminus \{ \pi^* \}: \, \vec{J}_{\pi^*}^\beta \preceq \vec{J}_\pi^\beta  \implies \forall \pi \in \Pi \setminus \{ \pi^* \}: \, \eta_{\pi^*}^\beta  < \eta_\pi^\beta 
\end{equation}
where $\pi^*$ is the optimal discounted cost policy.
\end{theorem}
\begin{proof}
Assume that when comparing $\pi^*$ with $\pi \in \Pi \setminus \{ \pi^* \}$ that
\begin{equation}
  \forall x \in \mathcal{X}: \quad  J_{\pi^*}(x) = J_{\pi}(x) - \epsilon(x)
\end{equation}
where $\epsilon(x) \geq 0$. As is well known that $\pi^* \in \Pi$ is a unique optimal policy for the infinite-horizon discounted-cost MDP (see chapter 1.2 of \cite{BertsekasVol2}) one can assume $\epsilon(x) > 0$ at least once. The difference between the system-based performance of the two policies is determined as
\begin{eqnarray}
    \eta_{\pi^*} - \eta_{\pi} & = & \sum_{x \in \mathcal{X}} p(x) \left(  J_{\pi^*}(x) - J_{\pi}(x) \right) \\
    & = &  \sum_{x \in \mathcal{X}} p(x) \left(  J_{\pi}(x) - \epsilon(x) - J_{\pi}(x) \right)\\
    & = &  -\sum_{x \in \mathcal{X}} \epsilon(x).
\end{eqnarray}
\end{proof}
If $\mathcal{X}(\pi^*,\pi) \subseteq \mathcal{X}$ is the set of states where strict state-based inequality holds then
\begin{eqnarray}
   \eta_{\pi} -  \eta_{\pi^*}  & = & \sum_{x \in \mathcal{X}(\pi^*,\pi)} \epsilon(x) \\
   & \geq & N\, \epsilon \\
    & > &  0
\end{eqnarray}
where $N = |\mathcal{X}(\pi^*,\pi)| \geq 1$ and $\epsilon = \min_{x \in \mathcal{X}}\{\epsilon(x)  \} > 0$ such that $\eta_{\pi^*} < \eta_\pi$.

Using $\vec{U}$ admits ignorance with regards to what state the system is assumed to be in when the optimal policy becomes applied. This makes sense when no prior knowledge exists of what polling model policy might be in use. If it is known that the polling model has adopted an exhaustive policy which is to be replaced by an MDP policy then $\vec{p}$ should be its stationary distribution. This paper adopts $\vec{U}$ as it is interested in the performance of policies on a system which it has no prior knowledge of.

\subsubsection{Empirical distributions}\label{method: empirical distributions}

The fact that $\eta_\pi$ is a scalar provides convenience over state-based optimality for the following application. Policies might theoretically be better than an alternative. However, this margin of superiority might not be statistically significant. Such significance can be determined through a hypothesis test as presented in section~\ref{section: hypothesis tests}. These tests rely on a comparing the \emph{uni-variate} empirical distributions $\hat{f}_\pi$ between to two different policies at to either reject or maintain the null hypothesis that their means or distributions are equal. State-based optimality would require either $|\mathcal{X}|$ such tests or a more complicated multi-variate significance test. Furthermore, each state would require a sufficient amount of sampled trajectories as to construct an empirical distribution $\hat{J}_\pi(x)$. System-based optimality is not faced with these issues as it permits a single uni-variate hypothesis test that requires only two empirical distributions $\hat{\eta}_{\pi}$ and $\hat{\eta}_{\pi'}$. Constructing $\hat{\eta}_{\pi}$ is simple as the necessary tools have already been described. This procedure is outlined in algorithm 
\begin{algorithm}[!htbp]
    \caption{Performance sampling}
    \label{algorithm: system-based performance}
    \begin{algorithmic}[1]
    \Procedure{Rollout}{$\theta$,$\pi$,$\vec{p}$,$\zeta_0$,$c_1$,$c_2$,$\beta$,$T$}
    \State $k \leftarrow 0$ \Comment{Event counter}
    \State $t_k\leftarrow 0$ \Comment{Global time}
    \State $\omega_{\lambda_1},\omega_{\lambda_2},\omega_{\mu},\omega_{\mu_2},\omega_{s_{1,2}},\omega_{s_{2,1}} \leftarrow \Omega_{\zeta_0}$ \Comment{Initialise random seed for all events}
    \State $\vec{\omega}_0 \leftarrow [\omega_{\lambda_1},\omega_{\lambda_2},\omega_{\mu},\omega_{\mu_2},\omega_{s_{1,2}},\omega_{s_{2,1}} ]$ \Comment{Store in a list/look-up table for updating}
    \State $x_0 \sim Cat(\vec{p},\Omega_{\zeta_0})$ \Comment{Categorical distribution}
    \State $\mathcal{J} \leftarrow 0$ \Comment{Simulated rollout (\ref{eq: compute rollout})}
    \While{$t_k < T$}
        \State $a+k \leftarrow \pi(x_k)$ \Comment{Consult the policy}
        \State $x_{k+1},\Delta_0,\Lambda_k,\vec{\omega}_{k+1} \sim \theta(x_k,a_k,\vec{\omega}_k)$ \Comment{Draw an event from the simulator}
        \State $c_k^\beta \leftarrow \mbox{\textsc{StepWiseCost}}(x_k,\Lambda_k,\Delta_k,t_k,c_1,c_2,\beta)$ \Comment{Use equation (\ref{eq: step-wise cost}) }
        \State $\mathcal{J} \leftarrow \mathcal{J} + e^{-\beta t_k}c_k^\beta $ 
        \State $k \leftarrow k + 1$
        \State $t_{k+1} \leftarrow t_k + \Delta t_k$
    \EndWhile
    \State \Return $\mathcal{J}$
    \EndProcedure
    \Procedure{SamplePerformance}{$\theta$,$\pi$,$\vec{p}$,$\zeta_0$,$c_1$,$c_2$,$\beta$,$T$,$N$}
    \State $\hat{\eta} \leftarrow \textsc{Array}(N)$ \Comment{Once-off memory allocation}
    \State $k\leftarrow 0$
    \While{$k < N$}
        \State $\hat{\eta}[k] = \mbox{\textsc{Rollout}}(\theta,\pi,\vec{p},\zeta_k,c_1,c_2,\beta,T)$
        \State $k \leftarrow k + 1$
        \State $\zeta_{k+1} \leftarrow \zeta_k + 1$ \Comment{Update random seed}
    \EndWhile
    \State \Return \textsc{RandomShuffle}$(\hat{\eta})$ \Comment{Prevent pair-wise correlation between different $\hat{\eta}_\pi$}
    \EndProcedure
    \end{algorithmic}
\end{algorithm}

The samples in the empirical distribution are randomly shuffled to prevent pair-wise correlation between two sampled distributions $\hat{\eta}_\pi$ and $\hat{\eta}_{\pi'}$. Such correlation exists because $\hat{\eta}_\pi[k]$ and $\hat{\eta}_{\pi'}[k]$ use the same random seed $\zeta_k$ and hence the same random numbers. The difference distribution $\Delta \hat{\eta} = \hat{\eta}_\pi \ominus \hat{\eta}_{\pi'}$ will be of importance when performing a single-population $t$-test when it attempts to reject the null hypothesis that its mean is zero. Significant correlation will severely underestimate the variance of this distribution such that $\Delta \hat{\eta} $ does not approximate the desired convolution that results from the difference between two random variables.

\subsection{Hypothesis tests}\label{section: hypothesis tests}

To interpret/investigate data an \emph{assumption} is usually made and \emph{tested}. Such an assumption is referred to as the \emph{null hypothesis} $H_0$ and can be rejected in favour of an \emph{alternative hypothesis} $H_A$ (another logical assumption). A statistical test can only reject the null hypothesis --- it cannot be accepted nor is the alternative hypothesis accepted upon its rejection. Hence the only outcomes are \emph{rejection} or \emph{failure to reject}. A statistical test derives a statistic which is used to report a $p$-value. This $p$-value is often acquired through querying a dedicated table with the relevant statistic. A $p$-value is defined as the probability of obtaining a datum at least as extreme as that found in the observed sample set \cite{nonparametric_tests_step_by_step,rice_stats}. It is used in conjunction with a significance level $\zeta \in (0,1)$ which is usually set to $\zeta=0.05$. If $p\leq \zeta$ then a result is said to be \emph{significant} such that the null hypothesis is rejected. Otherwise, it is not significant and the test fails to reject the null hypothesis.

\subsubsection{Student's t-test}\label{section: t test}
The \emph{one-sample} Student's $t$-test is a \emph{parametric} statistical test that assesses the null hypothesis that the \emph{sample mean} $\bar{X}$ is equal to some hypothesized \emph{population mean} $\mu$ \cite{rice_stats}. Mathematically, this can be expressed as $H_0: \bar{X} = \mu$. The $t$-statistic is computed as
\begin{equation}
    t = \frac{\left(\bar{X}-\mu\right)\sqrt{N}}{\hat{\sigma}}
\end{equation}
where $N = |X|$ is the sample size and $\hat{\sigma}$ is the sample standard deviation which uses $N-1$ as a denominator
\begin{equation}
    \hat{\sigma} = \left(\frac{\sum_{x\in X} (x-\bar{X})^2}{N-1}  \right)^{\frac{1}{2}}. \label{eq: sample std}
\end{equation}
As this is a parametric test, the $p$-value can be obtained from the Student's $t$-distribution $F_t$ directly where $F_t$ is a cumulative distribution function with $N-1$ degrees of freedom. If $|t|$ denotes the absolute value then a test against the \emph{two-sided} alternative hypothesis $H_A: \bar{X} \neq \mu$ has a $p$-value determined as $p_{\neq} = 2(1-F_t(|t|))$. The one-sided alternatives follows as $p_{<} = F_t(t)$ for $H_A: \bar{X} < \mu$ and $p_{>} = 1 - F_t(t)$ for $H_A: \bar{X} > \mu$. This test traditionally assumes that $X\sim \mathcal{N}$ is normally distributed \cite{rice_stats}. In practice, this assumption is often violated to some minor degree especially if the data-sets are considered large $N\geq100$ \cite{normality_assumption}.

\subsubsection{Welch's t-test}\label{section: welch}

The Welch's $t$-test is a parametric two-sample statistical test that assesses the null hypothesis of whether two samples have equal means $H_0: \bar{X} = \bar{Y}$ under the assumption that $\operatorname{var}(X)\neq \operatorname{var}(Y)$ \cite{welchs_t_test}. This assumption distinguishes it from the two-sample unpaired $t$-test. It requires a $t$-statistic
\begin{equation}
    t = \frac{\bar{X}-\bar{Y}}{\sqrt{s_X^2 + s_Y^2}}
\end{equation}
where $s^2$ is the \emph{sample standard-error} that can be obtained form the sample's standard deviation~(\ref{eq: sample std})
\begin{equation}
    s^2 = \frac{\sigma}{\sqrt{N}}.
\end{equation}
The $p$-values are obtained from a $t$-distribution $F_t$ with $\xi$ degrees of freedom. If $\xi_X = N_X - 1 = |X| - 1$ and $\xi_Y = N_Y - 1 = |Y| - 1$ then 
\begin{equation}
    \xi = \frac{\left( \frac{\sigma_X^2}{N_X} +\frac{\sigma_Y^2}{N_Y}  \right)^2}{\frac{\sigma_X^4}{N_X^2 \xi_X} +\frac{\sigma_Y^4}{N_Y^2 \xi_Y}}
\end{equation}
Using $F_t$ and $t$, one can obtain $p_{\neq}$, $p_{<}$ and $p_{>}$ as in the previous section. This test assumes that both $X$ and $Y$ are normally distributed and that in $X$ and $Y$ are unpaired. Despite the normality assumption, it has been reported to be robust against skewed distributions if the sample sizes are large \cite{robust_t}.

\subsubsection{Mann-Whitney U-test}\label{section:mann-whitney}
The Mann-Whitney $U$-test is a \emph{non-parametric}\footnote{This means that it sis not based in a parameterised probability distribution. Hence, tables play an important role in obtaining the $p$-values.} statistical test that assesses the null hypothesis that two samples $X\sim F_X$ and $Y\sim F^Y$ have the \emph{same distribution} where $F$ is a cumulative distribution function \cite{nonparametric_tests_step_by_step,rice_stats}. This can be mathematically expressed as $H_0: P(x\in X > y \in Y) = 1/2 =P(x\in X < y \in Y)$. This is the same as claiming the two distributions to be \emph{stochastically equal} $X \stackrel{s.t.}{=} Y$ such that $\forall z \in Z: F_X(z) = F_Y(z)$ where $Z$ is the support $X\cup Y \subseteq Z$. The $U$-statistic is obtained as
\begin{equation}
    \mathcal{U}(X,Y) = \sum_{x \in X}\sum_{y\in Y} S(x,y)
\end{equation}
where
\begin{equation}
    S(x,y) = \begin{cases}
    1, & x > y\\
    \frac{1}{2} , & x = y \\
    0, & x < y
    \end{cases}
\end{equation}
such that the final statistic is $U = \min\left\{\mathcal{U}(X,Y),\mathcal{U}(Y,X)\right\}$.

\subsubsection{Normality test}\label{section:normality test}
This paper uses the \emph{D'Agostino's} $k^2$ \emph{test} to assess the null hypothesis that a sample is normally distributed $H_0: X \sim \mathcal{N}$ \cite{norm_test}. The statistic of interest is $k^2 \sim \chi^2(\xi=2)$ such that it is a parametric test. This statistic is based on transforms of the \emph{skewness} $g_1$ and \emph{kurtosis} $g_2$ of the sample such that 
\begin{equation}
    k^2 = Z_1(g_1)^2 + Z_2(g_2)^2.
\end{equation}
If the $i^{th}$ \emph{central sample moment} is defined to be 
\begin{equation}
    m_i =\frac{1}{|X|} \sum_{x\in X}(x-\bar{X})^i
\end{equation}
then the skewness is determined as
\begin{equation}
    g_1 = \frac{m_3}{\left(m_2\right)^{\frac{3}{2}}}
\end{equation}
while kurtosis similarly follows
\begin{equation}
    g_2 = \frac{m_4}{\left(m_2\right)^2}.
\end{equation}
Sample skewness and kurtosis have distributions which are asymptotically normal such that central moments for these distributions $m_i(g_k)$ were derived in \cite{pearson1931_normality_test}. These expressions are all defined in terms of the sample size $N=|X|$. Although as asymptotically normal, transformations $Z_1$ and $Z_2$ are used to make the distributions as close to a standard normal as possible. Different transforms exist but the most widely used are
\begin{equation}
    Z_1(g_1) = \frac{1}{\ln{(W)}}\sinh^{-1}{\left(g_1\sqrt{\frac{W^2-1}{2 \,m_2(g_1)}}\right)}
\end{equation}
where
\begin{eqnarray}
     m_2(g_1) & = & \frac{6(N-2)}{(N+1)(N+2)}\\
     W & = & \sqrt{2\,G_2(g_1)+4} -1\\
     G_2(g_1) & = & \frac{36(N-7)(N^2+2N-5)}{(N-2)(N+5)(N+7)(N+9)}
\end{eqnarray}
such that $m_2(g_1)$ and $G_2(g_1)$ are the variance\footnote{This is because the mean is zero $m_1(g_1)=0$.} and kurtosis of $g_1$, respectively. The second transform follows as
\begin{eqnarray}
    Z_2(g_2) = \sqrt{\frac{9A}{2}}\left( 1 - \frac{2}{9A}  - \sqrt[3]{\frac{1 -\frac{2}{A}}{1 + \frac{g_2 - m_1(g_2)}{\sqrt{m_2(g_2)}}\times \sqrt{\frac{2}{A-4}}}   }  \right)
\end{eqnarray}
where 
\begin{eqnarray}
    m_1(g_2) & = & -\frac{6}{N+1} \\
    m_2(g_2) & = & \frac{24N(N-2)(N-3)}{(N+1)^2(N+3)(N+5)}\\
    A & = & 6 + \frac{8}{G_1(g_2)}\left(\frac{2}{G_1(g_2)} + \sqrt{1 + \frac{4}{(G_1(g_2))^2}}   \right)\\
    G_1(g_2)  & = &  \frac{6(N^2-5N+2)}{(N+7)(N+9)}\sqrt{ \frac{6(N+3)(N+5)}{N(N-2)(N-3)}  }
\end{eqnarray}
such that $G_1(g_2)$ is the skewness of $g_2$. If $F$ is the cumulative distribution function of $\chi^2(\xi=2)$ then $p=1 -F(k^2)$.

\subsection{Results}

\subsubsection{Polling model parameters}\label{results: model params}

Various polling model scenarios were investigated using the method outlined in section~\ref{method: polling model scenarios} as to produce stable systems. This paper presents the results of two such experiments as can be found in table~\ref{tab: polling scenarios}. The first scenario is a symmetric polling system \emph{except} for the fact that service durations in queue 1 have a much higher variance and longer-tails than those of queue 2. Furthermore, switching from queue 2 over to queue 1 is also the more variable/unreliable and long-tailed switching process. This should make queue 1 and queue 2 identical in expected performance but different in terms of reliability with queue 1 being the unreliable option. A CTMDP would not be able to account for this as discussed in section~\ref{section: SMDP vs CTMDP differences}. Such difference in the spread of the event-duration distributions would not affect the expected holding cost per stage in a SMDP. However, it is picked up in its transition model which leads to the possibility of a SMDP policy differing from the CTMDP counterpart. It should be noted that if two such policies differ then it illustrates that the SMDP has managed to \emph{implicitly} model some form of risk. The use of long-tails and higher variance suggest this to either be upside-risk, variance-penalised risk or a combination of both.

The second scenario allows for the heuristic policy from algorithm~\ref{algorithm: heuristic policy} to be included in the experiments. This policy is computationally cheap to produce. As such, it is worth assessing whether the more expensive MDP policies are significantly better. By (\ref{eq: slow-mode condition}), the slow mode occurs at queue 1 which in this case is the priority queue.

\begin{table}[!htbp]
    \centering
    \begin{tabular}{llllllllll}
    \toprule
        \textbf{Scenario} & $\lambda_1$ & $\lambda_2$ & $t_{\mu_1}$ & $t_{\mu_2}$ & $t_{s_{1,2}}$ & $t_{s_{2,1}}$ & $c_1$ &  $c_2$ & $\rho$ \\
        \midrule
        Asym. var. & 0.8 & 0.8 & Gam(1,0.4) & Gam(30,4/30) & Gam(30,4/30) & Gam(1,0.4) & 1 & 1 & 0.64\\
        Slow mode & 1.5 & 0.4 & Gam(30,0.1/30) & Gam(20,0.5/20) & Gam(30,2/30) & Gam(20,3/20) & 2 & 1  & 0.35 \\
        \bottomrule
    \end{tabular}
    \caption{Polling model scenarios with common parameters: $\beta = 0.05$, $X_1=40$, $X_2=40$, $N_1=35$, $N_2=35$, $K_{1,2} = 0$ and $K_{2,1} =0$. Gam($k$,$\theta$) represents a gamma distribution with shape $k$, scale $\theta$, mean $k\theta$ and variance $k \theta^2$.}
    \label{tab: polling scenarios}
\end{table}

\subsubsection{MDP policies}\label{results: policies}

In this section, the MDP policies are graphically assessed to see whether they differ. If no differences between the SMDP and CTMDP policies are found at this point in the experiment then a conclusion can be made that additional model complexity yields no gain. For the two scenarios considered this was not the case. The policies have been presented as arrays which prescribe an action for the given queue lengths. Such an array has been allocated to each server position. A bisque entry denotes the action to serve, grey pertains to idling and light blue initiates a switch-over.

Figure~\ref{fig: smdp asymmetric var} shows that the SMDP policy idles one customer less when the server is at queue 1 as compared to the symmetric CTMDP policy of figure~\ref{fig: ctmdp asymmetric var}. The SMDP policy has identified queue 1 to be unreliable. For the polling model with a slow-mode, it is seen that both policies have a cut-off at queue 2 where they switch over to the priority queue. This is very similar to the heuristic policy \cite{duenyas1996heuristic} of algorithm~\ref{algorithm: heuristic policy}. This cut-off is not a clearly defined index though and it differs between the two MDP policies. Furthermore, idling at queue 1 differs as well.

\begin{figure}[!tbp]
  \centering
  \begin{minipage}[b]{0.43\textwidth}
    \includegraphics[width=\textwidth]{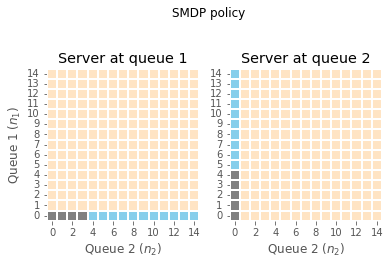}
    \caption{SMDP policy for polling model with asymmetric variance.}
    \label{fig: smdp asymmetric var}
  \end{minipage}
  \hfill
  \begin{minipage}[b]{0.43\textwidth}
    \includegraphics[width=\textwidth]{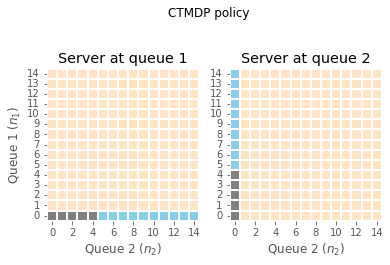}
    \caption{CTMDP policy for polling model with asymmetric variance.}
    \label{fig: ctmdp asymmetric var}
  \end{minipage}
\end{figure}

\begin{figure}[!tbp]
  \centering
  \begin{minipage}[b]{0.43\textwidth}
    \includegraphics[width=\textwidth]{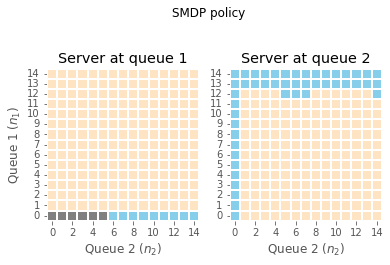}
    \caption{SMDP policy for polling model with slow mode.}
    \label{fig: smdp slow mode}
  \end{minipage}
  \hfill
  \begin{minipage}[b]{0.43\textwidth}
    \includegraphics[width=\textwidth]{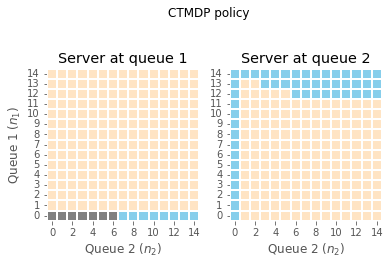}
    \caption{CTMDP policy for polling model with slow mode.}
    \label{fig: ctmdp slow mode}
  \end{minipage}
\end{figure}

It should be noted that both scenarios have used gamma distributions with a variance less than or equal to that of an exponential distribution. In fact, these gamma distributions have comparably small variances and short tails. This would appear to correlate with the fact that the SMDP policies (which seem to be able to account for variance and tailedness of distributions) idle less. To further investigate this possible relationship, a polling model was constructed that had event distributions with the same mean as an exponential distribution but higher variance and longer tails. The idea is to see whether the SMDP policy would idle more this time round. Exponential distributions have a \emph{coefficient of variation} equal to one $CV=\sigma /\bar{x}$. Most typical distributions have a $CV < 1$. However, setting the shape parameter $k$ to less than one in Gam($k$,$\theta$) results in $CV >1$. The SMDP policy for a slow-mode polling model that has distributions with $CV>1$ is presented in figure~\ref{fig:smdp_slow_mode_high_var}. It is observed to idle more than the CTMDP policy of figure~\ref{fig: ctmdp slow mode} which is on contrast to the SMDP policy obtained for low-variance short-tailed distributions from figure~\ref{fig: smdp slow mode}. This suggests that higher variance and long-tailedness in event durations lead to more idling.

A possible explanation for this risk-adverse behaviour follows. As the risk of transitioning to worse future states is increased (while maintaining the same expected cost of the outcome) through increasing the variance and tail of the distribution so does the certainty in idling make it a more attractive low-risk option. Idling guarantees only one customer arrival (of either class) and can only transition to one of two future states.

Lastly, note that the higher variance SMDP policy also adopted a more similar switching threshold at queue 2 to that of the CTMDP policy.

\begin{figure}[!htbp]
    \centering
    \includegraphics[width=0.5\textwidth]{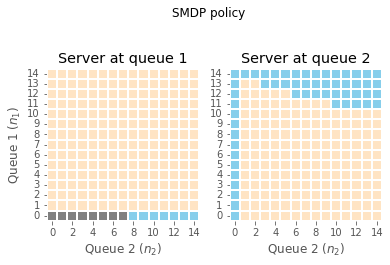}
    \caption{SMDP policy for a polling model with a slow mode and high variance distributions.}
    \label{fig:smdp_slow_mode_high_var}
\end{figure}

\subsubsection{Empirical distributions}\label{results: empirical dist}
These are the empirical distributions of the system-based performance that will be used by the parametric Welch's $t$-test and the non-parametric Mann-Whitney $U$-test. All distributions consists of $M=10000$ samples which have been generated using common random numbers. This translates to 10000 shared random seeds $\zeta_k$ used to initiate each sample. Furthermore, each sampled trajectory is of length $T=200$ such that any future sampled costs is almost negligible due to a discount of $\exp(-0.05\times200) = 4.54\times 10^{-5}$. 

The first scenario polling model is completely symmetric in terms of deterministic and expected values, the heuristic policy is not applicable and was not included in the experiment as seen in figure~\ref{fig: dist asymmetric var}. It was, however, included in the slow-mode scenario of figure~\ref{fig: dist slow mode} where it would appear to have outperformed the exhaustive policy.

\begin{figure}[!htbp]
  \centering
  \begin{minipage}[b]{0.49\textwidth}
    \includegraphics[width=\textwidth]{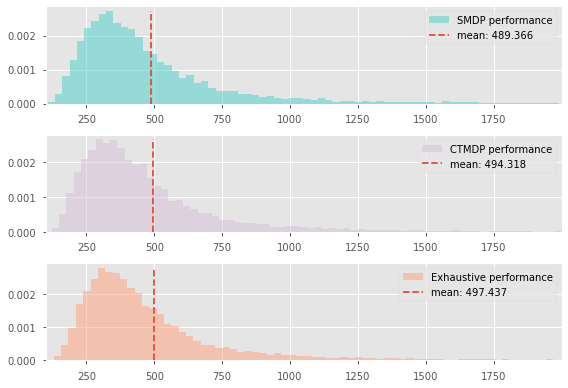}
    \caption{Empirical distributions $\hat{\eta}_\pi$ of system-based performance for the polling model with asymmetric variance.}
    \label{fig: dist asymmetric var}
  \end{minipage}
  \hfill
  \begin{minipage}[b]{0.49\textwidth}
    \includegraphics[width=\textwidth]{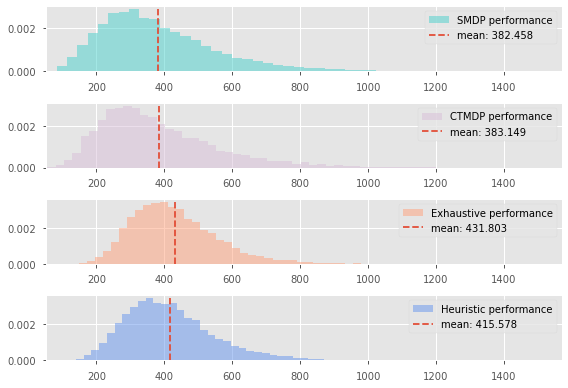}
    \caption{Empirical distributions $\hat{\eta}_\pi$ of system-based performance for the polling model with a slow-mode.}
    \label{fig: dist slow mode}
  \end{minipage}
\end{figure}

Summary statistics of these distributions have been provided in table~\ref{tab:summary stats asymmetric variance} and table~\ref{tab:summary stats slow-mode}. These tables also include the $k^2$ value for the D' Agostino's normality test of section~\ref{section:normality test} along with the $p$-values and whether the null hypothesis \emph{fails to be rejected}. All distributions had their null hypothesis rejected. The same conclusion could have been reached by graphical inspection of the  plotted distributions. This lack of normality puts the reliability of the Welch's $t$-test under question. However, as mentioned in section~\ref{section: welch}, robustness against skewed distributions with large sample size has been observed \cite{robust_t}. This paper considers $M=10000$ large. Nonetheless, it is good practice to include a Mann-Whitney $U$-test that does not rely on the normality assumption even though it does not strictly asses the difference in means.

\begin{table}[!htbp]
    \centering
    \begin{tabular}{llllllllll}
\toprule
\textbf{Policies} & \textbf{mean} &      \textbf{std} &      \textbf{min} &       \textbf{max} & \textbf{skewness} & \textbf{kurtosis} &         $k^2$ &    $p$ & \textbf{normal} \\
\midrule
SMDP &  489.366 &  345.952 &  106.848 &  6717.123 &    4.393 &   35.727 &   9349.427 &  0.0 &  False \\
CTMDP &  494.318 &  366.136 &  122.499 &  6142.736 &      4.4 &   32.169 &   9256.498 &  0.0 &  False \\
Exh. &  497.437 &  373.541 &  102.553 &  8986.153 &    5.959 &   74.359 &  11713.347 &  0.0 &  False \\
\bottomrule
\end{tabular}
    \caption{Summary statistics of $\hat{\eta}_\pi$ for the polling model with asymmetric variance.}
    \label{tab:summary stats asymmetric variance}
\end{table}

\begin{table}[!htbp]
    \centering
    \begin{tabular}{llllllllll}
\toprule
\textbf{Policies} & \textbf{mean} &      \textbf{std} &      \textbf{min} &       \textbf{max} & \textbf{skewness} & \textbf{kurtosis} &         $k^2$ &    $p$ & \textbf{normal} \\
\midrule
SMDP &  382.458 &  172.742 &  52.785 &  1569.504 &    1.168 &    2.166 &  1994.955 &  0.0 &  False \\
CTMDP &  383.149 &  173.884 &  54.594 &  1298.574 &    1.163 &      1.7 &  1863.162 &  0.0 &  False \\
Exh. &  431.803 &  130.958 &  124.35 &  1312.551 &    1.086 &     2.42 &  1911.637 &  0.0 &  False \\
Heur.  &  415.578 &  132.644 &  93.731 &  1234.293 &    0.906 &    1.518 &  1355.903 &  0.0 &  False \\
\bottomrule
\end{tabular}
    \caption{Summary statistics of $\hat{\eta}_\pi$ for the polling model with a slow-mode.}
    \label{tab:summary stats slow-mode}
\end{table}

All hypothesis tests assume that no pair-wise sample dependence between distributions exist. As to prevent any such dependence due to shared random numbers, all empirical distributions were randomly shuffled and the \emph{Pearson correlation coefficient} $r_{x,y}$ was tested for. This coefficient is calculated as
\begin{equation}
    r_{x,y} = \frac{\sum_{i=1}^M (x_i-\bar{x})(y_i-\bar{y})}{\left(\sum_{i=1}^M (x_i-\bar{x})^2(y_i-\bar{y})^2\right)^{\nicefrac{1}{2}}}
\end{equation}
where $x$ and $y$ are samples of distributions. Tables~\ref{tab:corr asymmetric variance} and ~\ref{tab:corr slow-mode} present $r_{x,y}$ for both polling models. The correlation is observed to be low such that all hypothesis tests can be used without hesitating in this regard.

\begin{table}[!htbp]
    \centering
    \begin{tabular}{lrrr}
\toprule
\diagbox{$\hat{\eta}_{\pi}$}{$\hat{\eta}_{\pi'}$} &  SMDP &  CTMDP &  Exh. \\
\midrule
SMDP &   1.000 &  -0.006 &   0.013 \\
CTMDP &  -0.006 &   1.000 &   0.004 \\
Exh. &   0.013 &   0.004 &   1.000 \\
\bottomrule
\end{tabular}
    \caption{Pearson correlation coefficient between $\hat{\eta}_\pi$ for the polling model with asymmetric variance.}
    \label{tab:corr asymmetric variance}
\end{table}

\begin{table}[!htbp]
    \centering
    \begin{tabular}{lrrrr}
\toprule
\diagbox{$\hat{\eta}_{\pi}$}{$\hat{\eta}_{\pi'}$} &  SMDP &  CTMDP &  Exh. & Heur.\\
\midrule
SMDP &   1.000 &  -0.006 &  -0.009 &  -0.013 \\
CTMDP &  -0.006 &   1.000 &  -0.014 &   0.007 \\
Exh. &  -0.009 &  -0.014 &   1.000 &  -0.008 \\
Heur. &  -0.013 &   0.007 &  -0.008 &   1.000 \\
\bottomrule
\end{tabular}
    \caption{Pearson correlation coefficient between $\hat{\eta}_\pi$ for the polling model with a slow mode.}
    \label{tab:corr slow-mode}
\end{table}

\subsubsection{Empirical difference distributions}\label{results: empirical diff dist}

It is useful to assess the difference distributions $\Delta \hat{\eta}_{\pi,\pi'} = \hat{\eta}_{\pi} \ominus \hat{\eta}_{\pi}$ as these might resemble normality better than their parent distributions. Occasionally, they may even fail to reject the null hypothesis in the normality tests. These distributions are used in a single-sample Student's $t$-test (see section~\ref{section: t test}) as to assess the null hypothesis that the mean is zero. Some of the difference distributions have been plotted in figure~\ref{fig: diff dist asymmetric variance} and figure~\ref{fig: diff dist slow-mode}. Normal distributions have been fitted to each data-set and as to visually aid in assessing normality. Normality clearly does not hold in the symmetric distributions of the polling model with asymmetric variance. Inspection does not yield such clear conclusions for the polling model with a slow-mode. As such, tables~\ref{tab:summary stats diff asymmetric var} and \ref{tab:summary stats diff slow-mode} provide both the summary statistics of the difference distributions and the relevant results from their normality tests. It is seen that despite the distributions of figure~\ref{fig: diff dist slow-mode} almost appearing normal, the D' Agostino's test rejects this hypothesis.

\begin{figure}[!htbp]
    \centering
    \includegraphics[width=0.7\textwidth]{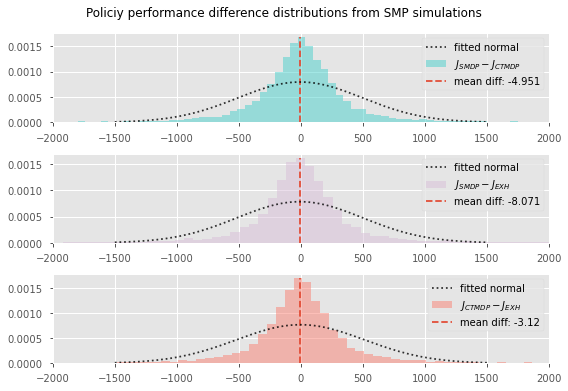}
    \caption{Difference distributions for polling model with asymmetric variance}
    \label{fig: diff dist asymmetric variance}
\end{figure}

\begin{figure}[!htbp]
    \centering
    \includegraphics[width=0.7\textwidth]{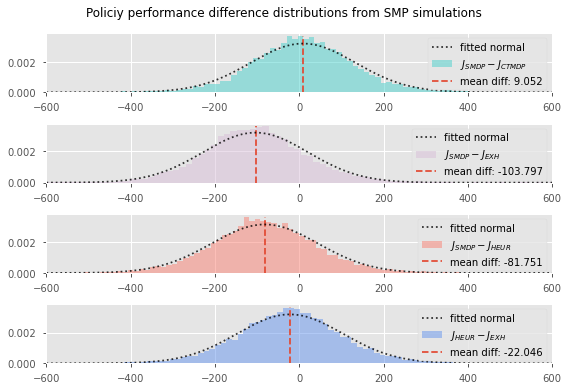}
    \caption{Difference distributions for polling model with a slow mode.}
    \label{fig: diff dist slow-mode}
\end{figure}

\begin{table}[!htbp]
    \centering
    \begin{tabular}{llllllllll}
\toprule
\textbf{Policies} & \textbf{mean} &      \textbf{std} &      \textbf{min} &       \textbf{max} & \textbf{skewness} & \textbf{kurtosis} &         $k^2$ &    $p$ & \textbf{normal} \\
\midrule
SMDP - CTMDP& -4.951 &  502.424 & -4550.159 &  6311.849 &    -0.16 &   16.257 &  2276.576 &  0.0 &  False \\
SMDP - Exh.& -8.071 &   506.73 &  -8344.24 &  6458.869 &    -0.85 &   28.045 &  3702.161 &  0.0 &  False \\
CTMDP - Exh.&  -3.12 &  522.135 & -8532.508 &  5864.779 &   -0.688 &   28.235 &  3438.243 &  0.0 &  False \\
\bottomrule
\end{tabular}
    \caption{Summary statistics of $\Delta \hat{\eta}_{\pi,\pi}$ for the polling model with asymmetric variance.}
    \label{tab:summary stats diff asymmetric var}
\end{table}

\begin{table}[!htbp]
    \centering
    \begin{tabular}{llllllllll}
\toprule
\textbf{Policies} & \textbf{mean} &      \textbf{std} &      \textbf{min} &       \textbf{max} & \textbf{skewness} & \textbf{kurtosis} &         $k^2$ &    $p$ & \textbf{normal} \\
\midrule
SMDP - CTMDP&  -0.691 &  244.785 & -1130.633 &  1298.552 &   -0.011 &    0.987 &  191.434 &  0.0 &  False \\
SMDP - Exh.& -49.345 &  216.069 & -1017.106 &  1235.259 &    0.339 &    1.195 &  429.352 &  0.0 &  False \\
SMDP - Heur.&  -33.12 &  220.264 & -1010.303 &  1323.806 &     0.37 &    1.081 &  431.655 &  0.0 &  False \\
Heur. - Exh.& -16.225 &  187.942 & -1053.237 &   904.973 &   -0.054 &    1.076 &  220.262 &  0.0 &  False \\
\bottomrule
\end{tabular}
    \caption{Summary statistics of $\Delta \hat{\eta}_{\pi,\pi}$ for the polling model with a slow-mode.}
    \label{tab:summary stats diff slow-mode}
\end{table}

\subsubsection{Hypothesis tests}\label{results: hypothesis tests}

The Welch's $t$-test assesses the null hypothesis
\begin{equation}
    H_0^W: \quad \mathbb{E}\left[ \hat{\eta}_\pi \right] = \mathbb{E}\left[ \hat{\eta}_{\pi'} \right] \label{eq: null welch}
\end{equation}
as to be rejected in favour of the alternative
\begin{equation}
    H_A^W: \quad \mathbb{E}\left[ \hat{\eta}_\pi \right] < \mathbb{E}\left[ \hat{\eta}_{\pi'} \right]. \label{eq: alt welch}
\end{equation}
These results of these tests can be found in tables~\ref{tab:welch asymmetric var} and \ref{tab:welch slow-mode} and were conducted using a significance level $\zeta=0.05$. Rejecting the null hypothesis is rare and has been highlighted for ease of reading.

\begin{table}[!htbp]
\footnotesize
    \centering
    \begin{tabular}{llllllllll}
\toprule
\multirow{2}{*}{\diagbox{$\hat{\eta}_{\pi}$}{$\hat{\eta}_{\pi'}$}} & \multicolumn{3}{c}{\textbf{SMDP}} & \multicolumn{3}{c}{\textbf{CTMDP}} & \multicolumn{3}{c}{\textbf{Exh.}}\\
\cmidrule(lr){2-4} 
\cmidrule(lr){5-7}
\cmidrule(lr){8-10}
{} & $t$  & $p$  & reject $H_0$ & $t$  & $p$  & reject $H_0$ & $t$ & $p$  & reject $H_0$ \\
\midrule
\textbf{SMDP}&        0.0 &        0.5 &              False &     -0.983 &      0.163 &              False &     -1.585 &      0.056 &              False \\
\textbf{CTMDP} &      0.983 &      0.837 &              False &        0.0 &        0.5 &              False &     -0.596 &      0.275 &              False \\
\textbf{Exh.} &      1.585 &      0.944 &              False &      0.596 &      0.725 &              False &        0.0 &        0.5 &              False \\
\bottomrule
\end{tabular}
    \caption{Welch's $t$-test for the polling model with asymmetric variance using a significance level $\zeta=0.05$}
    \label{tab:welch asymmetric var}
\end{table}

\begin{table}[!htbp]
\footnotesize
    \centering
    \begin{tabular}{lllllllllllll}
\toprule
\multirow{2}{*}{\diagbox{$\hat{\eta}_{\pi}$}{$\hat{\eta}_{\pi'}$}} & \multicolumn{3}{c}{\textbf{SMDP}} & \multicolumn{3}{c}{\textbf{CTMDP}} & \multicolumn{3}{c}{\textbf{Exh.}} & \multicolumn{3}{c}{\textbf{Heur.}}\\
\cmidrule(lr){2-4} 
\cmidrule(lr){5-7}
\cmidrule(lr){8-10}
\cmidrule(lr){11-13}
{} & $t$  & $p$  & reject $H_0$ & $t$  & $p$  & reject $H_0$ & $t$ & $p$  & reject $H_0$ & $t$ & $p$  & reject $H_0$ \\
\midrule

\textbf{SMDP} &        0.0 &        0.5 &              False &     -0.282 &      0.389 &              False &    -22.762 &        0.0 &               \textbf{True} &    -15.206 &        0.0 &               \textbf{True} \\
\textbf{CTMDP} &      0.282 &      0.611 &              False &        0.0 &        0.5 &              False &     -22.35 &        0.0 &               \textbf{True} &    -14.827 &        0.0 &               \textbf{True} \\
\textbf{Exh.} &     22.762 &        1.0 &              False &      22.35 &        1.0 &              False &        0.0 &        0.5 &              False &      8.704 &        1.0 &              False \\
\textbf{Heur.} &     15.206 &        1.0 &              False &     14.827 &        1.0 &              False &     -8.704 &        0.0 &               \textbf{True} &        0.0 &        0.5 &              False \\

\bottomrule
\end{tabular}
    \caption{Welch's $t$-test for the polling model with a slow mode using a significance level $\zeta=0.05$.}
    \label{tab:welch slow-mode}
\end{table}

The Mann-Whitney $U$-test assesses the null hypothesis 
\begin{equation}
    H_0^U: \quad \hat{\eta}_\pi \stackrel{s.t.}{=} \hat{\eta}_{\pi} \label{eq: mann null}
\end{equation}
as to be rejected against the alternative hypothesis
\begin{equation}
    H_A^U: \quad \hat{\eta}_\pi \stackrel{s.t.}{<} \hat{\eta}_{\pi}\label{eq: mann alt}.
\end{equation}
Stochastically smaller is the same as 
\begin{equation}
 \forall y \in \mathcal{Y}: \quad  \int_{0}^y \hat{\eta}_{\pi}(\xi)\,d\xi > \int_{0}^y \hat{\eta}_{\pi'}(\xi)\,d\xi
\end{equation}
where $\mathcal{Y} \subseteq \mathbb{R}_{\geq 0}$ is the support of the empirical probability density functions. The results of these non-parametric test are presented in tables~\ref{tab:mann asymmetric var} and \ref{tab:mann slow-mode}.

\begin{table}[!htbp]
\footnotesize
    \centering
    \begin{tabular}{llllllllll}
\toprule
\multirow{2}{*}{\diagbox{$\hat{\eta}_{\pi}$}{$\hat{\eta}_{\pi'}$}} & \multicolumn{3}{c}{\textbf{SMDP}} & \multicolumn{3}{c}{\textbf{CTMDP}} & \multicolumn{3}{c}{\textbf{Exh.}}\\
\cmidrule(lr){2-4} 
\cmidrule(lr){5-7}
\cmidrule(lr){8-10}
{} & $t$  & $p$  & reject $H_0$ & $t$  & $p$  & reject $H_0$ & $t$ & $p$  & reject $H_0$ \\
\midrule
\textbf{SMDP} &  50000000.0 &        0.5 &              False &  50092127.0 &      0.589 &              False &  49185078.0 &      0.023 &               True \\
\textbf{CTMDP} &  49907873.0 &      0.411 &              False &  50000000.0 &        0.5 &              False &  49085298.0 &      0.013 &               True \\
\textbf{Exh.} &  50814922.0 &      0.977 &              False &  50914702.0 &      0.987 &              False &  50000000.0 &        0.5 &              False \\
\bottomrule
\end{tabular}
    \caption{Mann-Whitney $U$-test for the polling model with asymmetric variance using a significance level $\zeta=0.05$}
    \label{tab:mann asymmetric var}
\end{table}

\begin{table}[!htbp]
    \tiny
    \centering
    \begin{tabular}{lllllllllllll}
\toprule
\multirow{2}{*}{\diagbox{$\hat{\eta}_{\pi}$}{$\hat{\eta}_{\pi'}$}} & \multicolumn{3}{c}{\textbf{SMDP}} & \multicolumn{3}{c}{\textbf{CTMDP}} & \multicolumn{3}{c}{\textbf{Exh.}} & \multicolumn{3}{c}{\textbf{Heur.}}\\
\cmidrule(lr){2-4} 
\cmidrule(lr){5-7}
\cmidrule(lr){8-10}
\cmidrule(lr){11-13}
{} & $t$  & $p$  & reject $H_0$ & $t$  & $p$  & reject $H_0$ & $t$ & $p$  & reject $H_0$ & $t$ & $p$  & reject $H_0$ \\
\midrule

\textbf{SMDP} &  50000000.0 &        0.5 &              False &  50129271.0 &      0.624 &              False &  37607261.0 &        0.0 &               \textbf{True} &  40861694.0 &        0.0 &               \textbf{True} \\
\textbf{CTMDP} &  49870729.0 &      0.376 &              False &  50000000.0 &        0.5 &              False &  37330408.0 &        0.0 &               \textbf{True} &  40592699.0 &        0.0 &               \textbf{True} \\
\textbf{Exh.} &  62392739.0 &        1.0 &              False &  62669592.0 &        1.0 &              False &  50000000.0 &        0.5 &              False &  53831237.0 &        1.0 &              False \\
\textbf{Heur.} &  59138306.0 &        1.0 &              False &  59407301.0 &        1.0 &              False &  46168763.0 &        0.0 &               \textbf{True} &  50000000.0 &        0.5 &              False \\
\bottomrule
\end{tabular}
    \caption{Mann-Whitney $U$-test for the polling model with a slow mode using a significance level $\zeta=0.05$.}
    \label{tab:mann slow-mode}
\end{table}

The Student's $t$-test is performed on the difference distributions and assesses the null hypothesis
\begin{eqnarray}
    H_0^t : \quad \mathbb{E}  \left[ \hat{\eta}_{\pi,\pi'}   \right] = 0 \label{eq: ttest null}
\end{eqnarray}
as to be rejected for the alternative hypothesis
\begin{eqnarray}
    H_A^t : \quad \mathbb{E}  \left[ \hat{\eta}_{\pi,\pi'}   \right] < 0 \label{eq: ttest alt}.
\end{eqnarray}
The results of these tests are presented in tables~\ref{tab:ttest asymmetric var} and \ref{tab:ttest slow-mode}.

\begin{table}[!htbp]
\footnotesize
    \centering
    \begin{tabular}{llllllllll}
\toprule
\multirow{2}{*}{\diagbox{$\hat{\eta}_{\pi}$}{$\hat{\eta}_{\pi'}$}} & \multicolumn{3}{c}{\textbf{SMDP}} & \multicolumn{3}{c}{\textbf{CTMDP}} & \multicolumn{3}{c}{\textbf{Exh.}}\\
\cmidrule(lr){2-4} 
\cmidrule(lr){5-7}
\cmidrule(lr){8-10}
{} & $t$  & $p$  & reject $H_0$ & $t$  & $p$  & reject $H_0$ & $t$ & $p$  & reject $H_0$ \\
\midrule
\textbf{SMDP} &        0.0 &        0.5 &              False &     -0.988 &      0.162 &              False &     -1.581 &      0.057 &              False \\
\textbf{CTMDP} &      0.986 &      0.838 &              False &       -0.0 &        0.5 &              False &     -0.592 &      0.277 &              False \\
\textbf{Exh.} &      1.589 &      0.944 &              False &      0.596 &      0.724 &              False &        0.0 &        0.5 &              False \\
\bottomrule
\end{tabular}
    \caption{Student's $t$-test for the polling model with asymmetric variance using a significance level $\zeta=0.05$}
    \label{tab:ttest asymmetric var}
\end{table}

\begin{table}[!htbp]
\footnotesize
    \centering
    \begin{tabular}{lllllllllllll}
\toprule
\multirow{2}{*}{\diagbox{$\hat{\eta}_{\pi}$}{$\hat{\eta}_{\pi'}$}} & \multicolumn{3}{c}{\textbf{SMDP}} & \multicolumn{3}{c}{\textbf{CTMDP}} & \multicolumn{3}{c}{\textbf{Exh.}} & \multicolumn{3}{c}{\textbf{Heur.}}\\
\cmidrule(lr){2-4} 
\cmidrule(lr){5-7}
\cmidrule(lr){8-10}
\cmidrule(lr){11-13}
{} & $t$  & $p$  & reject $H_0$ & $t$  & $p$  & reject $H_0$ & $t$ & $p$  & reject $H_0$ & $t$ & $p$  & reject $H_0$ \\
\midrule

\textbf{SMDP}&       -0.0 &        0.5 &              False &     -0.281 &      0.389 &              False &    -22.533 &        0.0 &               \textbf{True} &    -15.152 &        0.0 &               \textbf{True} \\
\textbf{CTMDP} &      0.282 &      0.611 &              False &        0.0 &        0.5 &              False &    -22.359 &        0.0 &               \textbf{True} &    -14.718 &        0.0 &               \textbf{True} \\
\textbf{Exh.} &     22.824 &        1.0 &              False &     22.392 &        1.0 &              False &        0.0 &        0.5 &              False &      8.713 &        1.0 &              False \\
\textbf{Heur.} &     15.088 &        1.0 &              False &     14.807 &        1.0 &              False &       -8.7 &        0.0 &               \textbf{True} &       -0.0 &        0.5 &              False \\

\bottomrule
\end{tabular}
    \caption{Student's $t$-test for the polling model with a slow mode using a significance level $\zeta=0.05$.}
    \label{tab:ttest slow-mode}
\end{table}
\break

All tests reach the same conclusions:
\begin{enumerate}
    \item The SMDP policy consistently outperforms the CTMDP policy but this margin is never deemed statistically significant.
    \item For the polling with asymmetric variance, no policy can be identified as the best in a statistically significant sense.
    \item For the polling model with a slow mode, both the SMDP and CTMDP policies outperform the exhaustive and heuristic policy by a margin that is statistically significant.
    \item For the polling model with a slow mode, the heuristic policy outperforms the exhaustive policy by a margin that is statistically significant.
\end{enumerate}
It can be said that the additional computational budget of the MDP policies does not provide a significant advantage over the very cheap exhaustive policy in the symmetric polling model with asymmetric variance.

\subsection{Other statistics}

\subsubsection{Embedded chain stationary distribution}

The SMDP models have made use of $\mathbf{P}_{\pi}$ which is the embedded chain (recall definition~\ref{def: embedded chain}) of the system under a policy. If follows that the stationary distribution of the embedded chain $\widetilde{\phi}$ can be obtained as in definition~\ref{def: stationary dist} assuming the system is ergodic (see definition~\ref{def: ergodic}). This may be a very large system of linear equations to solve exactly. Simulation can be used as an alternative using a very long trajectory $\mathcal{T}$ obtained from a SMP operating under $\pi$. Additionally, simulation does not from suffer issues related to state-space truncation. Algorithm~\ref{algorithm: embedded phi} presents such means of processing a simulated trajectory.

\begin{algorithm}[!htbp]\label{algorithm: embedded phi}
    \caption{Empirical embedded stationary distribution}
    \begin{algorithmic}[1]
    \Procedure{EmbeddedStationary}{$\mathcal{T}$,$B$}
    \State $\mathcal{T} \leftarrow \mathcal{T}[B:]$ \Comment{Discard first $B$ samples.}
    \State $N,N_1,N_2 \leftarrow 0,0,0$
    \State $\phi = \{\}$ \Comment{An empty look-up/hash table.}
    \For{$(\widetilde{x}_i,l_2^i,c_i^\beta,\Delta t_i,t_i) \in \mathcal{T}$}
        \If{$\widetilde{x}_i \not \in \phi\mbox{.keys}$}
            \State $\phi[\widetilde{x}_i] = 0$ \Comment{Add entry to the hash table.}
        \EndIf
        \State $N \leftarrow N + 1$
        \State $\alpha \leftarrow \nicefrac{1}{N}$ \Comment{Update weight.}
        \State $\phi[\widetilde{x}_i] \leftarrow \phi[\widetilde{x}_i] (1-\alpha) + \alpha $ \Comment{Mean update.}
        \State $x_1^i,x_2^i,l_1^i \leftarrow \widetilde{x}_i$ \Comment{Unpack state into components.}
        \State $N_1 \leftarrow N_1 + \mathbbm{1}\{l_1^i=1 \}$
        \State $N_2 \leftarrow N_2 + \mathbbm{1}\{l_1^i=2 \}$
    \EndFor
    \State \Return $\phi,N_1,N_2$
    \EndProcedure
    \end{algorithmic}
\end{algorithm}

Algorithm~\ref{algorithm: embedded phi} returns $\widetilde{\phi}$ as would solving the system of linear equations. Its output is, however, in the form of a hash table from which it is possible to construct a vector from as to mimic the linear equations solution. Each entry in $\widetilde{\phi}$ represents the joint-probability $\mathscr{P}(x_1,x_2,l_1) = \widetilde{\phi}(x_1,x_2,l_2)$. The simulation approach provides the number of times the server was present at each queue. Hence, $\widetilde{\phi}(l_i)= \mathscr{P}(l_1=i) = N_i/(N_1+N_2)$ such that each queue can have its own stationary distribution analysed 
\begin{equation}
    \widetilde{\phi}(x_1,x_2\mid l_1=i) = \frac{\widetilde{\phi}(x_1,x_2, l_1=i)}{\widetilde{\phi}(l_1=i)}.
\end{equation}
Of particular interest is the stationary distribution only over queue lengths
\begin{eqnarray}
    \widetilde{\phi}(x_1,x_2) & = & \widetilde{\phi}(x_1,x_2\mid l_1=1)\widetilde{\phi}(l_1=1) + \widetilde{\phi}(x_1,x_2\mid l_1=2)\widetilde{\phi}(l_1=2)\\
    & = & \widetilde{\phi}(x_1,x_2,l_1=1) + \widetilde{\phi}(x_1,x_2,l_1=2)
\end{eqnarray}
as this can be analysed along with the theoretically optimal limit-cycle of section~\ref{method: polling model scenarios} and \cite{van_eekelen}. Figure~\ref{fig:optimal limit cycle vs phi} illustrates such as approach for the symmetric polling model with asymmetric variance under the optimal SMDP policy.

\begin{figure}[!htbp]
    \centering
    \includegraphics[width=0.7\textwidth]{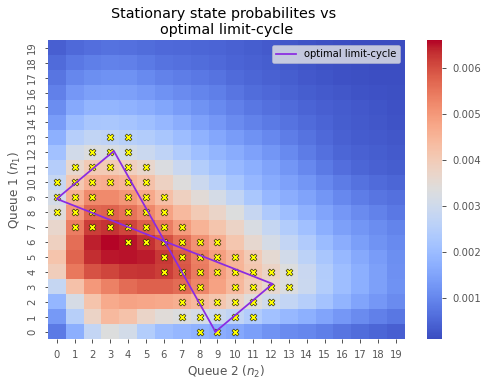}
    \caption{Optimal limit cycle as a pure bow-tie curve and $\widetilde{\phi}(x_1,x_2)$.}
    \label{fig:optimal limit cycle vs phi}
\end{figure}

As the optimal limit cycle does not prescribe an integer trajectory, this paper has assumed that a point in the limit cycle $(x_1,x_2)$ permits the following integer coordinates as part of the cycle
\begin{eqnarray}
    \mathbb{Z}(x_1,x_2) & = & \big\{ 
    \left(\lfloor x_1 \rfloor,\lfloor x_2 \rfloor \right) , \left(\lfloor x_1 \rfloor, \lceil x_2 \rceil \right) , 
    \left( \lceil x_1 \rceil, \lfloor x_2  \rfloor   \right) , \left( \lceil x_1 \rceil, \lceil x_2  \rceil   \right)
    \big\}
\end{eqnarray}
where $\lceil\cdot \rceil$ and $\lfloor\cdot \rfloor$ are the ceiling and floor operators, respectively. This set of unique integer coordinates $\mathbb{Z}(\mathcal{Y}_1,\mathcal{Y}_2) = \{ \mathbb{Z}(y_1,y_2): y_1 \in \mathcal{Y}_1,y_2 \in \mathcal{Y}_2 \}$ has been plotted as yellow crosses in figure~\ref{fig:optimal limit cycle vs phi}. It is assumed that $\mathcal{Y}_i$ is some finite grid as to represent the continuous value coordinates of the limit cycle. The probability of the system being in the optimal limit cycle is then obtained 
\begin{equation}
    \phi^* = \sum_{(x_1,x_2) \in \mathbb{Z}(\mathcal{Y}_1,\mathcal{Y}_2)} \widetilde{\phi}(x_1,x_2).
\end{equation}
The symmetric polling system with asymmetric variance operating under the optimal SMDP policy obtains $\phi^* = 0.334$. Hence, a policy spends majority of its time prescribing "corrective" actions as to steer the system towards an optimal limit cycle. This was noted in chapter 5 of \cite{van_eekelen} where a separate stabilising policy was used to prescribe such "corrective" actions. This was, however, a sub-optimal feedback policy. An optimal version can be found in \cite{van_zwieten_two_queue} where a quadratic program must be solved. A MDP policy is globally optimal and performs both corrective and optimal limit cycle actions.

Another interesting stationary probability, relevant to the next section, is that which pertains to how frequently actions $a$ are executed given that a server is at queue $i$

\begin{eqnarray}
    \widetilde{\phi}(l_2=a\mid l_1=i) & = & \sum_{x_1 \in \mathcal{X}_1} \sum_{x_2 \in \mathcal{X}_2} \widetilde{\phi}(x_1,x_2\mid l_1=i) \mathbbm{1}_{\{\pi(x_1,x_2,i)=a\}}. \label{eq: phi action per queue}
\end{eqnarray}
The probability of executing an action follows as
\begin{equation}
    \widetilde{\phi}(l_2=a) = \widetilde{\phi}(l_2=a\mid l_1=1)\widetilde{\phi}(l_1=1) + \widetilde{\phi}(l_2=a\mid l_1=2)\widetilde{\phi}(l_1=2)
\end{equation}
while a joint-probability can also be obtained
\begin{equation}
    \widetilde{\phi}(l_1=i,l_2=a) = \widetilde{\phi}(l_2=a\mid l_1=i) \widetilde{\phi}(l_1=i)
\end{equation}

\subsubsection{Overall stationary distribution} 

A necessary criteria for stability was presented in definition~\ref{def: statbility criteria}. A MDP policy can be assessed to see whether it meets this criteria after it has been solved for. This criteria essentially states that a server have the opportunity to be free or process no work. During this time it can idle or, more importantly, perform switch-overs. Hence, $\rho \leq \phi(l_2=1) < 1$ where $\phi(l_2=1)$ is the overall stationary probability of performing service. It is not a stationary probability of the embedded chain $\phi(l_2=1) \neq \widetilde{\phi}(l_2=1)$. As discussed on page~\pageref{eq: overall stationary semi-markov}, it represents the probability of finding the system performing a service at any given time of inspection. It follows that it represent the portion of time spent performing service and thus work.

The overall stationary probability can be computed using equation~(\ref{eq: overall stationary semi-markov}). Given the total event set of actions for each queue $\mathcal{E}_1 = \{ \mu_1,s_{1,2},\Lambda_1\}$ and $\mathcal{E}_2 = \{ \mu_2,s_{2,1},\Lambda_2 \}$ (see section~\ref{section: notation}) the portion of time the server spends working follows:
\begin{equation}
    \phi(l_2=1) = \frac{\widetilde{\phi}(l_2=1\mid l_1=1) \mathbb{E}\left[ t_{\mu_1}  \right]+\widetilde{\phi}(l_2=1\mid l_1=2) \mathbb{E}\left[ t_{\mu_2}  \right] }{\widetilde{\phi}(l_2=1\mid l_1=1)\left(\sum_{e\in\mathcal{E}_1} \mathbb{E}\left[ t_{e}  \right]\right) + \widetilde{\phi}(l_2=1\mid l_1=2)\left(\sum_{e\in\mathcal{E}_2} \mathbb{E}\left[ t_{e}  \right]\right)}
\end{equation}. \label{eq: prob of work}
For the symmetric polling model with asymmetric variance $\rho=0.64$ and $\phi(l_2=1) \approx 0.7 > \rho$ such that it was stable. The polling model with a slow mode had $\rho = 0.35$ and $\phi(l_2=1) \approx 0.51 > \rho$ such that it was stable as well. In some instances it might be more convenient to obtain $\phi(l_2=1)$ directly from simulation using a procedure as outlined in algorithm~\ref{algorithm: overall phi}. The full output from this algorithm can be used to obtain various other stationary probabilities as well. The probability of interest from this section follows by consulting the hash-table $\phi(l_2=1) = \phi[1]$.

\begin{algorithm}[!htbp]\label{algorithm: overall phi}
    \caption{Empirical stationary probabilities of actions}
    \begin{algorithmic}[1]
    \Procedure{ActionProbabilities}{$\mathcal{T}$,$B$}
    \State $\mathcal{T} \leftarrow \mathcal{T}[B:]$ \Comment{Discard first $B$ samples.}
    \State $T_1 \leftarrow \{0:0,1:0,2:0 \}$ \Comment{Total time of queue 1 actions.}
    \State $T_2 \leftarrow \{0:0,1:0,2:0 \}$ \Comment{Total time of queue 2 actions.}
    \For{$(\widetilde{x}_i,l_2^i,c_i^\beta,\Delta t_i,t_i) \in \mathcal{T}$}
    \State $x_1^i,x_2^i,l_1^i = \widetilde{x}_i$ \Comment{Unpack states-space.}
        \If{$l_1^i=1$}
            \State $T_1[l_2^i] \leftarrow T_1[l_2^i] + \Delta t_i$
        \EndIf
        \If{$l_1^i=2$}
            \State $T_2[l_2^i] \leftarrow T_2[l_2^i] + \Delta t_i$
        \EndIf
    \EndFor
    \State $T \leftarrow \sum_{i=0}^2 T_1[i]+ T_2[i]$ \Comment{Total simulation time.}
    \State $\phi \leftarrow \{ 0: (T_1[0]+T_2[0])/T,1:(T_1[1]+T_2[1])/T,2:(T_1[2]+T_2[2])/T \}$
    \State \Return $\phi,T,T_1,T_2$
    \EndProcedure
    \end{algorithmic}
\end{algorithm}

\newpage

\section{Conclusion}

The most complex model used for controlling a two queue polling system with switch-over durations has been a CTMDP. This paper has shown that under certain assumptions, as found in section~\ref{section: SMDP assumptions}, the polling system of interest can be modelled as a more sophisticated SMDP by using the competing process framework first introduced by Howard \cite{howard1960dynamic} (see \cite{howard2012dynamic} for the modern print). By allowing the service and switch-over durations to follow any arbitrary distribution over $\mathbb{R}_{\geq 0}$, the model of a controlled polling system is another step closer to being a DEDS (see table~\ref{tab:Different Models of DEDS}).

The SMDP approach has provided the following improvements over the CTMDP. Firstly, section~\ref{section: homogeneous arrival rates} has shown that non-homogeneous arrival rates are permitted. Secondly, in the case of homogeneous arrival rates, the SMDP model can capture risk in the probabilistic transitions by accounting for variance, tailedness and overall spread over the general even-duration distributions (see section~\ref{section: SMDP vs CTMDP differences}). This was seen to affect the trade-off between choosing whether to idle or switch as in section~\ref{results: policies}. An intuitive explanation was provided which claims idling (a low-risk decision) to become more attractive as the variance and tailedness of the switch-over duration increased (making it more risk-prone).

The additional modelling power of the SMDP model comes at a cost. Truncation errors (see section~\ref{section: tpm truncation}) need to be monitored and expensive numerical integrals involving the matrix exponential of the generator matrix needs to be performed (see section~\ref{section: smdp transition model} and \ref{section: smdp cost model}) as to build the model. Furthermore, the generator matrix scales poorly with increasing size of the state-space and is a very sparse object. Sparsity may be exploited in future work especially when obtaining the matrix exponential of it. The CTMDP is devoid of this issues. Section~\ref{section: CTMDP} showed that uniformisation and a sparsity exploiting graph structure allows for CTMDPs to be built and solved much more efficiently than the SMDP. A CTMDP can be used to model both a preemptive and non-preemptive polling system. The former would be difficult to implement in the SMDP setting. Fictitious checkpoints might be introduced during service and switch-over events as to allow for preemption. 

The empirical simulations studies of section~\ref{section: experiments} showed that while the SMDP policy produced better mean system-based performance on a SMP simulation than the CTMDP policy, this margin of improvement was not found to be statistically significant. This result questions whether the increased computational costs and modelling issues are worth the non-significant advantage. Both the CTMDP and SMDP were able to outperform the heuristic and exhaustive policies. Putting this all together, suggests that the CTMDP remains the best practical model to use. Further work should focus on extending the CTMDP model as to approximate general distributions (and a GSMDP) using the method of phases as in \cite{younes_GSMP_formalism,younes_GSMPD_solving1,younes_GSMPD_solving2,younes_thesis}. Such a model would be able to account for asynchronous and concurrent dynamics while being more computationally tractable than the less expressive SMDP model of this paper.

\newpage
\medskip
\printbibliography

\end{document}